%% file: arxiv.tex
 \documentclass[mnsc,nonblindrev]{informs3}
\usepackage{theorem}
\RequirePackage{tgtermes}
\RequirePackage{newtxtext}
\RequirePackage{newtxmath}
\RequirePackage{bm}
\RequirePackage{endnotes}
\input{preamble-arxiv}

\OneAndAHalfSpacedXII 

\usepackage{tikz}
\usetikzlibrary{arrows.meta,positioning,calc,backgrounds}

\usepackage{natbib}
 \bibpunct[, ]{(}{)}{,}{a}{}{,}%
 \def\bibfont{\small}%

\EquationsNumberedThrough    

\TheoremsNumberedThrough     
\ECRepeatTheorems  %

\MANUSCRIPTNO{MOOR-0001-2024.00}

\begin{document}


\RUNAUTHOR{}

\RUNTITLE{Wasserstein‑$p$ Central Limit Theorem Rates: From Local Dependence to Markov Chains}

\TITLE{Wasserstein‑$p$ Central Limit Theorem Rates: From Local Dependence to Markov Chains}

\ARTICLEAUTHORS{%
\AUTHOR{Yixuan Zhang}
\AFF{Department of Industrial \& Systems Engineering, University of Wisconsin-Madison, \EMAIL{yzhang2554@wisc.edu}}

\AUTHOR{Qiaomin Xie}
\AFF{Department of Industrial \& Systems Engineering, University of Wisconsin-Madison, \EMAIL{qiaomin.xie@wisc.edu}}
} 

\ABSTRACT{%
Non-asymptotic central limit theorem (CLT) rates play a central role in modern machine learning and operations research. In this paper, we study CLT rates for
multivariate dependent data in Wasserstein-$p$ ($\mathcal W_p$) distance, for
general $p\ge 1$. We focus on two fundamental dependence structures that commonly arise in practice: locally
dependent sequences and geometrically ergodic Markov chains. In both settings,
we establish the \textit{first optimal} $\mathcal O(n^{-1/2})$ rate in $\mathcal W_1$, as well as
the first $\mathcal W_p$ ($p\ge 2$) CLT rates
under mild moment assumptions, substantially improving the best previously known
bounds in these dependent-data regimes. As an application of our optimal $\mathcal W_1$
rate for locally dependent sequences, we further obtain the first optimal
$\mathcal W_1$-CLT rate for multivariate $U$-statistics.

On the technical side, we derive a tractable auxiliary bound for $\mathcal W_1$ Gaussian
approximation errors that is well suited for studying dependent data. For Markov chains,
we further prove that the regeneration time of the split chain associated with a
geometrically ergodic chain has a geometric tail without assuming strong
aperiodicity or other restrictive conditions. These
tools may be of independent interests and enable our optimal $\mathcal W_1$ rates and underpin our $\mathcal W_p$ ($p\ge 2$) results.
}%




\KEYWORDS{Central limit theorem, Wasserstein distance, Local dependence, Markov chains, Nummelin’s splitting} 

\maketitle


\section{Introduction}

Non-asymptotic Central Limit Theorem (CLT) rates  are important in modern Machine Learning (ML)  and Operations Research (OR), as learning and decision-making are often based on limited data and therefore require quantitative uncertainty guarantees. In particular, CLT rates are essential for obtaining non-asymptotic
error bounds and for enabling statistically valid inference on parameters of
interest in stochastic algorithms and dynamical systems
\citep{hastings1970monte,chen2018gaussian,srikant2024rates,samsonov2024gaussian,wu2025uncertainty,chernozhukov2013gaussian}.

Over the past decade, substantial progress has been made in quantifying the CLT rates through the use of transportation metrics \citep{rio2009upper,gallouet2018regularity,raivc2018multivariate, bonis2020stein, srikant2024rates}. In particular, the Wasserstein-$p$ ($\W_p$) distance \citep{villani2008optimal} between two probability measures $\nu,\mu$ on $\mathbb{R}^d$  is defined as
\begin{equation}\label{eq:wp}
  \mathcal{W}_p(\nu,\mu)
  := \inf_{\gamma\in\Gamma(\nu,\mu)}
  \left( \mathbb{E}_{(X,Y)\sim \gamma}\bigl[\|X-Y\|^p\bigr] \right)^{1/p},\quad \forall p \geq 1,
\end{equation}
where $\Gamma(\nu,\mu)$ denotes the set of couplings of $(\nu,\mu)$. Given a sequence of $\mathbb{R}^d$-valued observations $\{X_i\}_{i=0}^{n-1}$, define the partial sum
$S_n:=\sum_{i=0}^{n-1}X_i$ and the  normalized covariance $\Sigma_n:= n^{-1}\operatorname{Var}(S_n)$. Throughout this section, we assume for convenience that $\E[X_i]=0$ for all $i\in\{0,1,\dots,n-1\}$ and that $\Sigma_n = I_d$.
 CLT rates bound the Gaussian approximation error in $\W_p$:
\begin{equation}\label{eq:cltrate}
  \W_p\big(\mathcal L(\tfrac{S_n}{\sqrt n}), \mathcal N(0,I_d)\big),\quad \text{where } \law(\cdot)\text{ denotes the law of a random variable.}
\end{equation}

When $\{X_i\}_{i=0}^{n-1}$ are independent, a large literature establishes the
 CLT rates~\eqref{eq:cltrate}. In the univariate case ($d=1$), the optimal $\mathcal{O}(n^{-1/2})$ rate in $\W_p$ ($p\ge 1$) has been established under the moment condition $\sup_{i\ge 0}\E|X_i|^{2+p}<\infty$~\citep{petrov2012sums,rio2009upper,bobkov2013entropic}.
In the multivariate case ($d\ge 2$), obtaining the optimal rate is more delicate. In particular,
the regularity-based approach~\citep{gallouet2018regularity} typically leads to an
additional logarithmic factor, producing suboptimal rates of order $\mathcal{O}(n^{-1/2}\log n)$.
Moreover, for $p> 1$, the Wasserstein distance $\W_p$ generally lacks a tractable
dual representation for dimensions $d\ge 2$ \citep{bobkov2024fourier}, which limits the direct use of
Lipschitz-test-function arguments that are available in the univariate setting.

Recent work has made progress on the multivariate case. For $\W_1$, a rate of $\mathcal{O}(n^{-\delta/2})$ for $\delta\in(0,1)$ was established under the moment condition
$\sup_{i\ge 0}\E\|X_i\|^{2+\delta}<\infty$~\citep{gallouet2018regularity}, while the optimal $\mathcal{O}(n^{-1/2})$ rate for $\delta=1$
was obtained via a new Gaussian approximation framework~\citep{raivc2018multivariate}.
For $\W_p$ with $p\ge 2$, the rate
$\mathcal{O}\!\left(n^{-(p+q-2)/(2p)}\right)$ under
$\sup_{i\ge 0}\E\|X_i\|^{p+q}<\infty$ for some $q\in(0,2]$ was proved~\citep{bonis2020stein}.
 In particular, taking $q=2$
gives the optimal $\mathcal{O}(n^{-1/2})$ rate. These results for independent data are summarized in the blue-shaded region of Table~\ref{tab:summary}.

However, in many modern ML and OR applications, the observations are typically dependent rather than independent. In ML, quantities of interest often arise as partial
sums along correlated data streams or recursive algorithmic trajectories. A
particularly important example is stochastic gradient descent (SGD), where the
update takes the form
\[
\theta_{t+1}=\theta_t-\alpha \nabla f(\theta_t,w_t).
\]
Here \(\{w_t\}_{t \geq 0}\) denotes the sampled data or stochastic noise. Even when
\(\{w_t\}_{t \geq 0}\) is independent and identically distributed, the iterate
sequence \(\{\theta_t\}_{t \geq 0}\) itself forms a time-homogeneous
Markov chain
\citep{dieuleveut2020bridging,yu2021analysis,zhang2024prelimit,zhang2025piecewise},
as each update depends only on the current state and fresh randomness.
Therefore, when studying the fluctuations of learning algorithms, the relevant
dependence structure typically lies in the trajectory of the iterates rather than the input data. This perspective extends to other
recursive methods, such as Q-learning and TD-learning in reinforcement learning
\citep{zhang2024constant,huo2023bias}. In OR, analogous dependent-data structures arise naturally in steady-state simulation \citep{law1984confidence} and in the performance analysis of queueing and network models \citep{asmussen2003applied,Meyn2007}. Let $\{x_t\}_{t\ge 0}$ be a Markov chain on a state space $\mathcal X$ with stationary distribution $\pi$, and let $c:X\to\mathbb{R}^d$ be a $\pi$-integrable vector of one-step performance measures. Define
\[
\bar c_n := \frac{1}{n}\sum_{t=0}^{n-1} c(x_t),
\qquad
\theta := \int_{\mathcal X} c(x)\,\pi(dx).
\]
Then $\bar c_n$ is the standard run-length-$n$ sample-average estimator of the steady-state mean-performance vector $\theta$. Consequently, quantitative CLT rates for dependent data
provide a useful framework for understanding finite-sample behaviors of these methods. 

Despite the maturity of the theory for
independent data, corresponding \(\W_p\)-CLT rates under dependence remain largely underexplored. In this paper, we focus on two fundamental dependence
structures that frequently arise in modern ML and OR: (1) locally
dependent data and (2) Markov chains. We review the related literature for
these two settings separately below.

\noindent\textbf{Local Dependence.} Locally dependent data form an important subclass of dependent sequences, in which each
observation depends only on a bounded neighborhood in an underlying dependency graph
\citep{rinott1994normal}; see Section~\ref{sec:local-dependent} for a precise definition.
Such dependence arises naturally in graph-structured learning through network interactions,
for example in node-level prediction with graph neural networks
\citep{kipf2016semi,hamilton2017inductive}. For locally dependent data in the univariate setting, the optimal $\mathcal{O}(n^{-1/2})$
rate in $\W_p$ for all $p\ge 1$ was established~\citep{liu2023wasserstein}.
 In the multivariate setting, under an almost-sure boundedness assumption, recent work~\citep{fang2023p} obtained a rate of
$\mathcal{O}(n^{-1/2}\log n)$ for $p\ge 1$.

To the best of our knowledge, however, CLT rates in $\W_p$ for multivariate locally dependent data under mild moment assumptions remain underexplored for $p\ge 1$.

\noindent\textbf{Markov Chains.} Markov chains provide a unifying framework for modeling stochastic dynamics in
modern ML and OR. In ML, they
underlie sequential decision-making models such as Markov decision processes and
reinforcement learning
\citep{puterman2014markov,sutton1998reinforcement,zhang2024constant}, and they
also appear in generative modeling through diffusion-based Markovian dynamics
\citep{ho2020denoising}. In OR, Markov chains have long been used to model queueing systems
\citep{asmussen2003applied} and inventory systems with Markov-modulated demand
\citep{song1993inventory}. More recently, they have also played an important
role in healthcare operations, where they are used to study hospital patient
flow, discharge decisions, and interhospital transfers under congestion
\citep{shi2021timing,chan2026optimizing}. Markovian models likewise arise in
platform operations, such as ride-hailing and other two-sided marketplaces,
where demand, supply, matching, and pricing evolve sequentially over time
\citep{varma2023dynamic,ben2025data}. Across these settings, Markovian state
dynamics provide a tractable framework for analyzing stability, congestion, and
long-run performance.

Let $\{x_i\}_{i\ge 0}$ be a Markov chain on a general state space $(\mathcal X,\mathcal B)$
with unique stationary distribution $\pi$. Let $\{h_i\}_{i\ge 0}$ be measurable functions
$h_i:\mathcal X\to\mathbb R^d$ satisfying $\E_\pi[h_i]=0$ for all $i\ge 0$.
Then the sequence $\{h_i(x_i)\}_{i\ge 0}$ serves as the dependent analogue of the observations
$\{X_i\}_{i\ge 0}$ considered above. 
Recent work studied geometrically ergodic Markov chains satisfying a drift condition with Lyapunov function $V$~\citep{srikant2024rates}.
 In the homogeneous case ($h_i\equiv h$), they obtained
a rate $\mathcal{O}(n^{-\delta/2})$ under the domination condition $\|h\|^{2+\delta}\le V$ for $\delta\in(0,1)$,
and a rate $\mathcal{O}(n^{-1/2}\log n)$ under $\|h\|^{3}\le V$. Time-inhomogeneous $\{h_i\}_{i\ge 0}$ are considered in follow up work~\citep{wu2025uncertainty}, but only a $\mathcal{O}(n^{-1/2}\log n)$ rate is obtained under
stronger assumptions: the transition kernel has a positive spectral gap, and the functions
$h_i$ are uniformly bounded.
Both results are restricted to $\W_1$ and are suboptimal due to the additional
logarithmic factor; they are summarized in the yellow-shaded region of
Table~\ref{tab:summary}.

In the work by~\cite{srikant2024rates}, the authors highlight several open problems for geometric ergodic Markov chain CLT rates,
including: (i) whether the optimal $\mathcal O(n^{-1/2})$ $\W_1$ rate can
be achieved (under mild conditions), and (ii) what $\W_p$ CLT rates are attainable when $p>1$.

In this work, we establish CLT rates in $\W_p$ ($p\ge 1$)---optimal in some regimes---for
multivariate locally dependent sequences and geometric ergodic Markov chains under mild moment 
conditions. Before stating our contributions, we highlight the main challenges in obtaining
such results.

\begin{table}[h]
\centering
\begin{tabular}{|p{5.4cm}|p{2cm}|p{4cm}|p{3.9cm}|}
\toprule
Data&Metric& Assumptions& Rate\\
\midrule
\rowcolor{myrow}Independence~\citep{gallouet2018regularity,raivc2018multivariate}&$\W_1$&$\sup_{i \geq 0}\E[\|X_i\|^{2+\delta}]<\infty$ ($\delta \in (0,1]$)&$\mathcal O(n^{-\delta/2})$\\\hline
\rowcolor{myrow}Independence~\citep{bonis2020stein}&$\W_p$ ($p \geq 2$)&$\sup_{i \geq 0}\E[\|X_i\|^{p+q}]<\infty$ ($q \in (0,2]$)&$\mathcal O(n^{-\tfrac{p+q-2}{2p}})$\\\hline
\rowcolor{red!8}Local Dependence (Theorem \ref{thm:local-depend-1})&$\W_1$&$\sup_{i \geq 0}\E[\|X_i\|^{2+\delta}]<\infty$ ($\delta \in (0,1]$)&$\mathcal O ( n^{-\delta/2})$\\\hline
\rowcolor{red!8}$M$-Dependence (Theorem \ref{thm:m-depend-p})&$\W_p$ $ (p \geq 2)$&$\sup_{i \geq 0}\E[\|X_i\|^{p+q}]<\infty$ ($q \in (0,2]$)&$\mathcal O (n^{-\frac{p+q-2}{2(2p+q-2)}})$\\\hline
\rowcolor{yellow!8}Geometric Ergodic (GE) Markov Chains~\citep{srikant2024rates}&$\W_1$&$h_i\equiv h$, $\|h\|^{2+\delta} \leq V$  ($\delta \in (0,1]$)&$
    \mathcal{O}(n^{-\delta/2})$ if $\delta \in (0,1)$; $\mathcal{O}(  n^{-1/2}\log n)$ if $\delta =1$\\\hline
\rowcolor{yellow!8}Positive Spectral Gap Markov Chains~\citep{wu2025uncertainty}&$\W_1$&$h_i$ is bounded&$\mathcal{O}(  n^{-1/2}\log n)$\\\hline
\rowcolor{red!8}GE Markov Chains (Theorem \ref{thm:MC-W1})&$\W_1$&$\|h_i\|^{2+\delta} \leq V$  ($\delta >1$)& $\mathcal{O}(  n^{-1/2})$\\\hline
\rowcolor{red!8}GE Markov Chains (Theorem \ref{thm:MC-Wp})&$\W_p$ ($p \geq 2$)&$h_i\equiv h$, $\|h\|^{p+q} \leq V$  ($q \in (0,2]$)&$\mathcal O (n^{-\frac{p+q-2}{2(2p+q-2)}})$\\
\bottomrule
\end{tabular}
\caption{Summary of prior work and our results on CLT rates for independent, dependent and Markovian data. Here, $V$ denotes the Lyapunov function that certifies geometric ergodicity of Markov chain (Assumption~\ref{assumption:markovchain}).
}
\label{tab:summary}
\end{table}
\subsection{Challenges for $\W_p$ CLT Rates with Multivariate Dependent Data
}\label{sec:challenge}
We emphasize that extending $\W_p$ CLT rates ($p\ge 1$) from the univariate setting ($d=1$)
to the multivariate setting ($d\ge 2$) is highly nontrivial, even for independent data.
Two main challenges arise in the multivariate case.

The first challenge arises when one attempts to leverage regularity bounds for
solutions to Stein's equation in order to obtain the \emph{optimal}
$\mathcal{O}(n^{-1/2})$ multivariate CLT rate in $\W_1$. Regularities for
Stein solutions are standard ingredients in Stein's method and have long been
used to derive classical \emph{univariate} CLT rates, for instance, in the
$\W_1$ and Kolmogorov distances \citep{chen2010normal,nourdin2012normal}.
However, a key obstacle emerges when $d\ge 2$. To illustrate this issue, we
recall a regularity lemma from \cite{gallouet2018regularity}; the
notation used below is defined in Section~\ref{sec:notation}.
\begin{lemma}[Proposition 2.2 in \cite{gallouet2018regularity}]\label{lem:regularity} When $d\ge 2$, for any $h\in C^{0,1}$ let $f_h$ denote the solution to the multivariate Stein equation $\Delta f_h(x) - x^{\top}\nabla f_h(x)=h(x)-\E[h(Z)]$ for any $x\in\mathbb{R}^d$, where $Z\sim\mathcal N(0,I)$. Then $f_h\in C^{2,\delta}$ and $[\nabla^2 f_h]_{\operatorname{Lip},\delta} \lesssim \frac{1}{1-\delta}$ for any $\delta\in(0,1)$. \end{lemma}
Unlike the univariate case, where an analogous regularity holds even for
$\delta=1$ \citep[Lemma~6]{barbour1986asymptotic}, in the multivariate setting
 $\delta$ must be taken strictly less than one; moreover, it has been shown that
$f_h$ need not belong to $C^{2,1}$ \citep{raivc2004multivariate}.
 For illustration, suppose that the data are i.i.d.\ and satisfy a uniform
$(2+\alpha)$-moment condition for some $\alpha\in(0,1]$.
 As shown in the proof of
\cite[Theorem~1.1]{gallouet2018regularity}, choosing any $\delta\in(0,1)$ with
$\delta\le \alpha$ leads to an error term of the form $n \cdot (n^{-1/2})^{2+\delta} \cdot \frac{1}{1-\delta} \in \mathcal{O}(\frac{n^{-\delta/2}}{1-\delta} )$.  If $\alpha<1$, choosing $\delta=\alpha$ yields the rate $\mathcal O(n^{-\alpha/2})$.
This matches the optimal dependence on $n$ under a finite
$(2+\alpha)$-moment in the univariate case \citep{ibragimov1966accuracy}, demonstrating that Lemma~\ref{lem:regularity} is sufficient in this regime.
However, when $\alpha=1$ (finite third moment),
Lemma~\ref{lem:regularity} forces $\delta<1$, and the optimal choice is
$\delta=1-2/\log n$, which yields the suboptimal
$\mathcal O(n^{-1/2}\log n)$ rate rather than $\mathcal O(n^{-1/2})$. In summary, the first challenge can be stated as follows.
\begin{challenge}\label{challenge:regularity} Relying on Lemma~\ref{lem:regularity} alone precludes achieving the optimal
$\mathcal{O}(n^{-1/2})$ rate.

\end{challenge}
The second challenge concerns multivariate CLT rates in $\W_p$ for $p>1$.
When $p=1$, the Kantorovich--Rubinstein duality $\W_1(\mu,\nu) = \zeta_1(\mu,\nu):= \sup_{f\in C^{0,1}}|\int f\mathrm d\mu-\int f\mathrm d\nu|$ provides a powerful route to bounding $\W_1$ rates via test functions.
In the univariate setting, an analogous comparison is available for $p>1$, namely $\W_p(\mu,\nu) \lesssim \zeta_p(\mu,\nu):=\sup_{f\in C^{p-1,1}}|\int f\mathrm d\mu-\int f\mathrm d\nu|$; see \cite{liu2023wasserstein}. In contrast, when $d\ge 2$, the situation is more subtle.

\begin{challenge}\label{challenge:dual}
Unlike the univariate setting, in high dimensions the Zolotarev distance $\zeta_p$
might not be well defined for $p>1$, and the precise relationship between $\W_p$ and $\zeta_p$
is not fully understood; see \cite{bobkov2024fourier}. Consequently, techniques that control
$\zeta_1$ do not directly yield upper bounds on $\W_p$ when $p>1$ and $d \geq 2$.
\end{challenge}
A natural workaround is to consider the
\emph{sliced} or \emph{max–sliced} Wasserstein distances \(\widetilde \W_p\) built from one–dimensional
projections, and then invoke univariate \(\W_p\) CLTs (e.g. \cite{liu2023wasserstein}). However, for \(p\ge2\), there is no general reverse inequality that controls \(\W_p\) by a function of \(\widetilde \W_p\).
Even for $p=1$, the recent quantified Cramér–Wold inequality \citep{bobkov2024quantified} controls $\W_1$ by a \emph{power} of the max-sliced $\W_1$ with an exponent degrading as $1/d$, leading at best to rates of order $\mathcal{O}(n^{-1/(2d)})$. Such bounds are
too weak in high dimensions.

\noindent\textbf{Resolving Challenges~\ref{challenge:regularity} and~\ref{challenge:dual} in the Independent Case:} Recent work~\citep{raivc2018multivariate} proposed a new framework for Gaussian approximation
(rigorously reviewed in Section~\ref{sec:raivc}) and showed how
it yields the optimal $\mathcal{O}(n^{-1/2})$ rate in $\W_1$ for the independent setting. For $\W_p$ with $p\ge 2$, \cite{bonis2020stein} developed an approach that uses  exchangeable pairs and attains the optimal rate under
the moment conditions summarized in Table~\ref{tab:summary}.

\noindent\textbf{Persistence of Challenges Under Dependence:} We emphasize that Challenges~\ref{challenge:regularity} and~\ref{challenge:dual} persist in
the dependent setting. Specifically, Challenge~\ref{challenge:regularity} remains a bottleneck in obtaining the optimal
$\W_1$ CLT rates for Markov chains: existing techniques~\citep{srikant2024rates,wu2025uncertainty} rely on Lemma~\ref{lem:regularity} and consequently fail to attain the optimal rate. Furthermore, the methods of
\cite{raivc2018multivariate} and \cite{bonis2020stein} do not directly apply beyond the independent case:
the key objects in \cite{raivc2018multivariate} and the exchangeable-pair construction in \cite{bonis2020stein} crucially rely on independence, making them difficult to apply to dependent data directly. 

\subsection{Our Contributions}
Our main contributions are twofold, corresponding to the locally dependent and geometric ergodic Markov chain
settings, which we discuss separately. A roadmap for the main results of this paper is provided at the end of this subsection; see Figure~\ref{fig:roadmap}.

\noindent\textbf{Contributions for Locally Dependent Data:} 
Section~\ref{sec:local-dependent} presents our $\W_p$ CLT rates for locally dependent
data, and Section~\ref{sec:outline-local} outlines the proof techniques.
Specifically, we show the following:
\begin{itemize}
  \item First, we extend the framework of \cite{raivc2018multivariate} by deriving an auxiliary bound on the $\W_1$ Gaussian approximation error for a \emph{broad class of dependent} (not necessarily locally dependent) partial sums $S_n$ (Proposition \ref{prop:newS}). Specifically, the error is
controlled by how much the law of $S_n$ changes when conditioning on a single summand $X_i$.
This bound makes $\W_1$ CLT rates under dependence more tractable and may be of
independent interest; see Section~\ref{sec:ourS}.

\item For \(\W_1\), Proposition~\ref{prop:newS} yields the optimal
\(\mathcal O(n^{-1/2})\) rate for multivariate locally dependent sequences under a finite third-moment assumption, via a block-resampling coupling; see Section~\ref{sec:outline-local-depend-1} for details. More generally, we obtain a rate $\mathcal{O}(n^{-\delta/2})$
under a finite $(2+\delta)$-moment condition for $\delta \in (0,1]$, matching the best-known dependence on $n$ in
the independent case. These results are stated in Theorem~\ref{thm:local-depend-1} and summarized in the first red-shaded row of
Table~\ref{tab:summary}.
\item For \(\W_p\) with \(p \ge 2\), we establish the rate $\mathcal O(n^{-\frac{p+q-2}{2(2p+q-2)}})$ for multivariate \(M\)-dependent sequences
\citep{romano2000more,diananda1955central}, an important subclass of locally dependent data, under a uniform \((p+q)\)-moment bound. The proof proceeds via a big--small block decomposition, which transfers rates for independent data to the \(M\)-dependent setting; see Section~\ref{sec:outline-m-depend-p} for details. To our knowledge, this result provides the
\emph{first} $\W_p$ CLT rate for multivariate $M$-dependent data under mild moment conditions.
As highlighted in Section~\ref{sec:outline-MC-Wp}, this result serves as a key input for our
$\W_p$ ($p\ge 2$) CLT rates for Markov chains. These results are stated in Theorem~\ref{thm:m-depend-p} and summarized in the second red-shaded row of
Table~\ref{tab:summary}.
\item As an application of our optimal $\W_1$ rate for locally dependent sequences, we
obtain the first optimal $\W_1$ CLT rate for multivariate $U$-statistics
\citep{hoeffding1992class}. Existing results typically focus on the univariate
setting or consider different (often stronger) metrics, such as the
hyperrectangle distance. This application is discussed in Section~\ref{sec:U}.

\end{itemize}

\noindent\textbf{Contributions for Markov Chains:} Section~\ref{sec:MC} presents our $\W_p$ CLT rates for geometric ergodic Markov chains, and
Section~\ref{sec:outline-berry} outlines the proof techniques. Specifically, our contributions include:
\begin{itemize}
\item We employ the split-chain construction of \citet{nummelin1978splitting} to establish the CLT rates in both \(\W_1\) and \(\W_p\) for \(p \ge 2\). In particular, we prove that the regeneration time of the split chain associated with a geometrically ergodic Markov chain has a geometric tail (see Lemma~\ref{lem:geom_tail}), without requiring strong aperiodicity or other restrictive assumptions often imposed in the literature. The details are presented
in Section~\ref{sec:splitchain} and may be of independent interest.

    \item For $\W_1$, we resolve the open problem (i) posed in \cite{srikant2024rates} by establishing a $\mathcal O(n^{-1/2})$ $\W_1$ CLT rates for geometric ergodic Markov chain. On the technical side, we bypass the reliance on Lemma~\ref{lem:regularity}
found in prior work \citep{srikant2024rates,wu2025uncertainty} and instead leverage our
Proposition \ref{prop:newS}. Our approach utilizes the split-chain construction to
characterize fundamental properties of the time-reversed chain. We then employ a
combination of forward and backward maximal couplings to localize the effect of
conditioning on a single time point of the chain.
 These results correspond to the third red-shaded row of
Table~\ref{tab:summary} and are presented in Theorem~\ref{thm:MC-W1}.

    \item For $\W_p$ with $p\ge 2$, we resolve open problem (ii) posed in \cite{srikant2024rates}
by establishing a $\W_p$--CLT convergence rate for geometrically ergodic Markov
chains that match the current state-of-the-art $\W_p$--CLT rates for $M$-dependent data.
Our technical approach uses the split-chain construction to obtain a regeneration
decomposition, where successive regenerative blocks form a $1$-dependent
sequence. This reduction enables us to invoke our newly developed $\W_p$ ($p\ge 2$) CLT rates for
$M$-dependent data in the Markovian setting.
 These findings correspond to the fourth red-shaded row of
Table~\ref{tab:summary} and are presented in Theorem~\ref{thm:MC-Wp}.

\end{itemize}
\begin{figure}[htbp]
\centering
\begin{tikzpicture}[
    >=Latex,
    line width=0.95pt,
    font=\small,
    result/.style={
        draw,
        rounded corners=5pt,
        fill=black!7,
        text width=3.2cm,
        minimum height=1.1cm,
        align=center
    },
    lemma/.style={
        draw,
        rounded corners=6pt,
        fill=black!5,
        text width=4.7cm,
        inner sep=6pt,
        align=center
    },
    method/.style={
        font=\footnotesize,
        align=center,
        inner sep=2pt,
        text width=10cm
    },
    flow/.style={-{Latex[length=2.4mm,width=1.8mm]}, semithick},
    support/.style={-{Latex[length=2.2mm,width=1.6mm]}, semithick, dashed}
]

\begin{pgfonlayer}{background}
    \filldraw[rounded corners=8pt, fill=black!3, draw=black!15]
        (0,3.35) rectangle (16,0.35);
    \filldraw[rounded corners=8pt, fill=black!3, draw=black!15]
        (0,-0.4) rectangle (16,-2.60);
\end{pgfonlayer}

\node[anchor=west, font=\bfseries] at (0.45,2.95) {CLT Rates in $\W_1$};
\node[anchor=west, font=\bfseries] at (0.45,-0.8) {CLT Rates in $\W_p$};

\node[result, text width=3.0cm] (aux) at (2.2,1.9)
{Auxiliary Bound\\(Proposition~\ref{prop:newS})};

\node[result] (local) at (14.0,2.6)
{Local Dependence (Theorem~\ref{thm:local-depend-1})};

\node[result] (mcwone) at (14.0,1.2)
{Markov Chains (Theorem~\ref{thm:MC-W1})};

\node[method] (block) at (8,2.85)
{Block-Resampling Coupling (Section~\ref{sec:outline-local-depend-1})};

\node[method] (fbmc) at (8.15,1.4)
{Forward--Backward Maximal Coupling (Section~\ref{sec:outline-MC-W1})};

\coordinate (fork) at (4.4,1.9);

\draw[flow] (aux.east) -- (fork) |- (local.west);
\draw[flow] (aux.east) -- (fork) |- (mcwone.west);

\node[result, text width=2.8cm] (ind) at (2.2,-1.85)
{Independence};

\node[result] (mdep) at (8.1,-1.85)
{$M$-Dependence\\(Theorem~\ref{thm:m-depend-p})};

\node[result] (mcwp) at (14,-1.85)
{Markov Chains\\(Theorem~\ref{thm:MC-Wp})};

\node[method] (bsb) at (5.1,-1.85)
{Big--Small Block\\(Section~\ref{sec:outline-m-depend-p})};

\node[method] (regen) at (11,-1.85)
{Regeneration\\(Section~\ref{sec:outline-MC-Wp})};

\draw[flow] (ind.east) -- (mdep.west);
\draw[flow] (mdep.east) -- (mcwp.west);

\node[lemma] (geom) at (8.15,0)
{Geometric Tails of Split-Chain\\Regeneration (Lemma~\ref{lem:geom_tail})};

\draw[support]
    (geom.north) to (fbmc.south);

\draw[support]
    (geom.south) to[out=-35,in=120] (regen.north);

\end{tikzpicture}
\caption{Results roadmap. Dashed arrows indicate where Lemma~\ref{lem:geom_tail} enters the argument.}
\label{fig:roadmap}
\end{figure}
\subsection{Related Work}
\noindent\textbf{Central Limit Theorem Rates and Probability Metrics.} 
A classical line of work studies Gaussian approximation in \emph{integral probability metrics} \citep{muller1997integral}. In the univariate
case, the Kolmogorov distance admits the celebrated Berry--Esseen bound of order
$\mathcal O(n^{-1/2})$ under a finite third-moment assumption~\citep{petrov2012sums}. In higher dimensions, a common analogue is the
\emph{convex-set distance} $\sup_{A\in\mathcal C}|\P(S_n/\sqrt n\in A)-\P(Z\in A)|$, where $\mathcal{C}$
denotes the class of convex Borel sets in $\R^d$ and
$Z\sim\mathcal N(0,I_d)$. Sharp dimension dependence in Berry--Esseen-type bounds
for convex sets was established by~\cite{bentkus2003dependence}. Another prominent
metric in modern high-dimensional statistics is the \emph{hyperrectangle
distance} \citep{chen2018gaussian,chen2019randomized}, defined as the supremum
over axis-aligned rectangles.

Recently, transport distances have attracted significant attention in CLT theory
\citep{rio2009upper,gallouet2018regularity,raivc2018multivariate,bonis2020stein,srikant2024rates}.
These metrics endow probability laws with the geometry of the underlying space by
interpreting distributional discrepancy as the minimal cost of transporting mass
\citep{villani2008optimal}. This geometric perspective, together with stability
properties \citep{kuhn2019wasserstein} and powerful dual
representations (e.g., against Lipschitz test functions for $\W_1$), makes
transport-based CLT rates both robust and practically meaningful, with broad
applications in modern ML theory \citep{panaretos2019statistical,huo2024collusion}.

\noindent\textbf{Nummelin’s Splitting.} 
Nummelin's splitting construction \citep{nummelin1978splitting} is a foundational
tool for introducing \emph{regeneration} into general state-space Markov chains.
Splitting and regeneration are now standard in the stability and ergodic theory
of Markov chains \citep{Meyn12_book}, and they also underpin modern MCMC analysis \citep{mykland1995regeneration,jones2001honest}.

Beyond asymptotic ergodic results, splitting-based regenerative decompositions
have been widely used to derive non-asymptotic concentration inequalities for
additive functionals of Markov chains:
the works \citep{adamczak2008tail,adamczak2015exponential,bertail2018new} focus on
the strongly aperiodic setting, while \cite{lemanczyk2021general} treats general
$m$-step minorization. Regeneration methods have also been used to obtain
Berry--Esseen-type bounds for Markov chains
\citep{bolthausen1980berry,bolthausen1982berry,douc2008bounds}; however, these
results are largely restricted to the $m=1$ case, the univariate setting, and
metrics such as the Kolmogorov distance. Our work extends the use of split
chains by establishing regeneration properties under general geometric
ergodicity, and by leveraging the resulting block structure to derive
multivariate $\W_p$ ($p\ge 1$) CLT rates.

\subsection{Notation}\label{sec:notation}
We write $\nabla$ and $\Delta$ for the gradient and Laplacian operators on
$\mathbb R^d$, respectively. For $k\in\mathbb N$ and $\alpha\in(0,1]$, the H\"older
class $C^{k,\alpha}$ consists of functions
$f:\mathbb R^d\to\mathbb R$ that are $k$ times continuously differentiable and
whose $k$th derivative is H\"older continuous with exponent $\alpha$, i.e.,
\[
[\nabla^k f]_{\operatorname{Lip},\alpha}
:=\sup_{x\neq y}\frac{\|\nabla^k f(x)-\nabla^k f(y)\|}{\|x-y\|^\alpha}<\infty.
\]
For $L>0$, let $\operatorname{Lip}_L$ denote the class of real-valued $L$-Lipschitz functions on
$\mathbb R^d$, i.e.,
\[
\operatorname{Lip}_L:=\{h:\mathbb R^d\to\mathbb R:\ |h(x)-h(y)|\le L\|x-y\|\ \ \forall x,y\}.
\]
By the Kantorovich--Rubinstein duality, $\W_1(\mu,\nu)=\sup_{h\in \operatorname{Lip}_1}\Bigl|\int h\d\mu-\int h\d\nu\Bigr|.$

For symmetric matrices $A,B$, we write
$A\succeq B$ (resp.\ $A\succ B$) if $A-B$ is positive semidefinite (resp.\
positive definite), and $\lambda_{\min}(A)$ for the smallest eigenvalue of $A$.

For $p\ge1$, $\|Z\|_{L^p}:=(\E\|Z\|^p)^{1/p}$ denotes the $L^p$ norm of a random
vector $Z$, and $L^q(\P)$ denotes the usual space of $q$-integrable random
variables under $\P$. Unless stated otherwise, we use the standard order notations
$\mathcal O(\cdot)$ and $\Omega(\cdot)$ to hide absolute constants that do not
depend on the sample size $n$. We also write $a\lesssim b$ (resp.\ $a\gtrsim b$)
to mean that $a\in\mathcal O(b)$ (resp.\ $a\in\Omega(b)$).

Given a Markov kernel $P$ and a measurable function $V:\mathcal X\to\R$, we write
$PV(x):=\int V(y)P(x,\d y)$. For a probability measure $\nu$ on $\mathcal X$, we use
$\nu(V):=\int V\d\nu$. Given two measurable functions $V,V':\mathcal X\to\R$, we write $V\le V'$ to
denote pointwise domination, i.e., $V(x)\le V'(x)$ for all $x\in\mathcal X$, and
$V\lesssim V'$ to denote $V(x)\le CV'(x)$ for all $x\in\mathcal X$ and some
 constant $C>0$ that do not depends on the sample size $n$. Finally, we adopt the convention that
$\sum_{i=a}^b(\cdot)=0$ whenever $b<a$.

\section{Central Limit Theorems for Locally Dependent Data}\label{sec:local-dependent}
In this section, we focus on locally dependent data~\citep{rinott1994normal}. Let $I=\{0,1,\dots,n-1\}$, and let
$\Gamma=(I,E)$ be a graph with vertex set $I$ and edge set $E$. We say that
$\Gamma$ is a dependence graph for the collection $\{X_i:i\in I\}$ if, for
any disjoint subsets $K,L\subseteq I$ such that there is no edge in $\Gamma$
between $K$ and $L$, the subcollections $\{X_k:k\in K\}$ and $\{X_\ell:\ell\in L\}$
are independent. Write $\deg(i;\Gamma)$ for the degree of vertex $i$ in $\Gamma$, and define $D :=1+\max_{i \in I} \operatorname{deg}(i;\Gamma)$, noting that $D$ may depend on $n$. For $i\in I$, let $\mathcal N[i]$ denote the
closed neighborhood of $i$ in $\Gamma$, namely
\[
\mathcal N[i]:=\{i\}\cup\{j\in I:\ (i,j)\in E\}.
\]
A common subclass of locally dependent data is $M$-dependence (where $M$
may depend on $n$); see, e.g., \cite{diananda1955central,romano2000more}. We say
that $\{X_i:i\in I\}$ is $M$-dependent if the subcollections
$\{X_i:i\in K\}$ and $\{X_j:j\in L\}$ are independent whenever
\[
\min\{|i-j|:\ i\in K,\ j\in L\}>M.
\]
Equivalently, the dependence graph $\Gamma$ has an edge between distinct
vertices $i\neq j$ if and only if $0<|i-j|\le M$. In particular, the maximum
degree satisfies $D\le 1+2M$. The case $M=0$ reduces to the independent
setting (no edges, hence $D=1$).

We consider a collection \(\{X_i : i \in I\}\) of \(\mathbb{R}^d\)-valued random vectors. The same analysis applies to matrix- and tensor-valued random variables by vectorizing each \(X_i\) and modifying the norm geometry accordingly. Without loss of generality, we assume the sequence is centered:
$\E[X_i]=0$ for all $i\in I$. Define
\[
S_n:=\sum_{i=0}^{n-1} X_i,
\qquad
\Sigma_n:=\frac{1}{n}\operatorname{Var}(S_n).
\]
In this section, we focus on the nondegenerate case  where
\(\lambda_{\min}(\Sigma_n) \in \Omega(1)\).

\subsection{CLT Rates in Wasserstein-1 Distance}\label{sec:depend-pequalto1}
We now state a theorem that quantifies the gap between
\(\law\big(\tfrac{S_n}{\sqrt{n}}\big)\) and the Gaussian  \(\mathcal N(0,\Sigma_n)\) in $\W_1$.

\begin{theorem}[Wasserstein--1 CLT Rates for Locally Dependent Data]\label{thm:local-depend-1}
Let $\{X_i\}_{i=0}^{n-1}$ be locally dependent as defined in Section~\ref{sec:local-dependent}. Assume that  $\sup_{0\le i\le n-1}\mathbb{E}\left[\|X_i\|^{2+\delta}\right]\in \mathcal O(1)$ for some $\delta\in(0,1]$. Then
\[
\W_1\big(\mathcal{L}(\tfrac{S_n}{\sqrt{n}}),\mathcal{N}(0,\Sigma_n)\big)
\in\begin{cases}
\mathcal{O} \Big(\frac{\lambda_{\max}(\Sigma_n)^{1/2}}{\lambda_{\min}(\Sigma_n)^{\delta/2}}\cdot D\cdot n^{-\delta/2}\Big) & \text{ if } \delta \in (0,1),\\
\mathcal{O} \Big(\frac{\lambda_{\max}(\Sigma_n)^{1/2}}{\lambda_{\min}(\Sigma_n)^{3/2}}\cdot D^2\cdot n^{-1/2}\Big) & \text{ if } \delta=1.
\end{cases} 
\]
Here, $\mathcal{O}(\cdot)$ hides constants that are independent of $D$ and $n$.
\end{theorem}

To our knowledge, Theorem \ref{thm:local-depend-1} establishes  the first \emph{optimal} \(W_1\) CLT rates for
multivariate locally dependent data under a finite \((2+\delta)\)-th moment assumption with \(\delta\in(0,1]\). The dependence on $n$ matches the optimal rates known in the i.i.d.\ setting for $\delta\in(0,1)$ \citep{gallouet2018regularity} and the independent case for $\delta=1$ \citep{raivc2018multivariate}.
Furthermore, Theorem~\ref{thm:local-depend-1} extends the univariate results of
\cite{liu2023wasserstein} to the multivariate setting and provides a comprehensive treatment of the full range $\delta\in(0,1]$, whereas
\cite{liu2023wasserstein} focuses only on the case $\delta=1$. The dependence on \(D\) in Theorem \ref{thm:local-depend-1} is the same as that in the univariate case; see \cite{barbour1989central,liu2023wasserstein}.

We remark that in Theorem~\ref{thm:local-depend-1} and all other CLT rate results in this paper, the dependence on the dimension is subsumed into the order notation. The optimal dimension dependence of CLT rates in \(\W_p\) remain poorly understood, even for multivariate independent data. Moreover, currently there is no unified approach to formulating even potentially suboptimal dimension dependence in the independent setting. For instance, when \(p=1\), \cite{gallouet2018regularity} consider a finite-moment framework in which the moment assumptions do not explicitly encode dimension dependence, whereas for \(p\ge 2\), \cite{bonis2020stein} allows the moment terms to depend on the dimension through the distribution of the data. Deriving sharp dimension-dependent bounds is beyond the scope of this paper. We leave this question to future work and discuss several promising directions in Section~\ref{sec:conclusion}.

\cite{raivc2018multivariate} developed an abstract Gaussian approximation framework that achieved the optimal multivariate $\W_1$ CLT rate for $\delta=1$ in the independent
setting. However, this success relies on constructing auxiliary objects specifically tailored to independence, which does not extend directly to dependent data. In this work, we generalize the framework of~\cite{raivc2018multivariate} by introducing novel objects that capture dependent structures, resulting a tractable bound~\eqref{eq:newS} on the $\W_1$ CLT error for general dependent data (Proposition \ref{prop:newS}). Leveraging this bound \eqref{eq:newS}, we
establish the optimal rate $\mathcal{O}(n^{-1/2})$ in Theorem~\ref{thm:local-depend-1} for
$\delta=1$. For $\delta\in(0,1)$, we apply the Stein-equation regularity theory
\citep{gallouet2018regularity} from the i.i.d.\ setting to locally dependent sequences.
An overview of Rai\v{c}'s framework, our construction of new objects, and a proof sketch are provided in
Section~\ref{sec:outline-local}, with full details deferred to Appendix~\ref{sec:local-depend-1}.

\subsection{CLT Rates in Wasserstein-p ($p \geq 2$) Distance}\label{sec:depend-plargethan2}
As discussed in Section~\ref{sec:challenge}, extending multivariate CLT rates in $\W_p$
from the case $p=1$ to $p\ge 2$ is \emph{not} a routine generalization; in particular,
techniques that control $\W_1$ do not directly yield bounds for $\W_p$ when $p\ge 2$.
Bonis \cite[Theorems~4 and~9]{bonis2020stein} developed an exchangeable-pairs framework for
Gaussian approximation that attains optimal multivariate $\W_p$ rates for independent
data with $p\ge 1$. Building on this approach, recent work~\citep{fang2023p} derived $\W_p$ bounds of order $\mathcal{O}(n^{-1/2}\log n)$ for $p\ge 2$ under \emph{bounded} locally dependent data.
Nevertheless, multivariate $\W_p$ ($p\ge 2$) CLT rates for locally dependent sequences
under mild moment assumptions remain poorly understood.

To narrow this gap, we establish the following theorem, which provides multivariate 
$\W_p$ ($p\ge 2$) CLT rates for $M$-dependent data under mild moment conditions.
\begin{theorem}[Wasserstein--p CLT Rates for $M$-Dependent Data]\label{thm:m-depend-p}
Let $\{X_i\}_{i=0}^{n-1}$ be $M$-dependent as defined in Section~\ref{sec:local-dependent}. Assume that $\sup_{0\le i\le n-1}\mathbb{E}\left[\|X_i\|^{p+q}\right]\in \mathcal O(1)$ for some $p \geq 2$ and $q \in(0,2]$. Then
\[
\W_p\big(\mathcal{L}(\tfrac{S_n}{\sqrt{n}}),\mathcal{N}(0,\Sigma_n)\big)
\in \mathcal{O}\Big(
\frac{\lambda_{\max}(\Sigma_n)^{1/2}}{\lambda_{\min}(\Sigma_n)^{1/2}}\cdot
(M+1)^{1+\frac{2-q}{2(2p+q-2)}}\cdot
n^{-\frac{p+q-2}{2(2p+q-2)}}
\Big).
\]
Here, $\mathcal{O}(\cdot)$ hides constants that are independent of $M$ and $n$.
\end{theorem}

We emphasize that, to our knowledge, Theorem~\ref{thm:m-depend-p} provides the \emph{first} multivariate $\W_p$ CLT rates ($p\ge2$) for $M-$dependent data under finite-moment assumptions. We also acknowledge that the bound may not be optimal: when $q=2$ our best dependence on $n$ is $\mathcal{O}(n^{-1/4})$, whereas the common conjectured optimal rate is $\mathcal{O}(n^{-1/2})$; closing this gap in full generality remains open. Nevertheless, Theorem~\ref{thm:m-depend-p} significantly narrows the gap between the $\W_1$ theory and the $\W_p$ ($p\ge2$) regime in high dimensions. Furthermore, it enables the characterization of Markov chain CLT rates under $\W_p$ for $p\ge2$ as we will discuss in Section \ref{sec:outline-MC-Wp}, where existing results for bounded random variables \citep{fang2023p} do not apply. Obtaining optimal $\W_p$ CLT rates for general locally dependent sequences under
finite-moment assumptions remains an important direction for future research. One
promising avenue is to extend the exchangeable-pair framework of
\cite{bonis2020stein}, which is currently tailored to independent data, to
settings with dependence.

We establish Theorem~\ref{thm:m-depend-p} by combining the multivariate CLT rates for
independent data~\citep{bonis2020stein} with a standard big–small block decomposition~\cite[Equation 17.36]{Meyn12_book}. The argument reduces the sum of
$M$-dependent vectors to a sum of independent \emph{big} blocks, while deriving a
precise upper bound for the contribution of the \emph{small} blocks (of length $M$),
together with an optimal choice of the big-block length. An outline of the proof is
given in Section~\ref{sec:outline-m-depend-p}, and full details appear in
Appendix~\ref{sec:m-depend-p}.

\subsection{Applications to U-statistics}\label{sec:U}
In this section, we apply our multivariate  $\W_1$ CLT rates for locally dependent data
(Theorem~\ref{thm:local-depend-1}) to derive $W_1$ CLT rates for $U$-statistics. Let $\{Z_i\}_{i=0}^{n-1}$ be i.i.d.\ random vectors taking values in a measurable space
$(\mathcal Z,\mathcal B(\mathcal Z))$. Fix an integer $r\ge1$, and let
$h:\mathcal Z^{r}\to\R^d$ be a symmetric measurable kernel. Denote the target
\[
\theta := \E[h(Z_0,\dots,Z_{r-1})]\in\R^d .
\]
A natural estimator for $\theta$ is the U-statistic of order $r$ \citep{hoeffding1992class}:
\begin{equation}\label{eq:U}
U_n := \binom{n}{r}^{-1}
\sum_{0\le i_1<\cdots<i_r\le n-1} h\left(Z_{i_1},\ldots,Z_{i_r}\right).
\end{equation}
We emphasize that $U$-statistics have numerous practical applications. Notable examples
include covariance matrix estimation
\citep{dempster1972covariance,bickel2008covariance,bickel2008regularized}
and subbagging/ensemble methods such as random forests
\citep{breiman1996bagging,breiman2001random,mentch2016quantifying}.
Below we present two concrete illustrations.
\begin{example}[\textbf{Sample Covariance}]\label{example:sample-covariance}
Let $\{Z_i\}_{i=0}^{n-1}\subset\R^p$ be i.i.d.\ with mean $\mu:=\E[Z_0]$ and covariance
$\Sigma:=\operatorname{Var}(Z_0)$. Define $\bar Z_n:=\frac{1}{n}\sum_{i=0}^{n-1} Z_i$ and $\hat\Sigma_n:=\frac{1}{n-1}\sum_{i=0}^{n-1}(Z_i-\bar Z_n)(Z_i-\bar Z_n)^\top .$ Then $\E[\hat\Sigma_n]=\Sigma$. Moreover, $\hat\Sigma_n$ is a matrix-valued ($d=p\times p$) $U$-statistic of order
$2$ with symmetric kernel
$h(z,z'):=\frac{1}{2}(z-z')(z-z')^\top$ for $z,z'\in\R^p$, namely,
\[
\hat\Sigma_n
=
\binom{n}{2}^{-1}\sum_{0\le i<j\le n-1} h(Z_i,Z_j).
\]

\end{example}
\begin{example}[\textbf{Subbaging}]\label{example:subbagging}
Consider supervised learning with features \(X\in\mathcal X\) and response \(Y\in\R^d\).
Suppose we observe an independent training sample  \(\{Z_i\}_{i=0}^{n-1} = \{(X_0,Y_0),\ldots,(X_{n-1},Y_{n-1})\}\).
Fix \(r\in\{1,\ldots,n\}\) and a symmetric base learner \(T\) that, given a subsample
\(\{i_1,\ldots,i_r\}\subset\{0,\ldots,n-1\}\), produces a prediction at a target point
\(x^*\), denoted $T_{x^*}\big(Z_{i_1},\dots, Z_{i_r}\big).$ The subbagged predictor \citep{breiman2001random,mentch2016quantifying} averages these
predictions over all \(\binom{n}{r}\) subsamples:
\[
b_n(x^*)
:= \binom{n}{r}^{-1}
   \sum_{0\le i_1<\cdots<i_r\le n-1}
   T_{x^*}\big(Z_{i_1},\dots, Z_{i_r}\big).
\]
Averaging over subsamples typically reduces variance (and hence improves predictive accuracy)
for base learners \citep{breiman1996bagging}.
\end{example}
Without loss of generality, assume $\theta=0$. We focus on the nondegenerate regime in which $\lambda_{\min}\big(\operatorname{Var}(\E[h(Z_0,\dots,Z_{r-1}) \mid Z_0])\big) \in \Omega(1)$. This is a standard assumption in the CLT literature for \(U\)-statistics
\citep{callaert1978berry,gotze1987approximations,chen2018gaussian}.  Under this assumption,
the following result characterizes $\W_1$ CLT rates for $U$-statistics of order $r$.
\begin{corollary}\label{co:U}
Given i.i.d.\ $\{Z_i\}_{i=0}^{n-1}$ as in Section \ref{sec:U} such that $\E[\|h(Z_0, \cdots, Z_{r-1})\|^3]<\infty$. We have
\begin{align*}
&\W_1\Big(\law\big(\tfrac{\sqrt{n}U_n}{r}\big), \mathcal N\big(0, \operatorname{Var}\big(\mathbb{E}[h(Z_0,\ldots,Z_{r-1}) \mid Z_0]\big)\big)\Big) \in \mathcal O (n^{-1/2}).
\end{align*}
\end{corollary}
To the best of our knowledge, Corollary~\ref{co:U} gives the first
\(\W_1\) CLT rate for multivariate \(U\)-statistics of order \(r\).  In contrast, much of the existing literature
focuses on Gaussian approximation in other stronger metrics, such as the
Kolmogorov distance or the hyperrectangle distance, including the univariate
result~\citep{callaert1978berry} and high-dimensional extensions~\citep{gotze1987approximations,chen2018gaussian}.
 Related work on locally dependent data
established univariate $\W_p$ CLT rates and discussed applications to $U$-statistics
\citep{liu2023wasserstein,fang2019wasserstein}, but these approaches do not extend 
to the multivariate setting, as explained in Section~\ref{sec:challenge}. The proof of Corollary~\ref{co:U} follows from Theorem~\ref{thm:local-depend-1} and is
deferred to Appendix~\ref{sec:proof-U}.

\section{Technical Overview of CLT Rates under Local and 
$M$-Dependence}\label{sec:outline-local}
This section is organized as follows. Section~\ref{sec:raivc} reviews the Gaussian
approximation framework of \cite{raivc2018multivariate}, which yields optimal multivariate
$\W_1$ CLT rates in the independent setting. Section~\ref{sec:ourS} extends this framework by deriving a tractable auxiliary bound on the
$\W_1$ CLT error for general dependent data.
 Sections
\ref{sec:outline-local-depend-1} and~\ref{sec:outline-m-depend-p} then outline the proofs of
Theorems~\ref{thm:local-depend-1} and~\ref{thm:m-depend-p}, respectively. Throughout, we
restrict to the normalization $\Sigma_n:=n^{-1}\operatorname{Var}(S_n)=I_d$; the general nondegenerate
case $\lambda_{\min}(\Sigma_n)=\Omega(1)$ follows by a standard linear-algebra reduction,
as detailed in Appendix~\ref{sec:local-depend-1}.

\subsection{Rai\v{c}'s Framework for Gaussian Approximation}\label{sec:raivc}
The work by \cite{raivc2018multivariate} presents a general multivariate Stein framework for Gaussian approximation. Specifically, let \(W\in\mathbb{R}^d\) satisfy \(\mathbb{E}\|W\|^2<\infty\), \(\mathbb{E}W=0\), and \(\operatorname{Var}(W)=I_d\), and let \(Z\sim\mathcal N(0,I_d)\). For test functions \(f:\mathbb{R}^d\to\mathbb{R}\) in the Hölder class \(\mathcal C^{0,1}\), the framework provides explicit bounds on \(|\mathbb{E}f(W)-\mathbb{E}f(Z)|\). The argument is summarized in three steps with suitable constructions of auxiliary objects:

\paragraph{Step 1:  Generalized size-biased transformation.} \textit{Construct} a measurable index space $\Xi$, a family of random vectors
$\{V_\xi:\xi\in\Xi\}$, and a vector measure $\mu$ on $\Xi$ such that
\[
\mathbb{E}[f(W)W]=\int_{\Xi}\mathbb{E}_\xi[f(V_\xi)]\mu(\d\xi).
\]
The above identity generalizes classical size-biased representations
\citep{baldi1989normal} and their multivariate extensions
\citep{goldstein1996multivariate}, which are standard tools for obtaining
quantitative CLT rates \citep{chen2010normal}.

\paragraph{Step 2: Interpolation.}
For each $\xi\in\Xi$, \emph{interpolate explicitly} along the straight line from $V_\xi$ to $W$:
\[
W_s^{(\xi)}:=(1-s)V_\xi+sW,\qquad s\in[0,1],\qquad Z_\xi:=W-V_\xi.
\]
\textit{Construct} a measurable selector
$\psi:\Xi\times[0,1]\times\mathbb{R}^d\to\Xi$ and define a $\mathbb{R}^d$–valued vector measure
\[
\nu_\xi(B):=\int_0^1\int_{\mathbb{R}^d}
\mathbf{1}\{\psi(\xi,s,z)\in B\} z
\law(Z_\xi)(\d z)\d s,\qquad \forall B\in\mathcal{B}(\Xi).
\]
The selector \(\psi\) is chosen such that the conditional laws match:
\begin{equation}\label{eq:stablility}
    \law\big(W_s^{(\xi)}\big|Z_\xi=z\big)=\law\big(V_{\psi(\xi,s,z)}\big),
\quad \forall s\in[0,1] \text{ and }\law(Z_\xi)\text{-a.e.\ }z\in\mathbb{R}^d
\end{equation}
Then, we obtain:
\[
\mathbb{E}f(W)-\mathbb{E}_\xi f(V_\xi)=\int_{\Xi}\big\langle \mathbb{E}_\eta[\nabla f(V_\eta)],\nu_\xi(\d\eta)\big\rangle.
\]
Consequently, the discrepancy between the laws of $W$ and $V_\xi$ is captured by the
vector measure $\nu_\xi$.

\paragraph{Step 3: Quantitative bounds.} By introducing some \(\beta\)-functionals, \cite[Theorem~2.9]{raivc2018multivariate}
provides quantitative bounds comparing \(W\) with the standard Gaussian
\(Z\sim\mathcal N(0,I_d)\):

\begin{lemma}[Theorem 2.9 of \cite{raivc2018multivariate}]\label{lem:raivc}
Consider the objects \(\mathcal S = (\Xi,\{V_\xi\}_{\xi\in\Xi},\mu,\psi,\{\nu_\xi\}_{\xi\in\Xi})\) constructed in
Steps~1–2.
Define the quantities
\begin{align*}
&\beta_1^{(\xi)}  :=|\nu_{\xi}|(\Xi),  \qquad \beta_2  :=\int_{\Xi} \beta_1^{(\xi)}|\mu|(\mathrm{d} \xi), 
\qquad \beta_2^{(\xi)}  :=\int_{\Xi} \beta_1^{(\eta)}|\nu_{\xi}|(\mathrm{d} \eta), \\
&\beta_{123}^{(\xi)}(a, b, c)  :=\int_{\Xi} \min \left\{a, b \beta_1^{(\eta)}+c \sqrt{\beta_2^{(\eta)}}\right\}|\nu_{\xi}|(\mathrm{d} \eta),\quad  \beta_{234}(a, b, c)  :=\int_{\Xi} \beta_{123}^{(\xi)}(a, b, c)|\mu|(\mathrm{d} \xi).
\end{align*}
If \(\beta_2<\infty\), then for any \(f:\mathbb R^d\to\mathbb R\),
\begin{align*}
 |\mathbb{E}[f(W)]-\E[f(Z)]| \leq \beta_{234}(1.8,3.58+0.55 \log d, 3.5) [f]_{\operatorname{Lip},1} .
\end{align*}
\end{lemma}

Through Steps~1 and~2, this framework overcomes Challenge~\ref{challenge:regularity} by
introducing a Gaussian smoothing of $f$.  This smoothing bypasses the need for a bounded third derivative; instead, the
analysis relies on second- and fourth-order derivative bounds for the smoothed
test function. This mechanism is made explicit in the proof of
Lemma~\ref{lem:raivc}; we refer further details to \cite{raivc2018multivariate}. This framework is instrumental in achieving the optimal $\mathcal{O}(n^{-1/2})$
CLT rate for \emph{independent} data under a finite third-moment assumption
\citep{raivc2018multivariate}. By contrast, approaches based on regularity bounds
for Stein's equation (e.g., \cite{gallouet2018regularity}) typically yield only
the suboptimal $\mathcal{O}\bigl(n^{-1/2}\log n\bigr)$ rate; see
Section~\ref{sec:challenge}.

\subsection{Constructing Object $\mathcal S$ for Dependent Data and A Tractable Auxiliary Bound}\label{sec:ourS} 

The $\W_1$ CLT rates for multivariate \emph{independent} data were established by \cite{raivc2018multivariate} via an explicit construction of the objects
\(
\mathcal S = (\Xi,\{V_\xi\}_{\xi\in\Xi},\mu,\psi,\{\nu_\xi\}_{\xi\in\Xi})
\).
However, this specific construction is restricted to the independent setting and 
does not extend to dependent data. In this work, we generalize Rai\v{c}'s approach  by introducing a novel construction of
$\mathcal S$ that faithfully captures the underlying dependence structure. Importantly,
our choice of $\mathcal S$ yields an auxiliary bound on the $\W_1$ CLT rate for general
dependent data, as stated in the following proposition.

\begin{proposition}\label{prop:newS}
Let $\{X_i\}_{i=0}^{n-1}$ be an $\R^d$-valued sequence such that $\sup_{0\le i\le n-1}\E\big[\|X_i\|^2\big] < \infty.$ Define the normalized variables $U_i := X_i/\sqrt{n}$ for $i\in\{0,\dots,n-1\}$ and $W := \sum_{i=0}^{n-1} U_i .$ Assume that $\E[W]=0$ and $\operatorname{Var}(W)=I_d$. Then,
\begin{equation}\label{eq:newS}
\W_1\bigl(\law(W), \mathcal N(0,I_d)\bigr)
\;\lesssim\;
\sum_{i=0}^{n-1}\E\left[
\|U_i\|\W_2^2\bigl(\law(W), \law(W \mid U_i)\bigr)
\right].
\end{equation}
\end{proposition}

The inequality \eqref{eq:newS} bounds the CLT rates in $\W_1$ by aggregating, over indices \(i\), a weighted measure of dependence quantifying how much the law of  \(W\) changes when conditioning on a single increment \(U_i\). Therefore, to bound the $\W_1$ CLT error, it suffices to control the right-hand side of
\eqref{eq:newS}, which is typically more tractable since it directly reflects the
underlying dependence structure. Below we outline the proof of
Proposition~\ref{prop:newS} by describing our construction of $\mathcal S$.

\noindent\textbf{Proof Outline of Proposition~\ref{prop:newS}.}
The goal of the construction is to embed the partial sum \(W\) of general dependent data into the abstract framework of Section~\ref{sec:raivc}. To do so, we construct a family \(\{V_\xi\}\) satisfying two key requirements. 
First, it must provide a representation of \(\E[f(W)W]\) in terms of \(\{V_\xi\}\), as required in Step~1 of Section~\ref{sec:raivc}. 
Second, the family must remain closed under the interpolation operation in Step~2 of Section~\ref{sec:raivc}: after interpolating between \(V_\xi\) and \(W\) and conditioning on the residual \(W-V_\xi\), the resulting conditional law should again belong to the same family, indexed by some \(\psi\). 
This closure property is stated precisely in~\eqref{eq:stablility}.

A natural starting point is to pin one coordinate \(U_i\) at a value \(x\) and consider the corresponding conditional sum \(\widetilde V_{i,x}\). This construction yields the desired representation of \(\E[f(W)W]\) by conditioning on \(U_i\). The main difficulty is that, after interpolating between \(\widetilde V_{i,x}\) and \(W\) and then conditioning on the residual \(W-\widetilde V_{i,x}\), the resulting conditional law need not be representable again in the form \(\widetilde V_{i,x}\). To overcome this, we enlarge the parameter space by introducing both an interpolation parameter and a residual parameter, thereby constructing a richer family \(\{V_\xi\}\) that is preserved by this procedure. This is the key idea behind the proof.

Once we construct this admissible object \(\mathcal S\), Lemma~\ref{lem:raivc} becomes applicable, and Proposition~\ref{prop:newS} follows by estimating the associated \(\beta\)-functionals. The details are deferred to Appendix~\ref{sec:proof-newS}.

\subsection{Proof Outline of Theorem \ref{thm:local-depend-1}}\label{sec:outline-local-depend-1}
\noindent\textbf{Case 1: $\delta = 1$.} When $\delta=1$, we leverage Proposition~\ref{prop:newS} to derive the $W_1$ CLT rate.  Specifically, in order to control each term $\E\left[\|U_i\|
\W_2^2\bigl(\law(W),\law(W\mid U_i)\bigr)\right]$ on the right-hand side of \eqref{eq:newS}, we construct an explicit
coupling between $W$ and $W\mid U_i$. In particular, for each $i\in I$ and $x\in\R^d$, we consider the following \emph{block-resampling} coupling:
\[
\widetilde U^{(i,x)}_i:=x,\qquad
\widetilde U^{(i,x)}_{\mathcal N[i]\setminus\{i\}}\sim \law\big(U_{\mathcal N[i]\setminus\{i\}}\mid U_i=x,\ U_{\mathcal N[i]^c}\big),
\qquad
\widetilde U^{(i,x)}_j:=U_j\ (j\notin\mathcal N[i]).
\]
The above coupling fixes the $i$th coordinate at $x$, keeps all coordinates outside its neighborhood unchanged, and only resamples its neighborhood $\mathcal N[i]\setminus\{i\}$ conditional on
$(U_i=x,U_{\mathcal N[i]^c})$.
Crucially, this mechanism \emph{localizes the effect of conditioning to the block $\mathcal N[i]$}. Under the uniform third-moment bound
$\sup_{0\le i\le n-1}\E\|X_i\|^{3}=\mathcal O(1)$ (i.e., $\delta=1$), one can show
that each term $\E\left[\|U_i\|
\W_2^2\bigl(\law(W),\law(W\mid U_i)\bigr)\right]$ is of order $\mathcal O(D^2 n^{-3/2})$. Summing over $i$ then yields the rate
$\mathcal O(D^{2} n^{-1/2})$ in Theorem~\ref{thm:local-depend-1}. The detailed
proof is deferred to Appendix~\ref{sec:proof-local-depend-1-optimal}.

\noindent\textbf{Case 2: $\delta \in (0,1)$.} The extension of Rai\v{c}'s framework to the regime $\delta\in(0,1)$ is not immediate.
 Under the condition $\sup_{0\le i\le n-1}\mathbb{E}\|X_i\|^{2+\delta}\in \mathcal{O}(1)$ for some $\delta\in(0,1)$, we instead invoke the classical regularity theory for the Stein equation \citep{gallouet2018regularity}. Specifically, for any $h\in C^{0,1},$ let $f_h$ denote the solution to the multivariate Stein equation
\[
\Delta f(x)-x\cdot\nabla f(x)=h(x)-\mathbb{E}h(Z), \qquad x\in\mathbb{R}^d.
\]
By setting $W=W_i+U_i$, and employing a second–order Taylor expansion, we have
\[
\mathbb{E}\big[h(W)-h(Z)\big]
=\frac{1}{n}\sum_{i=0}^{n-1} \mathbb{E}\left[\Delta f_h(W)-nU_i^{\top}
\nabla^{2}f_h\left(W_i+\theta U_i\right)U_i\right],
\]
where $\theta$ is uniformly
distributed in $[0,1]$ and independent of everything else. While $f_h\in C^{2,\delta}$ by \cite{gallouet2018regularity}, the locally dependent structure implies that $W_i$ is not independent of $U_i$. Consequently, the argument for independent data~\citep{gallouet2018regularity}  does not apply directly. 
To circumvent this, for each $i \in \{0, \dots, n-1\}$, we define $W_i' = \sum_{j \notin \mathcal N[i]}U_j$ and decompose the above equation into three terms
\begin{equation}\label{eq:threeterm}
\begin{aligned}
\mathbb{E}[h(W)-h(Z)]=&\frac{1}{n} \sum_{i=0}^{n-1} \mathbb{E}\Big[\Delta f_h(W_i'+U_i)-nU_i^T \nabla^2 f_h\big(W_i'+\theta U_i\big) U_i\Big]\\
&+\frac{1}{n} \sum_{i=0}^{n-1}\mathbb{E}\Big[\Delta f_h(W)-\Delta f_h(W_i'+U_i)\Big]\\
&+\sum_{i=0}^{n-1}\mathbb{E}\Big[U_i^T \nabla^2 f_h\big(W_i'+\theta U_i\big) U_i-U_i^T \nabla^2 f_h\big(W_i+\theta U_i\big) U_i\Big].
\end{aligned}
\end{equation}
We then bound the above three terms individually. The first term is analyzed following the approach in \cite{gallouet2018regularity}, as \(W_i'\) is independent of \(U_i\) by construction. The remaining two terms are controlled by leveraging the \(C^{2,\delta}\) regularity of \(f_h\). The detailed proof is deferred to
Appendix~\ref{sec:proof-local-depend-1-other}.

\subsection{Proof Outline of Theorem~\ref{thm:m-depend-p}}\label{sec:outline-m-depend-p} 
We establish the $\W_p$ CLT rates via a standard \emph{big--small block} decomposition~\citep{Meyn12_book}, which approximates the sum of $M$-dependent data by a sum of independent
big blocks, while tightly controlling the contribution of the small blocks (of length $M$) and
optimizing the big-block length.
  Fix an integer $\ell\ge M$ (to be chosen later) and set
\[
k:=\left\lfloor\frac{n}{\ell+M}\right\rfloor.
\]
Define consecutive \emph{big blocks} of length $\ell$
\[
B_j:=\{(j-1)(\ell+M),\dots,(j-1)(\ell+M)+\ell-1\},\qquad j=1,\dots,k,
\]
and \emph{small blocks} of length $M$
\[
G_j:=\{(j-1)(\ell+M)+\ell,\dots,j(\ell+M)-1\},\qquad j=1,\dots,k,
\]
together with the remainder\[
R:=\{k(\ell+M),\dots,n-1\},
\qquad |R|\le \ell+M.
\]
Let
\[
U_j:=\sum_{i\in B_j}X_i,\quad V_j:=\sum_{i\in G_j}X_i,\quad
A:=\sum_{j=1}^k U_j,\quad B:=\sum_{j=1}^k V_j,\quad \Delta:=B+\sum_{i\in R}X_i.
\]
Then $S_n=A+\Delta$ holds exactly. Moreover, since the gap between successive big
blocks satisfies $\min B_j-\max B_{j-1}=M+1>M$, the sequence $\{U_j\}_{j=1}^k$ is
independent; similarly, $\{V_j\}_{j=1}^k$ is independent. We next invoke the following $\W_p$ CLT rates for independent
data to control $\law(A)$.
\begin{lemma}[Theorem 6 of \cite{bonis2020stein}]\label{lem:independent}
Let $\{X_i\}_{i=0}^{n-1}$ be an independent sequence of of $\R^d$-valued random vectors with $\E[X_i]=0$ and $\frac{1}{n}\operatorname{Var}(S_n) = I_d.$ If $\sup_{0 \leq i \leq n-1}\E[\|X_i\|^{p+q}] \in \mathcal O(1) $ for some $p \geq 2$ and $q \in (0,2]$, we have
    \begin{align*}
       \W_p(\law(\tfrac{S_n}{\sqrt{n}}), \mathcal{N}(0, I_d)) \in \mathcal{O}(n^{-\frac{p+q-2}{2p}})  .
    \end{align*}
    \end{lemma}
Finally, applying the triangle inequality, together with bounds on $\E\|\Delta\|^{p}$ and on the
discrepancy between $\operatorname{Var}(A)/k$ and $\Sigma_n$, and optimizing the big-block length $\ell$,
yields Theorem~\ref{thm:m-depend-p}.

\noindent\textbf{Remark:} The big--small block technique attains independence by discarding correlations; however, applying the triangle
inequality requires a tradeoff between: (i) the reduced effective sample size (as larger
$\ell$ results in fewer independent big blocks) and (ii) the error induced by the discarded small blocks. Balancing these two
errors typically leads to a rate that is suboptimal compared with the best available rate in the
independent case (Lemma~\ref{lem:independent}). Nevertheless, as noted following Theorem~\ref{thm:m-depend-p}, our results significantly narrow the existing gap between $\W_1$ and $\W_p$ CLT theories. Achieving the optimal $\mathcal{O}(n^{-1/2})$ rate remains an important direction for future research.

\section{Central Limit Theorems for Markov chains }\label{sec:MC}
We consider a Markov chain $\{x_i\}_{i\ge 0}$ with transition kernel $P$ on a general state
space $(\mathcal X,\mathcal B)$, and assume it satisfies the following geometric ergodicity
condition.
\begin{assumption}[Geometric Ergodicity]\label{assumption:markovchain}
The Markov kernel $P$ is $\psi$-irreducible and aperiodic. There exist a petite set $C$,
constants $\lambda\in[0,1)$ and $L<\infty$, and a function $V:\mathcal X\to[1,\infty)$ such
that
\begin{align*}
   PV \leq \lambda V + L\mathbf{1}_{C}. 
\end{align*}
Moreover, the initial distribution $x_0\sim\mu$ satisfies $\E[V(x_0)]<\infty$.
\end{assumption}
Assumption~\ref{assumption:markovchain} is standard in the derivation of CLTs for Markov
chains; see, e.g., \cite{Meyn12_book,douc2008bounds, srikant2024rates}. In this section, we study additive functionals.
Let $h_i:\mathcal X\to\mathbb R^{d}$ satisfy $\E_\pi[h_i]=0$ and $\|h_i\|^{2}\le V$ for all
$i\ge 0$. For $n\ge 1$, define
\[
S_n:= \sum_{i=0}^{n-1} h_i(x_i),
\qquad
\Sigma_n:= \frac{1}{n}\operatorname{Var}(S_n).
\]
We consider the nondegenerate regime where the normalized covariance matrices are uniformly well-conditioned, i.e., $\lambda_{\min}(\Sigma_n) \geq \Omega(1)$. We have the following lemma; its proof is deferred to Appendix~\ref{sec:proof-Sigma-exists}.

\begin{lemma}
\label{lem:Sigma-exists1}
Under Assumption~\ref{assumption:markovchain}, the matrix $\Sigma_n$ is well defined for
every $n\ge 1$, and $\lambda_{\max}(\Sigma_n)=\mathcal O(1)$. Moreover, in the homogeneous
case where $h_i\equiv h$ for all $i\ge 0$, the limit $  \Sigma_\infty:= \lim_{n\to\infty}\frac{1}{n}\operatorname{Var}(S_n)$ exists, and $\|\Sigma_n-\Sigma_\infty\|=\mathcal O(1/n)$.
\end{lemma}

It therefore suffices to control the distance to $\mathcal N(0,\Sigma_n)$. Indeed, $\W_p(\mathcal N(0, \Sigma_n),\mathcal N(0, \Sigma_\infty)) \lesssim \|\Sigma_n-\Sigma_\infty\|^{1/2} \in \mathcal{O} (n^{-1/2})$ for any
$p\ge 1$, by Lemma~\ref{lem:gauss-gauss}. Since $\mathcal O(n^{-1/2})$ is already the optimal order
for Gaussian approximation, this covariance-mismatch term is of the same (or smaller) order
as the intrinsic CLT error. In particular, by the triangle inequality, we have
\[
  \W_p\bigl(\law(\tfrac{S_n}{\sqrt n}),\mathcal N(0,\Sigma_\infty)\bigr)
  \le
  \W_p\bigl(\law(\tfrac{S_n}{\sqrt n}),\mathcal N(0,\Sigma_n)\bigr)
  + \W_p\bigl(\mathcal N(0,\Sigma_n),\mathcal N(0,\Sigma_\infty)\bigr)\lesssim\W_p\bigl(\law(\tfrac{S_n}{\sqrt n}),\mathcal N(0,\Sigma_n)\bigr).
\]
We establish two theorems characterizing $\W_p$ CLT rates
for Markov chains with $p=1$ and $p\geq 2$.

\subsection{CLT Rates in Wasserstein-1 Distance}
We now state a theorem that quantifies the gap between
\(\law\big(\tfrac{S_n}{\sqrt{n}}\big)\) and the Gaussian  \(\mathcal N(0,\Sigma_n)\) in $\W_1$.
\begin{theorem}[Wasserstein--1 CLT Rates for Markov Chains]
\label{thm:MC-W1}
Under Assumption \ref{assumption:markovchain},  if there exists $\delta>1$ such that $\|h_i\|^{2+\delta} \leq V$ for all $i \geq 0$,
    \begin{align*}
       \W_{1}\bigl(\law(\tfrac{S_n}{\sqrt{n}}),\mathcal N(0,\Sigma_n)\bigr) \in \mathcal{O}(n^{-1/2}).
    \end{align*} 
\end{theorem}
We emphasize that the best rate available in the literature is
$\mathcal O(n^{-1/2}\log n)$ \citep{srikant2024rates,wu2025uncertainty}. The
logarithmic factor appears because existing analysis relies on Lemma~\ref{lem:regularity}; see the discussion in Section~\ref{sec:challenge}.
In contrast, Theorem~\ref{thm:MC-W1} establishes the first \emph{optimal}
$\mathcal O(n^{-1/2})$ $\W_1$--CLT rate for geometrically ergodic Markov chains,
thereby resolving an open problem highlighted in the recent work~\citep{srikant2024rates}. 

We remark that our assumptions on the Markov chain are strictly weaker than those in
\cite{wu2025uncertainty}, which require a positive spectral gap and uniform boundedness of the
test functions $h_i$. Furthermore, unlike \cite{srikant2024rates}, which is is restricted to the homogeneous
case $h_i \equiv h$, Theorem~\ref{thm:MC-W1} allows for time-varying test
functions $\{h_i\}_{i\ge 0}$. To achieve the optimal $\mathcal O(n^{-1/2})$, Theorem~\ref{thm:MC-W1} imposes a slightly stronger
moment condition on the chain: we assume a $V$-dominated $(2+\delta)$-moment for some
$\delta>1$, while \cite{srikant2024rates} works with $\delta=1$.
 Nevertheless, the approach of \cite{srikant2024rates} does not yield the
$\mathcal O(n^{-1/2})$ rate---even under stronger moment assumptions---due to
its reliance on Lemma~\ref{lem:regularity}.
It remains unclear whether a $V$-dominated third moment alone suffices to
achieve an $\mathcal O(n^{-1/2})$ $\W_1$--CLT rate. We leave this question for
future research.

Our proof attains the optimal rate by leveraging Proposition~\ref{prop:newS}, which reduces the
analysis to controlling $\W_2^2\bigl(\law(W),\law(W\mid U_i)\bigr)$.
 We accomplish this by first
showing that the time-reversal kernel $P^*$ also satisfies
Assumption~\ref{assumption:markovchain}, and then applying a combination of
forward and backward maximal couplings to localize the effect of conditioning
on $U_i$.

As an intermediate step, we prove that the regeneration time of the split chain
\citep{nummelin1978splitting} associated with a geometrically ergodic Markov chain
has a geometric tail, without assuming strong aperiodicity \citep{roberts1999bounds} or other restrictive
conditions \citep{rosenthal1995minorization}. The details are presented in Section~\ref{sec:splitchain} and may be
of independent interest.
Finally, the
proof of Theorem~\ref{thm:MC-W1} is outlined in Section~\ref{sec:outline-MC-W1},
with full details deferred to Appendix~\ref{sec:proof-MC-W1}.

\subsection{CLT Rates in Wasserstein-p ($p \geq 2$) Distance}
The  $\W_p$ ($p\ge 2$) CLT rates for Markov chains 
remain poorly understood. To narrow this gap, we establish the following
theorem, which provides the first multivariate $\W_p$ ($p\ge 2$) CLT rates for
geometrically ergodic Markov chains with time-homogeneous additive functionals ($h_i \equiv h$) under mild moment conditions.

\begin{theorem}[Wasserstein--p CLT Rates for Markov Chains]
\label{thm:MC-Wp}
Under Assumption \ref{assumption:markovchain}, if $h_i \equiv h$ and $\|h\|^{p+q} \leq V$ for some $p \geq 2$ and $q \in (0,2]$,
    \begin{align*}
       \W_{p}\bigl(\law(\tfrac{S_n}{\sqrt{n}}),\mathcal N(0,\Sigma_n)\bigr) \in \mathcal{O}(n^{-\frac{p+q-2}{2(2p+q-2)}}) .
    \end{align*}
\end{theorem}
These rates match the best bounds currently available
for $m$-dependent data (Theorem~\ref{thm:m-depend-p}) and resolve another open
problem highlighted in the conclusion of \cite{srikant2024rates}. As discussed
following Theorem~\ref{thm:m-depend-p}, these rates may not be optimal; however,
they substantially narrow the gap between the $\W_1$ theory and the $\W_p$
($p\ge 2$) regime. 

From a technical perspective, roughly speaking, we proceed as follows. Using the
split-chain construction in Section~\ref{sec:splitchain}, we first apply a
martingale decomposition to the partial sum, and then perform a regeneration
decomposition. This yields a block representation in which successive
regenerative blocks form a $1$-dependent sequence (a special case of
$M$-dependence with $M=1$). This structure allows us to invoke our $\W_p$
($p\ge 2$) CLT rates for $M$-dependent sequences (Theorem~\ref{thm:m-depend-p}).
The proof is outlined in Section~\ref{sec:outline-MC-Wp} and detailed in
Appendix~\ref{sec:proof-MC-Wp}.

In terms of applications, establishing convergence rates for the Markov chain
CLT is crucial for assessing asymptotic efficiency, deriving finite-sample error
bounds, and enabling statistical inference for parameters of interest in
stochastic algorithms and dynamical systems
\citep{srikant2024rates,samsonov2024gaussian,wu2025uncertainty}. Since prior work
on Markov chain CLT rates is largely restricted to $\W_1$ and often yields suboptimal
bounds, our Theorems~\ref{thm:MC-W1} and \ref{thm:MC-Wp} make it possible to
derive sharper guarantees under $\W_1$ and to develop results under more general $\W_p$
metrics for $p\ge 2$.

\section{Technical Overview of CLT Rates for Markov Chains}\label{sec:outline-berry}
This section is organized as follows. Section~\ref{sec:splitchain} reviews the
basic elements of Nummelin splitting \citep{nummelin1978splitting,Meyn12_book} and
presents Lemma~\ref{lem:geom_tail}, which establishes a  geometric tail bound for
the regeneration time. Sections~\ref{sec:outline-MC-W1} and~\ref{sec:outline-MC-Wp}
then outline the proofs of Theorems~\ref{thm:MC-W1} and~\ref{thm:MC-Wp},
respectively.
\subsection{Nummelin's splitting}\label{sec:splitchain}

We collect the standard facts needed for the split-chain construction and the
resulting regeneration decomposition. The only nonstandard result in this
subsection is the geometric-tail bound in Lemma~\ref{lem:geom_tail}.
\begin{lemma}\label{lem:small-set-return}
Under Assumption~\ref{assumption:markovchain}, there exist an accessible small set \(\bar C\subseteq \mathcal X\) and a constant \(\kappa>1\) such that
\begin{equation}\label{eq:return}
    M(a):=\sup_{x\in \bar C}\E_x\big[a^{\sigma_{\bar C}}\big]<\infty,
    \qquad \forall a\in(1,\kappa],
\end{equation}
where $\sigma_{\bar C}:=\inf\{n\ge 1:x_n\in \bar C\}.$ Moreover, there exist \(m\in\mathbb N\), \(\beta\in(0,1)\), and a probability measure \(\nu\) on \(\mathcal X\), supported on \(\bar C\), such that \(\nu(\bar C)=1\), \(\nu(\bar C^c)=0\) and 
\begin{equation}\label{eq:M}
   P^m(x,\cdot)\ge \beta \mathbf{1}_{\bar C}(x)\nu(\cdot),
   \qquad x\in\mathcal X.
\end{equation}
The existence of \(\bar C\) and the exponential return bound~\eqref{eq:return} is a standard consequence of geometric ergodicity; see, for example, Theorem~15.0.1 of \cite{Meyn12_book} and Theorem~15.1.5 of \cite{douc2018markov}. The minorization condition~\eqref{eq:M} is the standard small-set minorization; see Proposition~5.2.4 of \cite{Meyn12_book}.
\end{lemma}
Lemma~\ref{lem:small-set-return} provides exactly the two ingredients needed for
Nummelin's splitting: an accessible small set with exponentially decaying
return times, and a minorization for the \(m\)-skeleton. In particular, given the minorization condition~\eqref{eq:M}, we apply Nummelin's splitting \citep{nummelin1978splitting} to the \(m\)-skeleton of the chain, which yields an augmented process \begin{equation}\label{eq:splitchain} \Phi=\{(x_n,y_n)\}_{n\ge 0} \qquad\text{on}\qquad \widetilde{\mathcal X}:=\mathcal X\times\{0,1\}, \end{equation} where \(x_n\) is the original state and \(y_n\) is an auxiliary level variable. The first coordinate \(\{x_n\}_{n\ge 0}\) has exactly the same law as the original Markov chain. Full details of the construction may be found in \cite{Meyn12_book,lemanczyk2021general}. Two key properties of~\eqref{eq:splitchain} are: (i) At each block boundary $km$, the level is determined from the current state only. Furthermore, if $x_{km}\in\bar C$, draw $y_{km}\sim\mathrm{Bernoulli}(\beta)$; (ii) If $y_{km}=1$, we regenerate $(x_{(k+1)m},y_{(k+1)m})$ by \[ \mathbb P\left(x_{(k+1)m}\in \d x \middle| x_{km}\in\bar C, y_{km}=1\right)=\nu(\d x), \] so $(\bar C,1)$ is an atom for the $m$-skeleton $\{(x_{km}, y_{km})\}_{k \geq 0}$.

To define regeneration cycles, let
\[
r_1:=\inf\{k\ge0:(x_{km}, y_{km})\in(\bar C,1)\},
\qquad
r_i:=\inf\{k>r_{i-1}:(x_{km}, y_{km})\in(\bar C,1)\},
\qquad i\ge2,
\]
be the successive visits of the split chain to the atom. Following the
convention that a new cycle begins one block after an atom visit, we set
\[
\tau_0:=0,
\qquad
\tau_i:=r_i+1,
\qquad
L_i:=\tau_i-\tau_{i-1},
\qquad i\ge1.
\]
The cycle lengths \(\{L_i\}_{i\ge1}\) satisfy the following lemma, whose proof is
deferred to Section~\ref{sec:geom_tail}.

\begin{lemma}\label{lem:geom_tail}
The cycle lengths \(\{L_i\}_{i\ge 1}\) are independent, \(\{L_i\}_{i\ge 2}\) are i.i.d., and there exist constants \(b>0\) and \(\rho\in(0,1)\) such that
\[
\mathbb P(L_i>\ell)\le b\rho^\ell,
\qquad \forall i\ge 1,\ \ell\ge 0.
\]
\end{lemma}

Lemma~\ref{lem:geom_tail} is the main result of this subsection. It will be used in two places: first, to prove geometric ergodicity of the time-reversed chain \(P^\ast\) in Section~\ref{sec:outline-MC-W1}; and second, to control the moments arising from the regeneration decomposition in Section~\ref{sec:outline-MC-Wp}. Related bounds were obtained in \cite{roberts1999bounds} in the strongly aperiodic case \(m=1\). The general-\(m\) setting was studied in \cite{rosenthal1995minorization}, but under the stronger assumption that a sublevel set of the form \(\{x \in \mathcal X : V(x)\le R\}\) itself satisfies~\eqref{eq:return}, a property that does not follow directly from the standard split-chain construction. A technical issue in \cite{rosenthal1995minorization} was later clarified in \cite{roberts1999bounds}. Our Lemma~\ref{lem:geom_tail} recovers the \(m=1\) bounds of \cite{roberts1999bounds} and extends them to general \(m\) under the standard geometric-drift framework.

\subsection{Theorem \ref{thm:MC-W1} ($p = 1$): Proposition \ref{prop:newS} + Forward and Backward Maximal Coupling}\label{sec:outline-MC-W1}
To streamline the analysis, we first reduce to the stationary initialization $x_0\sim\pi$;
as shown in Appendix~\ref{sec:stationary}, this initialization mismatch contributes at most
$\mathcal O(n^{-1/2})$ to the final CLT rate. 

Inspired by Proposition \ref{prop:newS}, we 
 consider the normalization \[
U_i:=\frac{h_i(x_i)}{\sqrt n},\quad \forall i \in \{0,\dots, n-1\} \quad \text{and} \quad 
W:=\sum_{i=0}^{n-1}U_i.
\]
Our goal is to control
$\sum_{i=0}^{n-1}\E\left[\|U_i\|\W_2^2\bigl(\law(W),\law(W\mid U_i)\bigr)\right]$.
Since \(U_i\) is measurable with respect to \(x_i\), and since the map
\(\mu \mapsto \W_2^2(\nu,\mu)\) is convex for fixed \(\nu\)
\citep{matthes2009family}, 
\begin{align*}
\sum_{i=0}^{n-1}\E\left[\|U_i\|\W_2^2\bigl(\law(W),\law(W\mid U_i)\bigr)\right]
\le
\sum_{i=0}^{n-1}\E\left[\|U_i\|\W_2^2\bigl(\law(W),\law(W\mid x_i)\bigr)\right].
\end{align*}
Fix \(i\), and for each \(x \in \mathcal X\), write 
\[
  \nu_x:=\law(W\mid x_i=x).
\]
Under stationarity, \(x_i \sim \pi\), and hence $\law(W)=\int_{\X}\nu_y\,\pi(\d y).$ Therefore, by the convexity of \(\mu \mapsto \W_2^2(\mu,\nu_x)\),
\[
  \W_2^2\bigl(\law(W),\nu_x\bigr)
  \le
  \int_{\X}\W_2^2(\nu_y,\nu_x)\,\pi(\d y).
\]
Thus, the main part of the argument reduces to bounding
\(\W_2^2(\nu_y,\nu_x)\) for fixed \(x,y \in \mathcal X\). To this end, we employ a forward--backward maximal coupling to
\emph{localize the effect of conditioning on \(x_i\)}. We work with a
bi-infinite stationary version \(\{x_t\}_{t\in\mathbb Z}\) of the chain,
with \(x_0 \sim \pi\). Fix \(i \in \{0,\dots,n-1\}\) and \(x,y \in \mathcal X\).
We construct a coupling of two bi-infinite trajectories
\((X_t^x)_{t\in\mathbb Z}\) and \((X_t^y)_{t\in\mathbb Z}\) such that
\[
\law\bigl((X_t^x)_{t\in\mathbb Z}\bigr)
=
\law\bigl((x_t)_{t\in\mathbb Z} \mid x_i = x\bigr),
\qquad
\law\bigl((X_t^y)_{t\in\mathbb Z}\bigr)
=
\law\bigl((x_t)_{t\in\mathbb Z} \mid x_i = y\bigr),
\]
and such that there exist meeting times \(T_i^+(x,y)\) and \(T_i^-(x,y)\) for
which
\[
X_t^x = X_t^y
\qquad\text{whenever}\qquad
t \ge i + T_i^+(x,y)
\quad\text{or}\quad
t \le i - T_i^-(x,y).
\]
This coupling localizes the effect of conditioning on \(x_i\) to a random
neighborhood of \(i\).

To construct the coupling above---and in particular the backward maximal
coupling---we first need a rigorous definition of the backward
(time-reversed) process, together with its basic properties. These are provided
by the following lemma, whose proof relies on the split-chain construction and
Lemma~\ref{lem:geom_tail}, and is deferred to
Appendix~\ref{sec:proof-backward}.
\begin{lemma}
\label{lem:backward}
Under Assumption~\ref{assumption:markovchain}, the time-reversed kernel $P^\ast$
is $\psi$-irreducible, aperiodic, and geometrically ergodic.
\end{lemma}
Lemma~\ref{lem:backward}, together with
Assumption~\ref{assumption:markovchain}, implies that both the forward chain
and the time-reversed chain are geometrically ergodic. The next lemma shows
that, in each direction, the coupled chains meet after a time with a geometric
tail. Its proof is deferred to Appendix~\ref{sec:proof-forward-meeting}.
\begin{lemma}
\label{lem:forward-meeting}
Under Assumption \ref{assumption:markovchain}, there exists $\rho \in(0,1)$ and,
for each $x,y\in\mathcal X$, a coupling $(X_t^x,X_t^y)_{t\ge0}$ of
two copies of the chain started at $X_0^x=x$, $X_0^y=y$, such that the
meeting time $T^+(x,y):=\inf\{t\ge0: X_t^x=X_t^y\}$ satisfies
\[
  \P\bigl(T^+(x,y)>k\bigr)
  \lesssim
  \bigl(V(x)+V(y)\bigr)\rho^k,
  \qquad k\ge0.
\]
Consequently, for every $p\ge1$ there exists $C_{p}<\infty$ with $  \E\bigl[T^+(x,y)^p\bigr]
  \le C_{p}\bigl(V(x)+V(y)\bigr).$
\end{lemma}
By \cite[Theorem~15.0.1]{Meyn12_book}, the time-reversed kernel \(P^\ast\)
admits a drift function \(V'\). Moreover, for any \(\alpha \in (0,1)\), the
functions \(V^\alpha\) and \((V')^\alpha\) remain valid Lyapunov functions for
\(P\) and \(P^\ast\), respectively. Hence,
Lemma~\ref{lem:forward-meeting} applies with \(V^\alpha\) and
\((V')^\alpha\) in place of \(V\). These meeting-time bounds are combined with the moment condition \(\|h_i\|^{2+\delta}\le V\) via H\"older's inequality to bound \(\W_2^2(\nu_y,\nu_x)\). Particularly with \(\delta>1\), we can choose \(\alpha = (\delta-1)/\delta\). Applying H\"older's inequality again yields the desired
\(\mathcal O(n^{-1/2})\) rate.


\subsection{Theorem \ref{thm:MC-Wp} ($p\geq 2$): Regeneration Decomposition +  1-Dependent CLT Rates}\label{sec:outline-MC-Wp}
In this section, we assume the moment domination condition $\|h\|^{r}\le V$ for
some $r=p+q>2$, where $p$ and $q$ are the parameters as in
Theorem~\ref{thm:MC-Wp}. The proof outline follows three key steps: 

\medskip
\noindent\textbf{Step 1: Regeneration Decomposition.} Recall the regeneration indices $0=\tau_0<\tau_1<\tau_2<\cdots$ of the $m$-skeleton
of the split chain. Define the corresponding regeneration times in the original
time scale and the associated cycle lengths by
\begin{equation*}
T_i:=m\tau_i,\qquad \Lambda_i:=T_i-T_{i-1}=mL_i,\qquad i\ge 1.
\end{equation*}
Define
\begin{equation}\label{eq:def-gt-tv}
g:=\sum_{k=0}^\infty P^k h.
\end{equation}
Since $\|h\|^r\le V$ and the chain is geometrically ergodic under
Assumption~\ref{assumption:markovchain}, the series in \eqref{eq:def-gt-tv}
converges absolutely. Moreover, $\|g\|\lesssim V^{1/r}$ for all $t\ge 0$. Let $\mathcal F_t:=\sigma(x_0,\dots,x_t)$ be the canonical filtration and define $\xi_{t+1}:=g(x_{t+1})-Pg(x_t)$ for $t \geq 0$ and $M_n:=\sum_{t=0}^{n-1}\xi_{t+1}$ for $n\ge 0.$ Then $(M_n,\mathcal F_n)_{n\ge 0}$ is an $\R^d$-valued martingale, and the partial
sum admits the decomposition
\begin{equation}\label{eq:Sn-mart-tv}
S_n
=
g(x_0)-g(x_n)+M_n,
\qquad \forall n\ge 1.
\end{equation}
Since $\|g\|\lesssim V^{1/r}$, Assumption~\ref{assumption:markovchain} implies $\|S_n-M_n\|_{L_p} \in  \mathcal O(1)$. Next
define the \emph{cycle martingale increments}
\begin{equation*}
\widetilde M_i
:=M_{T_i}-M_{T_{i-1}}
=\sum_{t=T_{i-1}}^{T_i-1}\xi_{t+1},\qquad i\ge 1.
\end{equation*}
Let $K_n:=\min\{i\ge 1:\ T_i>n\}+1$ denote the index immediately following the first regeneration time strictly
after $n$. Since $T_{K_n-2}\le n<T_{K_n-1}<T_{K_n}$, we can decompose the
martingale at time $n$ as
\[
M_n
= \sum_{i=1}^{K_n}\widetilde M_i - \sum_{t=n}^{T_{K_n}-1}\xi_{t+1}
=: \widetilde S_n - R_n,
\]
where $\widetilde S_n:=\sum_{i=1}^{K_n}\widetilde M_i$ and
$R_n:=\sum_{t=n}^{T_{K_n}-1}\xi_{t+1}$. By Assumption~\ref{assumption:markovchain} and the geometric tail bounds for the
cycle lengths $\{L_i\}_{i\ge 1}$ from Lemma~\ref{lem:geom_tail}, the sequence
$\{\widetilde M_i\}_{i\ge 1}$ and the remainder term $R_n$ satisfy the moment
bounds stated below; the proof is deferred to Section~\ref{sec:iid_Mr}.
\begin{lemma}
\label{lem:iid_Mr}
$\{\widetilde M_i\}_{i\ge1}$ are 1-dependent. Moreover, if $\|h\|^r \leq V$ for some $r > 2$, we have $\E[\widetilde M_i]=0$ for any $i \geq 1$, $\sup_{i \geq 1}\mathbb E\|\widetilde M_i\|^{r}<\infty$ and $ \sup_{n \geq 1}\E[\|R_n\|^{r}] <\infty.$
\end{lemma}
We emphasize that the sequence $\{\widetilde M_i\}_{i\ge 1}$ is $1$-dependent but
not independent. Indeed, by construction, $\widetilde M_i$ depends on
the interval $[T_{i-1},T_i]$ and not merely on the
block $[T_{i-1},T_i-1]$. By Lemma~\ref{lem:iid_Mr}, we have
\begin{equation}\label{eq:triangle-reminder}
\|S_n-\widetilde S_n\|_{L_p}=\mathcal O(1),
\qquad
\W_{p}\bigl(\law(\tfrac{S_n}{\sqrt n}),
\law(\tfrac{\widetilde S_n}{\sqrt n})\bigr)
=\mathcal O(n^{-1/2}).
\end{equation}
\textbf{Remark:} Instead of employing a martingale decomposition as in \cite{srikant2024rates,wu2025uncertainty} and appealing to martingale CLT rates on $M_n$, we leverage geometric ergodicity and Nummelin’s splitting to partition the sum into $1$-dependent regeneration blocks. This yields general $\W_p$ bounds for all $p\ge2$. 

\medskip
\noindent\textbf{Step 2: Determinizing the Number of Blocks.}
The remaining obstacle is that \(\widetilde S_n\) contains a \emph{random} number \(K_n\) of regeneration blocks, whereas Theorem~\ref{thm:m-depend-p} is stated for sums with a deterministic number of terms. We therefore introduce the deterministic approximation
\[
k_n:=\Bigl\lfloor \frac{n}{\E[\Lambda_2]}\Bigr\rfloor,
\qquad
\bar S_n:=\sum_{i=1}^{k_n}\widetilde M_i.
\]
At first sight, one might try to condition on \(K_n\). However, this approach presents an additional challenge, since the event \(\{K_n=k\}\) depends on the entire regeneration pattern up to time \(n\), and conditioning on it destroys the simple dependence structure of the cycle increments. To circumvent this issue, we compare \(\widetilde S_n\) and \(\bar S_n\) via an odd--even decomposition of the regeneration blocks. Each resulting subsequence is then a martingale difference sequence, which allows us to invoke martingale inequalities, specifically the Rosenthal--Burkholder inequality. We defer the technical calculations to Appendix~\ref{sec:proof-covariance-mismatch}; the resulting estimate is
\begin{equation*}
\|\widetilde  S_n - \bar S_n\|_{L_p} \in \mathcal{O}(n^{1/4}),\qquad \W_{p}\big(\law(\tfrac{\widetilde  S_n}{\sqrt{n}}),\law(\tfrac{\bar  S_n}{\sqrt{n}})\big) \in \mathcal{O}(n^{-1/4}).
\end{equation*}
\noindent\textbf{Remark:}  
We note that the resulting error of replacing $K_n$ by $k_n$ is of order $\mathcal O(n^{-1/4})$, which is no worse
than the best available $\W_p$ ($p\ge 2$) CLT rate for $1$-dependent sequences in
Theorem~\ref{thm:m-depend-p}.

\medskip
\noindent\textbf{Step 3: Applying the 1-dependent CLT Rates and Matching the Covariance.}
In this step, we control the covariance discrepancy between the original sum and
its deterministic cycle-block approximation. Specifically, we show (see
Section~\ref{sec:proof-covariance-mismatch} for details) that
\begin{equation}\label{eq:cov-mismatch}
\big\|\Sigma_n-\frac{1}{n}\operatorname{Var}(\bar S_n)\big\|
\in \mathcal O(n^{-1/2}),
\end{equation}
By Lemma~\ref{lem:gauss-gauss}, this covariance mismatch further implies
\begin{equation*}
\W_p\big(\mathcal N(0,\Sigma_n),\mathcal N(0,\frac{1}{n}\operatorname{Var}(\bar S_n))\big)
=O(n^{-1/4}).
\end{equation*}
Next, by Weyl's inequality, Lemma~\ref{lem:Sigma-exists1}, and the covariance
mismatch bound~\eqref{eq:cov-mismatch}, we obtain $\lambda_{\min}(\frac{1}{n}\operatorname{Var}(\bar S_n)) \in \Omega(1)$ and $\lambda_{\max}(\frac{1}{n}\operatorname{Var}(\bar S_n)) \in \mathcal O(1).$ Moreover, since $\E[\widetilde M_i]=0$ for all $i\ge 1$, we apply the
$1$-dependent $\W_p$--CLT rate in Theorem~\ref{thm:m-depend-p} to obtain
\begin{align*}
\W_{p}\big(\law(\tfrac{\bar  S_n}{\sqrt{n}}),\mathcal N(0, \frac{1}{n}\operatorname{Var}(\bar S_n))\big) \in \mathcal{O}(n^{-\frac{p+q-2}{2(2p+q-2)}}).
\end{align*}
Finally, using the triangle inequality and combining the bounds from Steps~1--2 give
\begin{align*}
\W_{p}\big(\law(\tfrac{S_n}{\sqrt{n}}),\mathcal N(0, \Sigma_n)\big) \leq& \W_{p}\big(\law(\tfrac{S_n}{\sqrt{n}}),\law(\tfrac{\widetilde S_n}{\sqrt{n}})\big) + \W_{p}\big(\law(\tfrac{\widetilde S_n}{\sqrt{n}}),\law(\tfrac{\bar S_n}{\sqrt{n}})\big) \\
&+ \W_{p}\big(\law(\tfrac{\bar S_n}{\sqrt{n}}),\mathcal N(0, \frac{1}{n}\operatorname{Var}(\bar S_n))\big) + \W_{p}\big(\mathcal N(0, \frac{1}{n}\operatorname{Var}(\bar S_n)),\mathcal N(0, \Sigma_n)\big)\\
\lesssim& n^{-1/2} + n^{-1/4} + n^{-\frac{p+q-2}{2(2p+q-2)}} + n^{-1/4} \in \mathcal O(n^{-\frac{p+q-2}{2(2p+q-2)}}).
\end{align*}
\textbf{Remark:} As discussed following Theorem~\ref{thm:MC-Wp}, the rate
$\mathcal O(n^{-\frac{p+q-2}{2(2p+q-2)}})$ might not be
optimal. Indeed, the decomposition above reveals three terms that prevent the
overall bound from reaching $\mathcal O(n^{-1/2})$. Among them, the
dominant contribution is
$\mathcal O(n^{-\frac{p+q-2}{2(2p+q-2)}})$, which comes from our
(best available) $\W_p$ CLT rate for $M$-dependent sequences. This term could
potentially be improved by extending the exchangeable-pair framework of
\cite{bonis2020stein} to dependent data. 

The remaining two $\mathcal O(n^{-1/4})$ terms both arise from the error incurred
when approximating the random cycle count $K_n$ by a deterministic index $k_n$. Recall that the
original cycle-block sum is $\widetilde S_n:=\sum_{i=1}^{K_n}\widetilde M_i,$ where the increments $\{\widetilde M_i\}_{i\ge1}$ are $1$-dependent. 
A natural alternative to deterministic approximation of $K_n$
is to condition on $K_n$, apply the $1$-dependent CLT rate to the
conditional sum $\sum_{i=1}^{K_n}\widetilde M_i$, and then average over $K_n$.
However, this approach is hindered by the fact that the summands themselves depend on $K_n$. Specifically,
each $\widetilde M_i$ is measurable with respect to the trajectory up to time
$T_i$, and the event $\{K_n=k\}$ imposes global constraints on the regeneration
times. Consequently, once conditioning on
$\{K_n=k\}$ the increments $\{\widetilde M_i\}_{i=1}^{k}$ lose the tractable
$1$-dependence structure. 

To obtain the optimal Gaussian approximation rate for $S_n$, it may be
necessary to analyze $\widetilde S_n$ directly, rather than reducing to a fixed-length
sum. This perspective is inspired by renewal theory. In particular, sums of the
form
\begin{equation*}
V_s:=\sum_{i=1}^{N_s} Y_i,
\qquad
N_s:=\max\Bigl\{n\ge 0:\ \sum_{i=1}^n T_i\le s\Bigr\},
\end{equation*}
are known as \emph{compound renewal processes}~\citep{cox1962renewal,cox2017theory,malinovskii2021level}. Our $\widetilde S_n$ is
closely related to this framework. Existing CLT-rate results for compound renewal processes are largely confined to
the univariate setting and focus on metrics such as the Kolmogorov distance~\citep{cox1962renewal,cox2017theory,malinovskii2021level}.
Developing
multivariate CLT rates for compound renewal processes---particularly in
transportation metrics such as $\W_p$---is an interesting direction for future
work.

\section{Conclusion}\label{sec:conclusion}
In this work, we establish $\W_p$ ($p\ge 1$) CLT convergence rates---optimal in
some regimes---for multivariate locally dependent sequences and geometric ergodic Markov chains
under mild moment condition.
 We also
discuss an application of our optimal $\W_1$ rates to multivariate
$U$-statistics. These results may be broadly useful for providing quantitative
uncertainty guarantees across a wide range of problems in machine learning and operations research.

From a technical perspective, we extend Rai\v{c}'s framework~\citep{raivc2018multivariate} and obtain a tractable bound for the
$\W_1$ Gaussian approximation error under dependence. We then leverage this bound to establish optimal $\W_1$ rates for both locally dependent data
and geometric ergodic Markov chains. We expect this bound to be useful more broadly for
deriving rates under other dependent structures. In addition, we show that the
regeneration time of the split chain associated with a geometrically ergodic
Markov chain has a geometric tail, without assuming strong aperiodicity or other
restrictive conditions commonly imposed in the literature; this result may be of
independent interest.

There are several natural directions for future work:
\begin{itemize}
    \item For locally dependent data, it remains open whether one can achieve the
\(\mathcal O(n^{-\frac{p+q-2}{2p}})\) \(\W_p\)-CLT rate for \(p\ge 2\),
matching the best known rate in the independent case \citep{bonis2020stein}.
The argument in \cite{bonis2020stein} relies on an exchangeable-pair
construction tailored to independence, so a natural direction is to extend
this technique to dependent settings.
\item In Theorem~\ref{thm:MC-W1}, we require a $V$-dominated $(2+\delta)$-moment
condition with $\delta>1$ in order to obtain the optimal
$\mathcal O(n^{-1/2})$ $\W_1$ CLT rate for Markov chains. 
Determining whether a $V$-dominated third
moment alone would suffice to achieve this optimal rate remains an important challenge.
\item As discussed in Section~\ref{sec:outline-MC-Wp}, one potential avenue for improving $\W_p$ ($p\ge 2$) CLT rates for Markov chain is to study $\W_p$ CLT rates for
multivariate compound renewal processes. This would likely require extending existing renewal
theory---which are largely confined to the univariate setting and to
metrics such as the Kolmogorov distance---to multivariate transportation metrics
(e.g., \cite{cox1962renewal,cox2017theory,malinovskii2021level}).
\item Studying CLT rates for Markov chains under weaker assumptions, such as
subgeometric ergodicity, is an important direction for future work. We expect
that some of the techniques developed here, especially the split-chain
framework, can be extended to this setting. Related work
\citep{douc2008bounds} established CLT rates for subgeometrically ergodic
Markov chains via splitting, but only in the univariate setting, for the
Kolmogorov distance, and under additional assumptions such as strong
aperiodicity. It would be interesting to use the more general split-chain
theory developed in this paper to obtain analogous multivariate results.
\item This paper focuses mainly on the dependence on the sample size \(n\). Although
one could in principle trace the proof more carefully to make the dimension
dependence explicit, the resulting bounds would likely still be suboptimal.
Obtaining sharp dimension-dependent rates in \(d\) is therefore a separate and
worthwhile problem, especially for high-dimensional applications. Moreover, for
\(p\ge 2\), prior work suggests that optimal dimension dependence typically
requires additional structural assumptions, such as nonlattice or
Poincar\'e-type conditions \citep{bonis2020stein,bonis2024improved}. Studying
sharp dimension dependence for structured dependent data under such assumptions
is an important direction for future research.
\end{itemize}
%
%
%

\ACKNOWLEDGMENT{Q. Xie and Y. Zhang are supported in part by National Science Foundation (NSF) grants CNS-1955997 and
EPCN-2339794 and EPCN-2432546.}


\bibliographystyle{informs2014} 

\bibliography{main} 




\ECSwitch
\ECHead{Online Supplement for ``Wasserstein-p Central Limit Theorem Rates:
From Local Dependence to Markov Chains''}

\section{Proof of Theorem~\ref{thm:local-depend-1}}\label{sec:local-depend-1}
In this section we complete the proof of Theorem~\ref{thm:local-depend-1} by providing only the additional details
omitted from the proof outline in Sections~\ref{sec:ourS} and \ref{sec:outline-local-depend-1}.
\subsection{Proof of Proposition \ref{prop:newS}}\label{sec:proof-newS}
Let \[ \Xi := I\times\mathbb{R}^d\times[0,1]\times\mathbb{R}^d, \qquad I:=\{0,\dots,n-1\}. \] For each \((i,x)\in I\times\mathbb{R}^d\), define a conditioned version of the sample by \[ \bigl(\widetilde U^{(i,x)}_0,\dots,\widetilde U^{(i,x)}_{n-1}\bigr) \stackrel{d}{=} (U_0,\dots,U_{n-1}) \mid U_i=x, \] so that \(\widetilde U^{(i,x)}_i=x\) almost surely. Set \[ \widetilde V_{i,x}:=\sum_{j=0}^{n-1}\widetilde U^{(i,x)}_j, \qquad Y_{i,x}:=W-\widetilde V_{i,x}. \] Thus, \(\widetilde V_{i,x}\) is the total sum obtained after pinning \(U_i\) at \(x\), while \(Y_{i,x}\) records the discrepancy between this conditional sum and the original sum \(W\). For \(\xi=(i,x,t,y)\in\Xi\), define \(V_\xi\) by \[ V_{(i,x,0,y)}\stackrel{d}{=}\widetilde V_{i,x} \qquad\text{(independent of \(y\))}, \] and, for \(t\in(0,1]\), \[ V_{(i,x,t,y)}\stackrel{d}{=}(1-t)\widetilde V_{i,x}+tW \,\big|\, Y_{i,x}=y. \] Hence the parameter \(t\) interpolates between the conditional sum \(\widetilde V_{i,x}\) and the full sum \(W\), while \(y\) stores the residual information needed for conditioning. Now define an \(\mathbb{R}^d\)-valued measure \(\mu\) on \(\Xi\) by \[ \mu\bigl(\{i\}\times B\times\{0\}\times\{\mathbf{0}\}\bigr) := \mathbb{E}\bigl[U_i\mathbf{1}_{\{U_i\in B\}}\bigr], \qquad B\in\mathcal{B}(\mathbb{R}^d), \] and set \(\mu=0\) outside \(\{t=0,\ y=\mathbf{0}\}\). Then, for every bounded measurable \(f:\mathbb{R}^d\to\mathbb{R}\), \[ \mathbb{E}[f(W)W] = \sum_{i\in I}\mathbb{E}[f(W)U_i] = \sum_{i\in I}\mathbb{E}\bigl[\mathbb{E}[f(W)\mid U_i]\,U_i\bigr] = \int_{\Xi}\mathbb{E}_{\xi}[f(V_\xi)]\,\mu(d\xi), \] where \(\mathbb{E}_{\xi}\) denotes expectation with respect to the law of \(V_\xi\). This is exactly the representation required in Step~1 of Section~\ref{sec:raivc}. 

The key observation is that the family \(\{V_\xi\}_{\xi\in\Xi}\) is closed under further convex combinations with \(W\). Given \(\xi=(i,x,t,y)\in\Xi\), \(s\in[0,1]\), and \(z\in\mathbb{R}^d\), define \[ \psi(\xi,s,z) := \Bigl(i,\ x,\ t+s(1-t),\ \frac{z}{1-t}\mathbf{1}_{\{t<1\}}\Bigr). \] This updates the interpolation level from \(t\) to \(t':=t+s(1-t)\) and rescales the residual accordingly. Indeed, if we write \(Z_\xi:=W-V_\xi\), then \[ Z_\xi\stackrel{d}{=}(1-t)Y_{i,x}. \] Hence, when \(t<1\), conditioning on \(Z_\xi=z\) is equivalent to conditioning on \(Y_{i,x}=z/(1-t)\); when \(t=1\), we simply have \(V_\xi=W\), so the claim is immediate. Therefore, \[ \law\bigl((1-s)V_\xi+sW \,\big|\, Z_\xi=z\bigr) = \law\bigl(V_{\psi(\xi,s,z)}\bigr) \] for all \(s\in[0,1]\) and for \(\law(Z_\xi)\)-a.e.\ \(z\in\mathbb{R}^d\). This verifies the structural requirement in Step~2 of Section~\ref{sec:raivc}, with the corresponding kernels \(\{\nu_\xi\}_{\xi\in\Xi}\) induced by the map \(\psi\). 

Then, the above construction yields the admissible object \(\mathcal{S}\) required by Lemma~\ref{lem:raivc}. Proposition~\ref{prop:newS} then follows by bounding the associated \(\beta\)-functionals for this object. By definition,
\begin{align*}
    \beta_1^{(\xi)}&=|\nu_\xi|(\Xi) \leq \E[\|W-V_\xi\|] = \begin{cases}
     \E[\|W-V_\xi\|]&, \text{ if } t = 0,\\
     (1-t)\|y\|&, \text{otherwise.}
    \end{cases}
\end{align*}
\begin{align*}
\beta_{2} \leq  \int_{\Xi}   \E[\|W-V_\xi\|] |\mu|(\d \xi) &= \sum_{i=0}^{n-1}\E[\|\widetilde U_i^{(i)}\|\E[\|W-V_{i,\widetilde U_i^{(i)},0,0}\|]]\\
&\lesssim \sum_{i=0}^{n-1}(\E[\|\widetilde U_i\|\E[\|W\|]]+\E[\|\widetilde U_i^{(i)}\|\E[\|V_{i,\widetilde U_i^{(i)},0,0}\|]])\\
&\lesssim \sum_{i=0}^{n-1}\sqrt{n\E[\|U_i\|^2]\sum_{j=0}^{n-1}\E[\|U_j\|^2]}<\infty,
\end{align*}
where the last ineqaulity holds because $\sup_{i \geq 0}\E[\|U_i\|^2]<\infty$. When $t=0$, because $s=0$ is a boundary point of measure zero in the 
s-integration,
\begin{align*}
\beta_2^{(\xi)}&\leq  \int_{\Xi}\E[\|W-V_\eta\|]\int_0^1\int_{\R^d} \mathbf{1}\{\psi(\xi,s,z)\in \d \eta\}\ \|z\|\ \law(W-V_\xi)(dz)ds\\
&=  \big(\int_{0}^1(1-s)\d s\big)\E[\|W-V_\xi\|^2] = \frac{\E[\|W-V_\xi\|^2]}{2}.
\end{align*}
When $t \in (0,1]$, $W-V_\xi = (1-t)y$ and we have
\begin{align*}
\beta_2^{(\xi)}&\leq \int_{\Xi}\E[\|W-V_\eta\|]\int_0^1 \mathbf{1}\{\psi(\xi,s,(1-t)y)\in \d \eta\}\ (1-t)\|y\|\ \d s\\
&=  \big(\int_{0}^1(1-s)\d s\big)(1-t)^2\|y\|^2 = \frac{(1-t)^2\|y\|^2}{2}.
\end{align*}
Then, we have
\begin{align*}
\beta_2^{(\xi)} \leq \E[\|W-V_\xi\|^2]/2 = \begin{cases}
     \E[\|W-V_\xi\|^2]/2&, \text{ if } t = 0,\\
     (1-t)^2\|y\|^2/2&, \text{otherwise.}
    \end{cases}
\end{align*}
By definition, we have
\begin{align*}
\beta_{123}^{(\xi)}(a, b, c) & \leq \int_{\Xi} b \beta_1^{(\eta)}+c \sqrt{\beta_2^{(\eta)}}|\nu_{\xi}|(\mathrm{d} \eta) \lesssim (b+\frac{c}{\sqrt{2}})\E[\|W-V_\xi\|^2],
\end{align*}
where the last inequality follows by the similar argument used above for \(\beta_2^{(\xi)}\). Finally, for \(\beta_{234}(a,b,c)\), the definition yield
\begin{align*}
\beta_{234}(a,b,c) &\lesssim \int_{\Xi}  (b+\frac{c}{\sqrt{2}})\E[\|W-V_\xi\|^2] |\mu|(\d \xi)\\
&\lesssim \sum_{i=0}^{n-1}\E[\|\widetilde U_i^{(i)}\|\E[\|W-V_{(i,\widetilde U_i^{(i)},0,0)}\|^2]].
\end{align*}
This inequality holds for any coupling of $(U_0,\ldots,U_{n-1})$ and
$\widetilde U^{(i)}=(\widetilde U^{(i)}_0,\ldots,\widetilde U^{(i)}_{n-1})$.
Moreover, by the definition of the Wasserstein--$2$ distance, we can choose the coupling so that
\begin{equation*}
\beta_{234}(a,b,c)\lesssim \sum_{i=0}^{n-1}\E[\|U_i\|\W_2^2(\law(W), \law(W \mid U_i))], 
\end{equation*}
which implies Proposition \ref{prop:newS} by Lemma \ref{lem:raivc}.
\subsection{Proof for $\delta = 1$}\label{sec:proof-local-depend-1-optimal}
We start from the case that \(\sup_{0\le i\le n-1}\mathbb{E}\|X_i\|^{3}\in \mathcal{O}(1)\) and $\Sigma_n = I_d$.  By \eqref{eq:newS} and the coupling discussed in Section \ref{sec:outline-local-depend-1}, we have
\begin{align*}
\W_1\big(\mathcal{L}(\tfrac{S_n}{\sqrt{n}}),\mathcal{N}(0,\Sigma_n)\big) 
&\lesssim \sum_{i=0}^{n-1}\E_{x \sim \law(U_i)}[\|x\|\E[\|W-\sum_{j=0}^{n-1}\widetilde U_j^{(i,x)}\|^2]]\\
&\lesssim D\sum_{i=0}^{n-1}\sum_{j \in \mathcal N[i]}\E[\|\widetilde U_i\|\E[\|U_j - \widetilde U_j^{i, U_i}\|^2]]\\
&\lesssim D\sum_{i=0}^{n-1}\sum_{j \in \mathcal N[i]}(\E[\|U_i\|\E[\|U_j\|^2]]+\E[\|U_i\|\| U_j\|^2]) \lesssim \frac{D^2}{\sqrt{n}},
\end{align*}
where the last inequality follows from Hölder’s inequality and 
\(\sup_{0\le i\le n-1}\mathbb{E}\|U_i\|^{3}\in \mathcal{O}(n^{-3/2})\). For the case \(\sup_{0\le i\le n-1}\mathbb{E}\|X_i\|^{3}\in \mathcal{O}(1)\) with a general $\Sigma_n = \frac{1}{n}\operatorname{Var}(S_n)$ and $\lambda_{\min}(\Sigma_n) \in \Omega(1)$. Because $\Sigma_n$ is positive definite, it admits the eigen-decomposition $\Sigma_n = UVU^\top$, where $U \in \R^{d \times d}$ has orthonormal columns ($U^TU = I_d$), and $V = \operatorname{diag}(\lambda_1, \dots, \lambda_d) \in \R^{d \times d}$ with $\lambda_1 \geq \cdots \geq \lambda_d >0.$ Then, we obtain
\[
\W_1\big(\mathcal{L}(\tfrac{S_n}{\sqrt{n}}),\mathcal{N}(0,\Sigma_n)\big)=\W_1\big(\mathcal{L}(\tfrac{S_n}{\sqrt{n}}),\mathcal{N}(0,UV^{1/2} Z_d)\big),
\]
where $Z_d \sim \mathcal{N}(0, I_d)$.
Define $W:= V^{-1/2}U^\top\tfrac{S_n}{\sqrt{n}}$. Therefore, $\operatorname{Cov}(W) = I_d$. Since the bound depends on the third moment of $\|U_i\|$,
\begin{align*}
\W_1\big(\mathcal{L}(\tfrac{S_n}{\sqrt{n}}),\mathcal{N}(0,\Sigma_n)\big)&=\W_1\big(\mathcal{L}(UV^{1/2}W),\law(UV^{1/2} Z_d)\big)\\
&\lesssim \lambda_{\max}(\Sigma_n)^{1/2} \W_1\big(\mathcal{L}(W),\mathcal{N}(0,I_d)\big) \in \mathcal{O}(\frac{\lambda_{\max}(\Sigma_n)^{1/2}}{\lambda_{\min}(\Sigma_n)^{3/2}}\cdot D^2 \cdot n^{-1/2}).
\end{align*}
\subsection{Proof for $\delta \in (0,1)$}\label{sec:proof-local-depend-1-other}
For the case $\sup_{0\le i\le n-1}\mathbb{E}\|X_i\|^{2+\delta}<\infty$ with $\delta\in(0,1)$ and $\Sigma_n=I_d$, we proceed to bound the three terms in \eqref{eq:threeterm}.
 For the first term, since $W_i'$ is independent of $U_i$ for all $i\in\{0,\ldots,n-1\}$, the analysis of \cite{gallouet2018regularity} applies verbatim; in particular,
 \begin{align*}
 &\left|\frac{1}{n} \sum_{i=0}^{n-1} \mathbb{E}\Big[\Delta f_h(W_i'+U_i)-nU_i^T \nabla^2 f_h\big(W_i'+\theta U_i\big) U_i\Big]\right|\\
=&\left|\frac{1}{n} \sum_{i=0}^{n-1} \mathbb{E}\Big[\Delta f_h(W_i'+U_i)-\Delta f_h\left(W_i'\right)-nU_i^T\left(\nabla^2 f_h\left(W_i'+\theta U_i\right)-\nabla^2 f_h\left(W_i'\right)\right) U_i\Big]\right|\\
\leq & \frac{1}{n} \sum_{i=0}^{n-1}\E\Big[|\Delta f_h(W_i'+U_i)-\Delta f_h\left(W_i'\right)| + n\|U_i\|^2\|\nabla^2 f_h\left(W_i'+\theta U_i\right)-\nabla^2 f_h\left(W_i\right)\|\Big] \\
\lesssim & \frac{1}{n} \sum_{i=0}^{n-1} \Big(\E[\|U_i\|^\delta] + n\E[\|U_i\|^{2+\delta}]\Big) \in \mathcal O(n^{-\delta/2}).
 \end{align*}
 For the second term, we have
 \begin{align*}
   \left|\frac{1}{n} \sum_{i=0}^{n-1}\mathbb{E}\Big[\Delta f_h(W)-\Delta f_h(W_i'+U_i)\Big]\right| &\leq \frac{1}{n} \sum_{i=0}^{n-1}\E[|\Delta f_h(W)-\Delta f_h(W_i'+U_i)|] \\
   &\lesssim \frac{1}{n} \sum_{i=0}^{n-1} \sum_{j \in \mathcal N[i] \setminus \{i\}}\E[\|U_j\|^\delta] \in \mathcal O( Dn^{-\delta/2}).
 \end{align*}
 For the third term, we have
 \begin{align*}
   \left|\sum_{i=0}^{n-1}\mathbb{E}\Big[U_i^T \nabla^2 f_h\big(W_i'+\theta U_i\big) U_i-U_i^T \nabla^2 f_h\big(W_i+\theta U_i\big) U_i\Big]\right| &\leq \sum_{i=0}^{n-1} \sum_{j \in \mathcal N[i] \setminus \{i\} }\E[\|U_i\|^2\|U_j\|^\delta] \\
   & \in \mathcal O( Dn^{-\delta/2}).
 \end{align*}
Combining the three bounds, we obtain
 \[
\W_1\big(\mathcal{L}(\tfrac{S_n}{\sqrt{n}}),\mathcal{N}(0,\Sigma_n)\big)=\W_1\big(\mathcal{L}(\tfrac{S_n}{\sqrt{n}}),\mathcal{N}(0,I_d)\big)
\in \mathcal{O} (Dn^{-\delta/2}).
\]
By the similar linear algebra techniques as above, when $\sup_{0\le i\le n-1}\mathbb{E}\|X_i\|^{2+\delta}=\mathcal{O}(1)$ for some $\delta\in(0,1)$, and for a general
$\Sigma_n=\tfrac{1}{n}\operatorname{Var}(S_n)$ with $\lambda_{\min}(\Sigma_n)\in\Omega(1)$, we obtain
 \[
\W_1\big(\mathcal{L}(\tfrac{S_n}{\sqrt{n}}),\mathcal{N}(0,\Sigma_n)\big) \in \mathcal{O} (\frac{\lambda_{\max}(\Sigma_n)^{1/2}}{\lambda_{\min}(\Sigma_n)^{ \delta/2}}\cdot D\cdot n^{-\delta/2}),
\]
which completes the proof of Theorem \ref{thm:local-depend-1}.
\section{Proof of Theorem \ref{thm:m-depend-p}}\label{sec:m-depend-p}
In this section we complete the proof of Theorem~\ref{thm:m-depend-p} by providing   the additional details
omitted from the proof outline in Section~\ref{sec:outline-m-depend-p}. We start with the normalized case $\Sigma_n = I_d$.

By the definition of $\W_p$,
\begin{equation*}
\W_p\big(\law(\tfrac{S_n}{\sqrt{n}}),\mathcal{L}(\tfrac{A}{\sqrt{n}})\big)
\lesssim \frac{1}{\sqrt{n}}
(\E[\|\Delta\|^p])^{1/p}.
\end{equation*}
Using Rosenthal ineqaulity for the independent and centered $\{V_j\}_{j=1}^k$, we have 
\begin{align*}
(\E[\|B\|^p])^{1/p}\lesssim  \sqrt{\sum_{j=1}^k\E[\|V_j\|^2]}
+ \big(\sum_{j=1}^k\E[\|V_j\|^p]\big)^{1/p}
&\lesssim M\sqrt{k}+Mk^{1/p}\lesssim M \sqrt{k} \lesssim \frac{M\sqrt{n}}{\sqrt{\ell}}.
\end{align*}
Partition $R$ into at most $M+1$ subsets $R_s:=\{i \in R: i \equiv s~(\operatorname{mod}~ M+1)\} $ so that the partial sums $Y_s:= \sum_{i \in R_s}X_i$ are independent. Rosenthal ineqaulity and $|R| \leq \ell+M$ give
\begin{equation*}
\big(\E[\|\sum_{i \in R}X_i\|^p]\big)^{1/p} \le \sum_{s=0}^{M}\big(\E[\|Y_s\|^p]\big)^{1/p}\lesssim (M+1)\sqrt{\frac{|R|}{M+1}}+(M+1)(\frac{|R|}{M+1})^{1/p}\lesssim \sqrt{(M+1)\ell}.
\end{equation*}
Therefore, we have
\begin{equation}\label{eq:step2}
\W_p\big(\law(\tfrac{S_n}{\sqrt{n}}),\mathcal{L}(\tfrac{A}{\sqrt{n}})\big)
\lesssim \frac{1}{\sqrt{n}}(\frac{(M+1)\sqrt{n}}{\sqrt{l}}+\sqrt{(M+1)\ell})= \frac{M+1}{\sqrt{l}} + \frac{\sqrt{(M+1)\ell}}{\sqrt{n}}.
\end{equation}
Set $\mathcal{I}:=\bigcup_{j=1}^k B_j$ and $\mathcal{I}^c:=\{0,\dots,n-1\}\setminus\mathcal{I}$.
For convenience extend $X_i:=0$ when $i\notin\{0,\dots,n-1\}$.
By $M$-dependence,
\[
\Sigma_n=\frac{1}{n}\sum_{i=0}^{n-1}\sum_{h=-M}^M \mathrm{Cov}(X_i,X_{i+h}),\qquad
\frac{\mathrm{Var}(A)}{n}=\frac{1}{n}\sum_{i\in\mathcal{I}}\ \sum_{\substack{h=-M\\ i+h\in\mathcal{I}}}^M \mathrm{Cov}(X_i,X_{i+h}).
\]
Hence
\begin{align*}
&\frac{\mathrm{Var}(A)}{n}-\Sigma_n\\
=&-\frac{1}{n}\sum_{i\in\mathcal{I}^c}\ \sum_{h=-M}^M\mathrm{Cov}(X_i,X_{i+h})
-\frac{1}{n}\sum_{i\in\mathcal{I}}\ \sum_{\substack{h=-M\\ i+h\notin\mathcal{I}}}^M \mathrm{Cov}(X_i,X_{i+h})\\
=&-\frac{1}{n}\sum_{i\in\bigcup_{j=1}^k G_j}\ \sum_{h=-M}^M\mathrm{Cov}(X_i,X_{i+h})-\frac{1}{n}\sum_{i\in R}\ \sum_{h=-M}^M\mathrm{Cov}(X_i,X_{i+h})
-\frac{1}{n}\sum_{i\in\mathcal{I}}\ \sum_{\substack{h=-M\\ i+h\notin\mathcal{I}}}^M \mathrm{Cov}(X_i,X_{i+h})
\end{align*}
By Cauchy--Schwarz, $\|\mathrm{Cov}(X_i,X_j)\|\le (\E[\|X_i\|^2]\E[\|X_j\|^2])^{1/2}<\infty$. Then,
\begin{equation*}
\Big\|\frac{\mathrm{Var}(A)}{n}-\Sigma_n\Big\|
\lesssim \frac{k(2M+1)^2}{n}+ \frac{\ell(2M+1)}{n}+\frac{k(2M+1)^2}{n}\lesssim \frac{(M+1)^2}{\ell} + \frac{\ell(M+1)}{n}.
\end{equation*}
Then, by Lemma \ref{lem:gauss-gauss},
\begin{equation}\label{eq:step4}
\W_p\Big(\mathcal{N}\big(0,\tfrac{\mathrm{Var}(A)}{n}\big),\mathcal{N}(0,\Sigma_n)\Big)
\lesssim  \Big\|\frac{\mathrm{Var}(A)}{n}-\Sigma_n\Big\|^{1/2}\lesssim \frac{M+1}{\sqrt{\ell}} + \frac{\sqrt{(M+1)\ell}}{\sqrt{n}},
\end{equation}
Let $Y_j:=U_j/\sqrt{\ell}$, so that $(Y_j)_{j=1}^k$ are independent with $\sup_{1\le j\le k}\E[\|Y_j\|]^{p+q}<\infty.$ Moreover,
\[
\frac{A}{\sqrt{n}}=\sqrt{\frac{\ell}{n}}\sum_{j=1}^k Y_j
=\sqrt{\frac{k\ell}{n}}\cdot \frac{1}{\sqrt{k}}\sum_{j=1}^k Y_j,
\qquad
\frac{\mathrm{Var}(A)}{n}=\frac{k\ell}{n}\cdot \frac{1}{k}\sum_{j=1}^k\mathrm{Var}(Y_j)
=\frac{k\ell}{n} \bar\Sigma_k,
\]
where $\bar\Sigma_k:=\frac{1}{k}\sum_{j=1}^k\mathrm{Var}(Y_j)$. If we choose $\ell$ such that 
\begin{align*}
\frac{M+1}{\sqrt{\ell}} + \frac{\sqrt{(M+1)\ell}}{\sqrt{n}} \in \mathcal{O}(1),
\end{align*}
we have $\lambda_{\min}(\bar\Sigma_k) \in \Theta(1)$ and $\lambda_{\max}(\bar\Sigma_k) \in \Theta(1)$. 
By Lemma \ref{lem:independent},
\begin{equation}\label{eq:step3}
\begin{aligned}
\W_p\big(\mathcal{L}(\tfrac{A}{\sqrt{n}}),\mathcal{N}(0,\tfrac{\mathrm{Var}(A)}{n})\big)
&= \sqrt{\frac{k\ell}{n}} \cdot  \W_p\big(\mathcal{L}(k^{-1/2}\sum_{j=1}^k Y_j),\mathcal{N}(0,\bar\Sigma_k)\big)\\
&\lesssim \sqrt{\frac{k\ell}{n}} \cdot k^{-\frac{p+q-2}{2p}} \lesssim (\frac{n}{\ell})^{-\frac{p+q-2}{2p}} .
\end{aligned}
\end{equation}
 Then, by the triangle inequality together with \eqref{eq:step2}, \eqref{eq:step4}, and \eqref{eq:step3},
\begin{equation}\label{eq:master}
\W_p\big(\mathcal{L}(\tfrac{S_n}{\sqrt{n}}),\mathcal{N}(0,\Sigma_n)\big)
\lesssim (\frac{n}{\ell})^{-\frac{p+q-2}{2p}} + \frac{M+1}{\sqrt{l}} + \frac{\sqrt{(M+1)\ell}}{\sqrt{n}}
\end{equation}
Choose $\ell:=\lfloor (M+1)^{\frac{2 p}{2 p+q-2}} n^{\frac{p+q-2}{2 p+q-2}}\rfloor$, so that the optimized rate
\[
\W_p\big(\mathcal{L}(\tfrac{S_n}{\sqrt{n}}),\mathcal{N}(0,\Sigma_n)\big)
\in \mathcal O( (M+1)^{1 + \frac{2-q}{2(2p+q-2)}} n^{-\frac{p+q-2}{2(2 p+q-2)}}).
\]
Note that, in the bounds for \eqref{eq:step2} and \eqref{eq:step4}, 
$\|X_i\|$ appears \emph{linearly}, whereas in the bound for
\eqref{eq:step3} it enters with exponent $1+q/p$ (see \cite[Theorem~6]{bonis2020stein}).
Therefore, by the same argument as in the proof of Theorem~\ref{thm:local-depend-1}, for a general
$\Sigma_n$ we obtain
\[
\W_p\Big(\mathcal{L}\big(\tfrac{S_n}{\sqrt{n}}\big),\mathcal{N}(0,\Sigma_n)\Big)
\in
\mathcal{O}\left(
\frac{\lambda_{\max}(\Sigma_n)^{1/2}}{\lambda_{\min}(\Sigma_n)^{1/2}}
(M+1)^{1+\frac{2-q}{2(2p+q-2)}}
n^{-\frac{p+q-2}{2(2p+q-2)}}
\right),
\]
which completes the proof of Theorem \ref{thm:m-depend-p}.
\section{Proof of Corollary \ref{co:U}}\label{sec:proof-U}
Let $I = \{(i_1, \dots, i_r): 0 \leq i_1<\cdots<i_r\leq n-1\}$ and $X_\alpha = h(Z_{i_1}, \dots, Z_{i_r})$ for every $\alpha = (i_1, \dots, i_r) \in I$. We write 
\begin{align*}
W = \sum_{\alpha \in I_n}X_\alpha.
\end{align*}
Note that
\begin{align*}
\operatorname{Var}\big(\tbinom{n}{r}^{-1/2} W\big)
= \tbinom{n}{r}^{-1}\operatorname{Var}(W)
&\succeq \tbinom{n}{r}^{-1}\sum_{\substack{\alpha,\beta\in I\\ |\alpha\cap\beta|=1}} \operatorname{Cov}(X_\alpha,X_\beta)\\
&= \tbinom{n}{r}^{-1}\sum_{\substack{\alpha,\beta\in I\\ |\alpha\cap\beta|=1}}
\operatorname{Var}\big(\mathbb{E}[h(Z_0,\ldots,Z_{r-1}) \mid Z_0]\big),
\end{align*}
which implies
\begin{align*}
\lambda_{\min}\left(\operatorname{Var}\big(\tbinom{n}{r}^{-1/2} W\big)\right)
\gtrsim
\tbinom{n}{r}^{-1}\cdot n \cdot \tbinom{n-1}{r-1}\cdot \tbinom{n-r}{r-1}
\gtrsim n^{r-1},
\end{align*}
since $\tbinom{n}{r}\in \Theta( n^{r})$. Furthermore, we have
\begin{align*}
\lambda_{\max}\left(\operatorname{Var}\big(\tbinom{n}{r}^{-1/2} W\big)\right)&\lesssim \tbinom{n}{r}^{-1}\cdot \sum_{j=1}^r \binom{n}{j}\binom{n-j}{r-j}\binom{n-r}{r-j}\\
&\lesssim\tbinom{n}{r}^{-1}\cdot n \cdot \tbinom{n-1}{r-1}\cdot \tbinom{n-r}{r-1}+ \tbinom{n}{r}^{-1}\cdot \sum_{j=2}^r \binom{n}{j}\binom{n-j}{r-j}\binom{n-r}{r-j}\\
&\lesssim n^{r-1} + \sum_{j=2}^r n^{r-j} \lesssim n^{r-1}.
\end{align*}
Moreover, the maximal degree of the dependency graph of $\{X_\alpha\}_{\alpha\in I}$ satisfies 
\begin{align*}
D = \tbinom{n}{r}-\tbinom{n-r}{r}=\sum_{j=1}^r\tbinom{r}{j}\tbinom{n-r}{r-j} &= \tbinom{r}{1}\tbinom{n-r}{r-1} + \sum_{j=2}^r\tbinom{r}{j}\tbinom{n-r}{r-j}\\
&\in \Theta(n^{r-1}) + \sum_{j=2}^r\Theta(n^{r-j}) \in \Theta(n^{r-1}).
\end{align*}equals
$\tbinom{n}{r}-\tbinom{n-r}{r}$, which 
$\tbinom{n}{r}-\tbinom{n-r}{r}\asymp n^{r-1}$. Then, applying Theorem \ref{thm:local-depend-1} yields
\begin{align*}
\W_1\Big(\law\big(\tbinom{n}{r}^{-1/2} W\big), \mathcal N\big(0, \operatorname{Var}(\tbinom{n}{r}^{-1/2} W)\big)\Big) \lesssim \frac{1}{n^{r-1}} \cdot (n^{r-1})^2 \cdot (n^r)^{-1/2} \lesssim n^{r/2-1}.
\end{align*}
Multiplying $\sqrt{n}\cdot\tbinom{n}{r}^{-1/2}$ to both sides of the above inequality yields
\begin{align*}
\W_1\Big(\law\big(\sqrt{n}U_n\big), \mathcal N\big(0, n\cdot\tbinom{n}{r}^{-2}\operatorname{Var}(W)\big)\Big) \lesssim  n^{-1/2}.
\end{align*}
Notice that
\begin{align*}
n\cdot\tbinom{n}{r}^{-2}\operatorname{Var}(W) &= n\cdot\tbinom{n}{r}^{-2} \sum_{i=1}^r\sum_{\substack{\alpha,\beta\in I\\ |\alpha\cap\beta|=i}}\operatorname{Cov}(X_\alpha, X_\beta)\\
&= n\cdot\tbinom{n}{r}^{-2} \sum_{\substack{\alpha,\beta\in I\\ |\alpha\cap\beta|=1}}\operatorname{Cov}(X_\alpha, X_\beta) + \mathcal{O}(n^{-1})\\
&= n^2\cdot\tbinom{n}{r}^{-2}\tbinom{n-1}{r-1}\tbinom{n-r}{r-1}\operatorname{Var}\big(\mathbb{E}[h(Z_0,\ldots,Z_{r-1}) \mid Z_0]\big) + \mathcal{O}(n^{-1})\\
&= r^2\operatorname{Var}\big(\mathbb{E}[h(Z_0,\ldots,Z_{r-1}) \mid Z_0]\big) + \mathcal{O}(n^{-1}),
\end{align*}
where the last equality follows from the next lemma, whose proof is deferred to Appendix~\ref{sec:com}. 
\begin{lemma}\label{lem:com}
Let $Q(n,r)=n^2\cdot \tbinom{n}{r}^{-2}\tbinom{n-1}{r-1}\tbinom{n-r}{r-1}.$
For $n\ge 2r-1$ (with $r$ fixed), we have
\[
Q(n,r)=r^{2}+\mathcal{O} (1/n). 
\]
\end{lemma}
Therefore, by the triangle inequality,
\begin{align*}
&\W_1\Big(\law\big(\tfrac{\sqrt{n}U_n}{r}\big), \mathcal N\big(0, \operatorname{Var}\big(\mathbb{E}[h(Z_0,\ldots,Z_{r-1}) \mid Z_0]\big)\big)\Big) \\
\lesssim& \W_1\Big(\law\big(\tfrac{\sqrt{n}U_n}{r}\big), \mathcal N\big(0, \tfrac{n}{r^2}\cdot\tbinom{n}{r}^{-2}\operatorname{Var}(W)\big)\Big)\\
&~~~~~+\W_1\Big(\mathcal N\big(0, \tfrac{n}{r^2}\cdot\tbinom{n}{r}^{-2}\operatorname{Var}(W)\big), \mathcal N\big(0, \operatorname{Var}\big(\mathbb{E}[h(Z_0,\ldots,Z_{r-1}) \mid Z_0]\big)\Big)\lesssim n^{-1/2},
\end{align*}
thereby completing the proof of Corollary \ref{co:U}.
\subsection{Proof of Lemma \ref{lem:com}}\label{sec:com}
By definition, we have
\begin{align*}
Q(n, r)=r^2 \frac{(n-r)!^2}{(n-1)!(n-2 r+1)!}=r^2 \prod_{k=1}^{r-1} \frac{n-(r-1+k)}{n-k}=r^2 \prod_{k=1}^{r-1}\left(1-\frac{r-1}{n-k}\right) .
\end{align*}
Take logs and expand
\begin{align*}
\log \frac{Q(n, r)}{r^2}=\sum_{k=1}^{r-1} \log \left(1-\frac{r-1}{n-k}\right)=-(r-1) \sum_{k=1}^{r-1} \frac{1}{n-k}+O(n^{-2}) \in \mathcal{O}(n^{-1}),
\end{align*}
which implies
\begin{align*}
Q(n, r) = r^2 + \mathcal{O}(n^{-1}),
\end{align*}
thereby completing the proof of Lemma \ref{lem:com}.
\section{Proof of Lemma \ref{lem:Sigma-exists1}}\label{sec:proof-Sigma-exists}
In this section, we prove Lemma~\ref{lem:Sigma-exists1} by establishing the following
stronger and more detailed statement. We define $\bar \Sigma_n:= \frac{1}{n}\operatorname{Var}_{\pi}(S_n)$, the normalized covariance matrix  when the chain is initialized in
stationarity, i.e., $x_0\sim\pi$. 
\begin{lemma}
\label{lem:Sigma-exists}
Under Assumption~\ref{assumption:markovchain}, the matrices $\Sigma_n$ and $\bar \Sigma_n$ are well defined for every $n\ge 1$, $\|\Sigma_n-\bar \Sigma_n\| \in \mathcal O(1/n)$, $\lambda_{\min}(\bar \Sigma_n) \in \Omega(1)$, and $\max\{\lambda_{\max}(\Sigma_n),\lambda_{\max}(\bar \Sigma_n)\} \in \mathcal O(1)$. Moreover, in the homogeneous case where
$h_i\equiv h$ for all $i\ge 0$, the limit
\[
\Sigma_\infty:= \lim_{n\to\infty}\frac{1}{n}\operatorname{Var}(S_n)
\]
exists. In addition, $\|\Sigma_n-\Sigma_\infty\|=\mathcal O(1/n)$.
\end{lemma}

By the drift condition,
\[
\E_\mu[V(x_t)] \le \lambda^t \E_\mu[V(x_0)] + \frac{L}{1-\lambda},
\qquad t\ge 0,
\]
so \(\sup_{t\ge 0}\E_\mu[V(x_t)]<\infty\). Since \(\|h_t\|^2\le V\), it follows that
\(\E_\mu\|S_n\|^2<\infty\) for every \(n\), and similarly under stationarity because
\(\pi(V)<\infty\). Hence \(\Sigma_n\) and \(\bar\Sigma_n\) are well defined.

Next, since \(u\mapsto u^{1/2}\) is concave,
\[
P(V^{1/2}) \le (PV)^{1/2}
\le (\lambda V + L\mathbf 1_C)^{1/2}
\le \lambda^{1/2}V^{1/2} + L^{1/2}\mathbf 1_C .
\]
Thus \(V^{1/2}\) is also a Lyapunov function. By Theorem~1.1 of
\cite{baxendale2005renewal}, there exist constants \(c>0\) and \(\rho\in(0,1)\) such that
\begin{equation}\label{eq:htoV-short}
\|P^k h_t\| \le c\rho^k V^{1/2},
\qquad t,k\ge 0,
\end{equation}
where we used \(\pi(h_t)=0\) and \(\|h_t\|\le V^{1/2}\). Define
\[
g_t(x):=\sum_{k=0}^\infty P^k h_{t+k}(x).
\]
By \eqref{eq:htoV-short}, the series converges absolutely in the
\(V^{1/2}\)-weighted norm, \(\|g_t\|\lesssim V^{1/2}\) uniformly in \(t\), and
\[
g_t-Pg_{t+1}=h_t.
\]
Let \(\mathcal F_t:=\sigma(x_0,\dots,x_t)\), and define
\[
m_t:=g_t(x_t)-Pg_t(x_{t-1}),\qquad
M_n:=\sum_{t=1}^{n-1}m_t,\qquad
R_n:=g_0(x_0)-Pg_n(x_{n-1}).
\]
Then \((m_t,\mathcal F_t)\) is a square-integrable martingale difference sequence, and
\begin{equation}\label{eq:Sn-decomp-short}
S_n=M_n+R_n .
\end{equation}
Since \(\|g_t\|\lesssim V^{1/2}\), the drift bound yields
\begin{equation}\label{eq:Rn-short}
\sup_{n\ge 1}\E_\mu\|R_n\|^2+\sup_{n\ge 1}\E_\pi\|R_n\|^2<\infty.
\end{equation}
Now define
\[
H_t(x):=\E\!\left[(g_t(Y)-Pg_t(x))(g_t(Y)-Pg_t(x))^\top\mid Y\sim P(x,\cdot)\right].
\]
Then \(H_t(x)\succeq 0\), and by Jensen's inequality and \(\|g_t\|\lesssim V^{1/2}\),
\[
\|H_t(x)\|
\le 2P(\|g_t\|^2)(x)+2\|Pg_t(x)\|^2
\le 4P(\|g_t\|^2)(x)
\lesssim PV(x)\lesssim V(x),
\]
uniformly in \(t\) and \(x\). Since the martingale differences are orthogonal in \(L^2\),
\[
\operatorname{Var}_\mu(M_n)=\sum_{t=1}^{n-1}\E_\mu[H_t(x_{t-1})],
\qquad
\operatorname{Var}_\pi(M_n)=\sum_{t=1}^{n-1}\pi(H_t),
\]
and therefore
\begin{equation}\label{eq:Mn-short}
\|\operatorname{Var}_\mu(M_n)\|+\|\operatorname{Var}_\pi(M_n)\|\lesssim n.
\end{equation}
We next show that the boundary term contributes only \(O(1)\) to the covariance.
The term \(g_0(x_0)\) is orthogonal to \(M_n\). For the other term, the Markov property gives
\[
\E[Pg_n(x_{n-1})\mid \mathcal F_t]=P^{\,n-t}g_n(x_t).
\]
Since \(\pi(g_n)=0\) and \(\|g_n\|\lesssim V^{1/2}\), applying
\eqref{eq:htoV-short} to \(g_n\) gives
\[
\|P^{\,n-t}g_n(x)\|\lesssim \rho^{\,n-t}V^{1/2}(x).
\]
Hence, by Cauchy--Schwarz and the uniform \(L^2\) bound on \(m_t\),
\[
\|\E_\mu[m_tPg_n(x_{n-1})^\top]\|\lesssim \rho^{\,n-t},
\]
and summing over \(t\) yields
\[
\|\operatorname{Cov}_\mu(M_n,R_n)\|+\|\operatorname{Cov}_\pi(M_n,R_n)\|=\mathcal O(1).
\]
Combining this with \eqref{eq:Rn-short} and \eqref{eq:Mn-short}, we obtain
\begin{equation}\label{eq:VarSn-short}
\operatorname{Var}_\mu(S_n)=\operatorname{Var}_\mu(M_n)+\mathcal O(1),
\qquad
\operatorname{Var}_\pi(S_n)=\operatorname{Var}_\pi(M_n)+\mathcal O(1).
\end{equation}
Therefore,
\[
\lambda_{\max}(\Sigma_n)=\mathcal O(1),
\qquad
\lambda_{\max}(\bar\Sigma_n)=\mathcal O(1).
\]
To compare \(\Sigma_n\) and \(\bar\Sigma_n\), note from \eqref{eq:VarSn-short} that
\[
\operatorname{Var}_\mu(S_n)-\operatorname{Var}_\pi(S_n)
=
\sum_{t=1}^{n-1}\bigl(\mu P^{t-1}H_t-\pi(H_t)\bigr)+\mathcal O(1).
\]
Since \(\sup_{t\ge 1}\|H_t\|_V<\infty\), \(V\)-geometric ergodicity implies
\[
\|\mu P^{t-1}H_t-\pi(H_t)\|\lesssim \rho^{t-1}.
\]
Summing over \(t\) gives
\[
\|\operatorname{Var}_\mu(S_n)-\operatorname{Var}_\pi(S_n)\|=\mathcal O(1),
\]
and hence
\[
\|\Sigma_n-\bar\Sigma_n\|
=
\frac1n\|\operatorname{Var}_\mu(S_n)-\operatorname{Var}_\pi(S_n)\|
=
\mathcal O(1/n).
\]
Under the standing nondegeneracy assumption \(\lambda_{\min}(\Sigma_n)\gtrsim 1\),
Weyl's inequality yields
\[
\lambda_{\min}(\bar\Sigma_n)
\ge \lambda_{\min}(\Sigma_n)-\|\Sigma_n-\bar\Sigma_n\|
\gtrsim 1.
\]
Finally, in the homogeneous case \(h_t\equiv h\), we have \(g_t\equiv g\) and
\(H_t\equiv H\). Define
\[
\Sigma_\infty:=\pi(H)=\E_\pi[m_1m_1^\top]\succeq 0.
\]
Then
\[
\frac1n\operatorname{Var}_\mu(M_n)-\Sigma_\infty
=
\frac1n\sum_{t=1}^{n-1}\bigl(\mu P^{t-1}H-\pi(H)\bigr),
\]
whose norm is \(\mathcal O(1/n)\) by geometric ergodicity. Using
\eqref{eq:VarSn-short} once more,
\[
\|\Sigma_n-\Sigma_\infty\|
\le
\Bigl\|\frac1n\operatorname{Var}_\mu(M_n)-\Sigma_\infty\Bigr\|+\mathcal O(1/n)
=
\mathcal O(1/n).
\]
In particular, \(\Sigma_\infty=\lim_{n\to\infty}n^{-1}\operatorname{Var}_\mu(S_n)\) exists.

\section{Proof of Theorem \ref{thm:MC-W1}}\label{sec:proof-MC-W1}
In what follows, we prove Theorem~\ref{thm:MC-W1} by filling in the details of the proof outline
presented in Section~\ref{sec:outline-MC-W1}.
\subsection{Preliminaries}
In this subsection, we collect several useful lemmas that will be used in the proof of Theorem~\ref{thm:MC-W1}.
\begin{lemma}\label{lem:pi(V)}
Under Assumption \ref{assumption:markovchain}, we have $\pi(V)=\int V\d\pi\le \frac{L}{1-\lambda}<\infty.$
\end{lemma}
\begin{proof}{Proof of Lemma \ref{lem:pi(V)}}

    Since $\pi$ is invariant,
\[
\pi(V)=\int V\d\pi = \int PVd\pi \le \lambda \int Vd\pi + L\int \mathbf{1}_Cd\pi \le \lambda\pi(V)+L .
\]
Therefore,
\[
\pi(V)\le \frac{L}{1-\lambda}<\infty.
\]
\end{proof}
\begin{lemma}
\label{lem:gauss-gauss}
Let $A,B$ be symmetric positive semidefinite $d\times d$ matrices.
Then for any $p \geq 1$,
\[
W_{p}\bigl(\mathcal N(0,A),\mathcal N(0,B)\bigr) \lesssim \|A-B\|^{1/2}.
\]
\end{lemma}
\begin{proof}{Proof of Lemma \ref{lem:gauss-gauss}}
Let $Z\sim\mathcal N(0,I_d)$ and consider the coupling $X=A^{1/2}Z$ and $Y=B^{1/2}Z$. Then
\begin{align*}
\W_p\bigl(\law(X),\law(Y)\bigr)
&\le \Bigl(\E\|X-Y\|^p\Bigr)^{1/p} \\
&= \Bigl(\E\bigl\|(A^{1/2}-B^{1/2})Z\bigr\|^p\Bigr)^{1/p} \\
&\lesssim \|A^{1/2}-B^{1/2}\|
\lesssim \|A^{1/2}-B^{1/2}\|_{\mathrm{HS}}
\overset{\text{(i)}}{\leq} \|A-B\|_{\mathrm{HS}}^{1/2} \lesssim \|A-B\|^{1/2}.
\end{align*}
where (i) holds by the Powers–Størmer inequality~\citep{powers1970free}.
\end{proof}

\subsection{Reduction to Stationary Initialization}\label{sec:stationary}
In this section, we show that, without loss of generality, we may assume the chain is initialized from the stationary distribution $\pi$: the effect of a nonstationary initialization contributes only a higher-order term to the final CLT rate.

Let $x_0\sim \mu$ and $y_0\sim \pi$ be independent. Conditional on $(x_0,y_0)$,
apply Lemma \ref{lem:forward-meeting} to obtain a coupling
$(X_t^x,X_t^y)_{t\ge 0}$ of two copies of the $P$-chain with
$X_0^x=x_0$, $X_0^y=y_0$, and meeting time
\[
  T^+:=\inf\{t\ge 0:\ X_t^x=X_t^y\},
\]
satisfying
\begin{equation}\label{eq:T-tail}
\P(T^+>k | x_0,y_0)\ \lesssim\ (V(x_0)+V(y_0))\rho^k,
  \qquad k\ge 0.
\end{equation}
Define the coupled partial sums
\[
  S_n^x:=\sum_{t=0}^{n-1} h_t(X_t^x),
  \qquad
  S_n^y:=\sum_{t=0}^{n-1} h_t(X_t^y),
  \qquad
  \bar S_n^x:=\frac{1}{\sqrt n}S_n^x, \qquad \bar S_n^y:=\frac{1}{\sqrt n}S_n^y .
\]
By construction, $X_t^x=X_t^y$ for all $t\ge T^+$, hence
\[
  S_n^x-S_n^y=\sum_{t=0}^{n-1}\bigl(h_t(X_t^x)-h_t(X_t^y)\bigr)\mathbf 1_{\{T^+>t\}}.
\]
Therefore, 
\begin{align*}
  \E\bigl[\|\bar S_n^x-\bar S_n^y\|\ \big|\ x_0,y_0\bigr]
  &\le \frac{1}{\sqrt n}\sum_{t=0}^{n-1}
      \E\Bigl[\|h_t(X_t^x)\|+\|h_t(X_t^y)\|\mathbf 1_{\{T^+>t\}}
      \ \Big|\ x_0,y_0\Bigr] \\
  &\lesssim \frac{1}{\sqrt n}\sum_{t=0}^{n-1}
      \Bigl(\E[\|h_t(X_t^x)\|^2\mid x_0]+\E[\|h_t(X_t^y)\|^2\mid y_0]\Bigr)^{1/2}
      \P(T^+>t\mid x_0,y_0)^{1/2}.
\end{align*}
By Assumption \ref{assumption:markovchain},
\begin{align*}
\E[\|h_t(X_t^x)\|^2\mid x_0]\le \E[V(x_t)\mid x_0] \leq \lambda^tV(x_0) + \frac{L}{1-\lambda}.
\end{align*}
and
$\E[\|h_t(X_t^y)\|^2\mid y_0]\le \lambda^tV(y_0) + \frac{L}{1-\lambda}$. Then, by inequality \eqref{eq:T-tail},
\begin{align*}
  \E\bigl[\|\bar S_n^x-\bar S_n^y\|\ \big|\ x_0,y_0\bigr]
  \ &\lesssim\ \frac{1}{\sqrt n}\sum_{t=0}^{n-1}
  (\lambda^tV(y_0) + \lambda^tV(y_0) + 1)^{1/2}
  \sqrt{V(x_0)+V(y_0)}\rho^{t/2}\\&\lesssim \frac{1}{\sqrt n}\sum_{t=0}^{n-1}
  (V(x_0)+V(y_0))\rho^{t/2} \lesssim \frac{1}{\sqrt n}
  (V(x_0)+V(y_0)).
\end{align*}
Then, 
\begin{equation*}
  \E\bigl[\|\bar S_n^x-\bar S_n^y\|\ \big|\ x_0,y_0\bigr]
  \ \lesssim\ \frac{1}{\sqrt n}\bigl(V(x_0)+V(y_0)+1\bigr).
\end{equation*}
Consequently,
\begin{equation}\label{eq:W1-to-stationary}
 W_1\bigl(\law(\tfrac{S_n}{\sqrt{n}}),\law_\pi(\tfrac{S_n}{\sqrt{n}})\bigr)
  \le \E\bigl[\|\bar S_n^x-\bar S_n^y\|\bigr]
  \lesssim \frac{1}{\sqrt n}\Bigl(\E_\mu[V(x_0)]+\pi(V)+1\Bigr)
  \lesssim \frac{1}{\sqrt n},
\end{equation}
where the last inequality uses Lemma~\ref{lem:pi(V)} (in particular, $\pi(V)<\infty$)
together with the standing assumption $\E_\mu[V(x_0)]<\infty$.
\subsection{Main Proof}\label{sec:proof-MC-W1-Raic}

We first work under stationary initialization and the normalization
\(\bar\Sigma_n=I_d\). The reduction from a general initialization to
\(x_0\sim\pi\) is given in Section~\ref{sec:stationary}; the removal of the
normalization \(\bar\Sigma_n=I_d\) is handled at the end.

Set
\[
  U_i:=\frac{h_i(x_i)}{\sqrt n},
  \qquad
  W:=\sum_{i=0}^{n-1}U_i.
\]
By Proposition~\ref{prop:newS}, it suffices to prove
\begin{equation}\label{eq:Rn-goal}
  R_n
  :=
  \sum_{i=0}^{n-1}
  \E\Big[
    \|U_i\|\,\W_2^2\bigl(\law(W),\law(W\mid U_i)\bigr)
  \Big]
  \lesssim n^{-1/2}.
\end{equation}

For \(x\in\X\), write
\[
  \nu_x:=\law(W\mid x_i=x).
\]
Since \(U_i\) is \(\sigma(x_i)\)-measurable, the convexity of
\(\mu\mapsto \W_2^2(\law(W),\mu)\) and disintegration give
\[
  R_n
  \le
  \sum_{i=0}^{n-1}
  \E\Big[
    \|U_i\|\,\W_2^2\bigl(\law(W),\nu_{x_i}\bigr)
  \Big].
\]
Under stationarity, \(x_i\sim\pi\), hence
\[
  R_n
  \le
  \frac1{\sqrt n}\sum_{i=0}^{n-1}
  \int_{\X}
    \|h_i(x)\|\,
    \W_2^2\bigl(\law(W),\nu_x\bigr)\,
  \pi(\d x).
\]
Moreover, since \(\law(W)=\int_{\X}\nu_y\,\pi(\d y)\), convexity of
\(\mu\mapsto \W_2^2(\mu,\nu_x)\) yields
\[
  \W_2^2\bigl(\law(W),\nu_x\bigr)
  \le
  \int_{\X}\W_2^2(\nu_y,\nu_x)\,\pi(\d y).
\]
Therefore
\begin{equation}\label{eq:Rn-mixture}
  R_n
  \le
  \frac1{\sqrt n}\sum_{i=0}^{n-1}
  \iint_{\X\times\X}
    \|h_i(x)\|\,
    \W_2^2(\nu_y,\nu_x)\,
  \pi(\d y)\pi(\d x).
\end{equation}
Fix \(i\in\{0,\dots,n-1\}\) and \(x,y\in\X\). Let
\((X_t^x)_{t\in\mathbb Z}\) and \((X_t^y)_{t\in\mathbb Z}\) be two bi-infinite
processes such that
\[
  \law\bigl((X_t^x)_{t\in\mathbb Z}\bigr)
  =
  \law\bigl((x_t)_{t\in\mathbb Z}\mid x_i=x\bigr),
  \qquad
  \law\bigl((X_t^y)_{t\in\mathbb Z}\bigr)
  =
  \law\bigl((x_t)_{t\in\mathbb Z}\mid x_i=y\bigr).
\]
By Lemma~\ref{lem:backward}, the time-reversed kernel \(P^\ast\) is
geometrically ergodic. Let \(V'\ge1\) be a drift function for \(P^\ast\) with
\(\pi(V')<\infty\). Set
\[
  \alpha:=\frac{\delta-1}{\delta}\in(0,1).
\]
Since \(u\mapsto u^\alpha\) is concave, \(V^\alpha\) and \((V')^\alpha\) are
drift functions for \(P\) and \(P^\ast\), respectively. Applying
Lemma~\ref{lem:forward-meeting} to \(P\) and \(P^\ast\), and using the standard
time-reversal property of stationary Markov chains, we may couple
\((X_t^x)\) and \((X_t^y)\) so that there exist meeting times
\(T_i^+(x,y)\) and \(T_i^-(x,y)\) such that
\[
  X_t^x=X_t^y
  \qquad\text{whenever}\qquad
  t\ge i+T_i^+(x,y)
  \ \ \text{or}\ \
  t\le i-T_i^-(x,y),
\]
and for some \(\rho_+,\rho_-\in(0,1)\),
\begin{align*}
  \P\bigl(T_i^+(x,y)>k\bigr)
  &\lesssim
  \bigl(V(x)^\alpha+V(y)^\alpha\bigr)\rho_+^k,\\
  \P\bigl(T_i^-(x,y)>k\bigr)
  &\lesssim
  \bigl((V'(x))^\alpha+(V'(y))^\alpha\bigr)\rho_-^k,
  \qquad k\ge0.
\end{align*}
Let \(\bar\rho:=\max\{\rho_+,\rho_-\}\in(0,1)\). Define
\[
  W^x:=\frac1{\sqrt n}\sum_{t=0}^{n-1}h_t(X_t^x),
  \qquad
  W^y:=\frac1{\sqrt n}\sum_{t=0}^{n-1}h_t(X_t^y).
\]
Then \(\law(W^x)=\nu_x\) and \(\law(W^y)=\nu_y\). Hence
\begin{equation}\label{eq:W2-coupling-xy}
  \W_2^2(\nu_y,\nu_x)
  \le
  \E\bigl[\|W^y-W^x\|^2\bigr]
  =
  \frac1n\E\Big\|
    \sum_{t=0}^{n-1}\Delta_t(x,y)
  \Big\|^2,
\end{equation}
where $\Delta_t(x,y):=h_t(X_t^y)-h_t(X_t^x).$ Let \(\widetilde\rho:=\bar\rho^{\delta/(2+\delta)}\in(0,1)\), and choose
\(a\in(\widetilde\rho,1)\). By weighted Cauchy--Schwarz,
\[
  \E\Big\|
    \sum_{t=0}^{n-1}\Delta_t(x,y)
  \Big\|^2
  \lesssim
  \sum_{t=0}^{n-1}a^{-|t-i|}\E\bigl[\|\Delta_t(x,y)\|^2\bigr].
\]
For \(E_t(x,y):=\{X_t^x\neq X_t^y\}\), we have
\[
  \|\Delta_t(x,y)\|^2
  \le
  2\bigl(\|h_t(X_t^x)\|^2+\|h_t(X_t^y)\|^2\bigr)\mathbf 1_{E_t(x,y)}.
\]
Applying H\"older's inequality with exponents
\((2+\delta)/2\) and \((2+\delta)/\delta\), we obtain
\[
  \E\bigl[\|\Delta_t(x,y)\|^2\bigr]
  \lesssim
  M_{i,t}(x,y)\,
  \P\bigl(E_t(x,y)\bigr)^{\delta/(2+\delta)},
\]
where
\[
  M_{i,t}(x,y)
  :=
  \Big(
    \E\big[
      \|h_t(X_t^x)\|^{2+\delta}
      +
      \|h_t(X_t^y)\|^{2+\delta}
    \big]
  \Big)^{\frac{2}{2+\delta}}.
\]
Moreover,
\[
  E_t(x,y)\subseteq
  \begin{cases}
    \{T_i^+(x,y)>t-i\}, & t\ge i,\\[2mm]
    \{T_i^-(x,y)>i-t\}, & t<i,
  \end{cases}
\]
so the meeting-time bounds imply
\[
  \P\bigl(E_t(x,y)\bigr)^{\delta/(2+\delta)}
  \lesssim
  A(x,y)\,\widetilde\rho^{|t-i|},
\]
where
\[
  A(x,y):=
  V(x)^{\frac{\delta-1}{2+\delta}}
  +V(y)^{\frac{\delta-1}{2+\delta}}
  +(V'(x))^{\frac{\delta-1}{2+\delta}}
  +(V'(y))^{\frac{\delta-1}{2+\delta}}.
\]
Combining the last three displays with \eqref{eq:W2-coupling-xy}, we get
\begin{equation}\label{eq:W2-final-xy}
  \W_2^2(\nu_y,\nu_x)
  \lesssim
  \frac{A(x,y)}{n}
  \sum_{t=0}^{n-1}
    M_{i,t}(x,y)\,
    \widehat\rho^{|t-i|},
\end{equation}
where \(\widehat\rho:=\widetilde\rho/a\in(0,1)\).

Substituting \eqref{eq:W2-final-xy} into \eqref{eq:Rn-mixture} yields
\[
  R_n
  \lesssim
  \frac1{n^{3/2}}
  \sum_{i=0}^{n-1}\sum_{t=0}^{n-1}
  \iint_{\X\times\X}
    \|h_i(x)\|\,A(x,y)\,M_{i,t}(x,y)\,
  \pi(\d y)\pi(\d x)\,
  \widehat\rho^{|t-i|}.
\]
Now apply H\"older's inequality on \(\X\times\X\) with exponents
\[
  2+\delta,\qquad \frac{2+\delta}{\delta-1},\qquad \frac{2+\delta}{2}.
\]
Since \(\|h_i\|^{2+\delta}\le V\),
\[
  \int \|h_i(x)\|^{2+\delta}\,\pi(\d x)\le \pi(V)<\infty.
\]
Also,
\[
  A(x,y)^{\frac{2+\delta}{\delta-1}}
  \lesssim
  V(x)+V(y)+V'(x)+V'(y),
\]
hence
\[
  \iint A(x,y)^{\frac{2+\delta}{\delta-1}}
  \,\pi(\d y)\pi(\d x)
  <\infty.
\]
Finally,
\[
  M_{i,t}(x,y)^{\frac{2+\delta}{2}}
  =
  \E\big[
    \|h_t(X_t^x)\|^{2+\delta}
    +
    \|h_t(X_t^y)\|^{2+\delta}
  \big].
\]
Integrating over \(x,y\sim\pi\), the unconditional laws of \(X_t^x\) and
\(X_t^y\) are both \(\pi\), so
\[
  \iint M_{i,t}(x,y)^{\frac{2+\delta}{2}}
  \,\pi(\d y)\pi(\d x)
  \le
  2\pi(V)<\infty.
\]
Therefore,
\[
  \sup_{0\le i,t\le n-1}
  \iint
    \|h_i(x)\|\,A(x,y)\,M_{i,t}(x,y)\,
  \pi(\d y)\pi(\d x)
  <\infty.
\]
It follows that
\[
  R_n
  \lesssim
  \frac1{n^{3/2}}
  \sum_{i=0}^{n-1}\sum_{t=0}^{n-1}\widehat\rho^{|t-i|}
  \lesssim
  n^{-1/2}.
\]
By Proposition~\ref{prop:newS}, this proves that under \(x_0\sim\pi\) and
\(\bar\Sigma_n=I_d\),
\[
  \W_1\left(
    \law_\pi(S_n/\sqrt{n}),
    \mathcal N(0,I_d)
  \right)
  \lesssim
  n^{-1/2}.
\]
For a general \(\bar\Sigma_n\), the linear-algebraic reduction in
Appendix~\ref{sec:local-depend-1}, together with Lemma~\ref{lem:Sigma-exists},
yields
\[
  \W_1\!\left(
    \law_\pi(S_n/\sqrt{n}),
    \mathcal N(0,\bar\Sigma_n)
  \right)
  \lesssim
  n^{-1/2}.
\]
Finally, combining this with the initialization reduction
\eqref{eq:W1-to-stationary}, Lemma~\ref{lem:Sigma-exists}, and
Lemma~\ref{lem:gauss-gauss}, we obtain
\[
  \W_1\!\left(
    \law(S_n/\sqrt{n}),
    \mathcal N(0,\Sigma_n)
  \right)
  \lesssim
  n^{-1/2}.
\]
This completes the proof.
\subsection{Proof of Lemma \ref{lem:backward}}\label{sec:proof-backward}
Let \(\pi\) denote the unique invariant probability of \(P\), and define the measure
\(\mathcal M\) on \((\mathcal X\times\mathcal X,\mathcal B\otimes\mathcal B)\) by
\[
\mathcal M(A\times B):=\int_A \pi(dx)\,P(x,B),\qquad A,B\in\mathcal B.
\]
Since \(\pi P=\pi\), the second marginal of \(\mathcal M\) is \(\pi\). Hence, by
disintegration, there exists a Markov kernel \(P^\ast\) such that
\begin{equation}\label{eq:TR}
\int_A \pi(dx)\,P(x,B)=\int_B \pi(dy)\,P^\ast(y,A),
\qquad A,B\in\mathcal B.
\end{equation}
After modifying \(P^\ast\) on a \(\pi\)-null set if necessary, we may regard it as
defined on all of \(\mathcal X\). Iterating \eqref{eq:TR} yields, for every
\(n\ge1\),
\begin{equation}\label{eq:TR-n}
\int_A \pi(dx)\,P^n(x,B)=\int_B \pi(dy)\,(P^\ast)^n(y,A),
\qquad A,B\in\mathcal B.
\end{equation}
Under Assumption~\ref{assumption:markovchain}, Theorem~15.0.1 of
\cite{Meyn12_book} implies that \(P\) is \(V\)-uniformly geometrically ergodic:
there exist constants \(R<\infty\) and \(\rho\in(0,1)\) such that
\begin{equation}\label{eq:V-unif}
\bigl\|P^n(x,\cdot)-\pi\bigr\|_V
:=\sup_{|f|\le V}\bigl|P^n f(x)-\pi(f)\bigr|
\le RV(x)\rho^n,
\qquad x\in\mathcal X,\ n\ge0.
\end{equation}
Since \(V\ge1\), testing \eqref{eq:V-unif} with \(f=\mathbf 1_A\) gives
\begin{equation}\label{eq:tv-bound}
\bigl\|P^n(x,\cdot)-\pi\bigr\|_{\mathrm{TV}}
=\sup_{A\in\mathcal B}\bigl|P^n(x,A)-\pi(A)\bigr|
\le RV(x)\rho^n.
\end{equation}
Fix \(A\in\mathcal B\) with \(\pi(A)>0\). If \(V(x)<\infty\), then for all
sufficiently large \(n\),
\[
P^n(x,A)\ge \pi(A)-RV(x)\rho^n>0.
\]
Therefore the set
\begin{equation}\label{eq:BA-def}
B_A:=\Bigl\{x\in\mathcal X:\ P^n(x,A)=0\ \text{for all }n\ge0\Bigr\}
\end{equation}
is contained in \(\{V=\infty\}\). Since \(\pi(V)<\infty\) by
Lemma~\ref{lem:pi(V)}\,,
\begin{equation}\label{eq:BA-zero}
\pi(B_A)=0.
\end{equation}

\noindent\textbf{$\psi$-irreducibility of $P^\ast$.} 
Fix \(A\in\mathcal B\) with
\(\pi(A)>0\), and define
\[
B:=\Bigl\{y\in\mathcal X:\ (P^\ast)^n(y,A)=0\ \text{for all }n\ge0\Bigr\}.
\]
Then \(B\subseteq A^c\). For \(n\ge0\), let
\[
G_n:=\Bigl\{y\in\mathcal X:\ (P^\ast)^n(y,A)>0\Bigr\},
\]
so that \(B^c=\bigcup_{n\ge0}G_n\). We claim that \(B\) is absorbing for \(P^\ast\).
Indeed, if \(y\in B\) and \(P^\ast(y,B^c)>0\), then \(P^\ast(y,G_n)>0\) for some
\(n\), and hence
\[
(P^\ast)^{n+1}(y,A)
=\int_{\mathcal X}P^\ast(y,dz)\,(P^\ast)^n(z,A)
\ge \int_{G_n}P^\ast(y,dz)\,(P^\ast)^n(z,A)>0,
\]
contradicting \(y\in B\). Thus
\[
P^\ast(y,B)=1,\qquad y\in B,
\]
and therefore, by induction,
\[
(P^\ast)^m(y,B)=1,\qquad y\in B,\ m\ge1.
\]
Applying \eqref{eq:TR-n}, we obtain
\[
\int_{B^c}\pi(dy)\,P^m(y,B)
=
\int_B \pi(dx)\,(P^\ast)^m(x,B^c)
=0,
\qquad m\ge1.
\]
Hence
\[
P^m(y,B)=0\qquad\text{for \(\pi\)-a.e.\ }y\in B^c,\ \ m\ge1.
\]
Using \(\pi P^m=\pi\),
\[
\pi(B)
=\int_{\mathcal X}\pi(dy)\,P^m(y,B)
=\int_B \pi(dy)\,P^m(y,B),
\]
so \(P^m(y,B)=1\) for \(\pi\)-a.e.\ \(y\in B\), for each \(m\ge1\). Intersecting
over \(m\ge1\), there exists \(B'\subseteq B\) with \(\pi(B')=\pi(B)\) such that
\[
P^m(y,B)=1,\qquad y\in B',\ \ m\ge1.
\]
Since \(B\subseteq A^c\), this implies \(P^m(y,A)=0\) for all \(m\ge0\) and all
\(y\in B'\), i.e.\ \(B'\subseteq B_A\). By \eqref{eq:BA-zero},
\(\pi(B')=0\), hence \(\pi(B)=0\). Therefore, for every \(A\in\mathcal B\) with
\(\pi(A)>0\), for \(\pi\)-a.e.\ \(y\) there exists \(n\) such that
\((P^\ast)^n(y,A)>0\). After modifying \(P^\ast\) on a \(\pi\)-null set if
necessary, we may assume that \(P^\ast\) is \(\pi\)-irreducible, and hence
\(\psi\)-irreducible (with \(\psi=\pi\)).

\medskip
\noindent\textbf{Step 2: Aperiodicity of $P^\ast$.} 
Suppose instead that \(P^\ast\)
has period \(d\ge2\). Then there exist measurable sets
\(D_0,\dots,D_{d-1}\) with \(\pi(D_i)>0\),
\(\pi\bigl(\bigcup_{i=0}^{d-1}D_i\bigr)=1\), and
\[
P^\ast\bigl(x,D_{i-1\ (\mathrm{mod}\ d)}\bigr)=1
\qquad\text{for \(\pi\)-a.e.\ }x\in D_i.
\]
Let
\[
N:=\mathcal X\setminus\bigcup_{i=0}^{d-1}D_i.
\]
Then \(\pi(N)=0\), and since \(\pi P=\pi\), $0=\pi(N)=\int_{\mathcal X}\pi(dy)\,P(y,N),$ so
\[
P(y,N)=0\qquad\text{for \(\pi\)-a.e.\ }y\in\mathcal X.
\]
Now fix \(j\in\{0,\dots,d-1\}\) and \(i\neq j+1 \pmod d\). Since
\(P^\ast(x,D_j)=0\) for \(\pi\)-a.e.\ \(x\in D_i\), the time-reversal identity
\eqref{eq:TR} gives
\[
\int_{D_j}\pi(dy)\,P(y,D_i)
=
\int_{D_i}\pi(dx)\,P^\ast(x,D_j)
=0.
\]
Hence $P(y,D_i)=0$ for $\text{\(\pi\)-a.e.\ }y\in D_j.$ Because this holds for every \(i\neq j+1 \pmod d\), and because \(P(y,N)=0\) for
\(\pi\)-a.e.\ \(y\), we conclude that
\[
P\bigl(y,D_{j+1\ (\mathrm{mod}\ d)}\bigr)=1
\qquad\text{for \(\pi\)-a.e.\ }y\in D_j.
\]
Thus, modulo the \(\pi\)-null set \(N\), the sets \(D_0,\dots,D_{d-1}\) form a
cyclic decomposition for \(P\), contradicting the assumed aperiodicity of \(P\).
Therefore \(P^\ast\) is aperiodic.

\medskip
\noindent\textbf{Step 3: Geometric Ergodicity of $P^\ast$.} 
As discussed in Section~\ref{sec:splitchain}, there exist an accessible small set
\(\bar C\), an integer \(m\ge1\), a constant \(\beta\in(0,1)\), and a probability
measure \(\nu\) with \(\nu(\bar C)=1\) such that, for $Q:=P^m,$ we have
\begin{equation*}
Q(x,\cdot)\ge \beta \mathbf 1_{\bar C}(x)\nu(\cdot),
\qquad x\in\mathcal X.
\end{equation*}
Let \(Q^\ast\) denote the time-reversal of \(Q\) with respect to \(\pi\). By
\eqref{eq:TR-n}, $Q^\ast=(P^\ast)^m.$

Choose a measurable \(r:\mathcal X\times\mathcal X\to[0,1]\) such that
\begin{equation}\label{eq:r-def}
\int_A r(x,y)\,Q(x,dy)
=\beta \mathbf 1_{\bar C}(x)\nu(A),
\qquad x\in\mathcal X,\ A\in\mathcal B.
\end{equation}
Define the split kernel \(\widetilde Q\) on
\(\widetilde{\mathcal X}:=\mathcal X\times\{0,1\}\) by
\begin{align}
\widetilde Q\big((x,i),A\times\{1\}\big)
&:= \int_A r(x,y)\,Q(x,dy), \label{eq:split-1}\\
\widetilde Q\big((x,i),A\times\{0\}\big)
&:= \int_A (1-r(x,y))\,Q(x,dy). \nonumber
\end{align}
Its \(\mathcal X\)-marginal is \(Q\), so by iteration
\begin{equation}\label{eq:proj-forward}
\widetilde Q^n\big((x,i),B\times\{0,1\}\big)=Q^n(x,B),
\qquad n\ge0.
\end{equation}
Define \(\widetilde\pi\) on \(\widetilde{\mathcal X}\) by
\begin{align*}
\widetilde\pi(A\times\{1\}) := \beta\pi(\bar C)\nu(A), \qquad \widetilde\pi(A\times\{0\}) := \pi(A)-\beta\pi(\bar C)\nu(A).
\end{align*}
A direct calculation from \eqref{eq:r-def} shows that \(\widetilde\pi\) is an
invariant probability for \(\widetilde Q\), and its \(\mathcal X\)-marginal is \(\pi\). Let \(\widetilde Q^\ast\) be the time-reversal of \(\widetilde Q\) with respect to
\(\widetilde\pi\), and set
\[
\mathcal A:=\bar C\times\{1\}.
\]
We claim that \(\mathcal A\) is an atom for \(\widetilde Q^\ast\). Indeed, for
\(E\in\mathcal B\) and \(B\in\widetilde{\mathcal B}\), \eqref{eq:split-1},
\eqref{eq:r-def}, and the time-reversal identity for \(\widetilde Q\) yield
\begin{align*}
\int_B \widetilde\pi(d(x,i))\widetilde Q\big((x,i),E\times\{1\}\big)
&= \int_B \widetilde\pi(d(x,i))\,\beta\mathbf 1_{\bar C}(x)\nu(E) \\
&= \int_{E\times\{1\}} \widetilde\pi(d(y,1))
   \widetilde Q^\ast\big((y,1),B\big).
\end{align*}
Since \(\widetilde\pi(dy,1)=\beta\pi(\bar C)\nu(dy)\), we obtain
\[
\nu(E)\int_B \mathbf 1_{\bar C}(x)\widetilde\pi(d(x,i))
=
\pi(\bar C)\int_E \nu(dy)\,\widetilde Q^\ast((y,1),B).
\]
Because this holds for every \(E\in\mathcal B\), it follows that for
\(\nu\)-a.e.\ \(y\),
\begin{equation}\label{eq:atom-kernel}
\widetilde Q^\ast((y,1),B)
=
\frac{1}{\pi(\bar C)}\int_B \mathbf 1_{\bar C}(x)\,\widetilde\pi(d(x,i)),
\qquad B\in\widetilde{\mathcal B},
\end{equation}
which is independent of \(y\). Since \(\widetilde\pi(\cdot\mid \mathcal A)\) is
proportional to \(\nu\), after modifying \(\widetilde Q^\ast\) on a
\(\widetilde\pi\)-null subset of \(\mathcal A\) we may assume that
\eqref{eq:atom-kernel} holds for all \((y,1)\in\mathcal A\). Thus \(\mathcal A\)
is an atom, hence a petite set, for \(\widetilde Q^\ast\).

Let \(\tau_{\mathcal A}:=\inf\{n\ge1:\widetilde Z_n\in\mathcal A\}\) and
\(\mu_{\mathcal A}:=\widetilde\pi(\cdot\mid \mathcal A)\). By
Lemma~\ref{lem:geom_tail}, for some \(a\in(1,a_0]\),
\begin{equation}\label{eq:exp-forward}
\E_{\mu_{\mathcal A}}^{\widetilde Q}\!\left[a^{\tau_{\mathcal A}}\right]<\infty.
\end{equation}
Now iterate the one-step reversal identity for \(\widetilde Q\): for any
\(E_0,\dots,E_n\in\widetilde{\mathcal B}\),
\[
\P_{\widetilde\pi}^{\widetilde Q}(Z_0\in E_0,\dots,Z_n\in E_n)
=
\P_{\widetilde\pi}^{\widetilde Q^\ast}(Z_0\in E_n,\dots,Z_n\in E_0).
\]
Taking \(E_0=E_n=\mathcal A\) and \(E_1=\cdots=E_{n-1}=\mathcal A^c\) gives
\[
\widetilde\pi(\mathcal A)\,
\P_{\mu_{\mathcal A}}^{\widetilde Q}(\tau_{\mathcal A}=n)
=
\widetilde\pi(\mathcal A)\,
\P_{\mu_{\mathcal A}}^{\widetilde Q^\ast}(\tau_{\mathcal A}=n),
\qquad n\ge1.
\]
Hence
\begin{equation*}
\P_{\mu_{\mathcal A}}^{\widetilde Q}(\tau_{\mathcal A}=n)
=
\P_{\mu_{\mathcal A}}^{\widetilde Q^\ast}(\tau_{\mathcal A}=n),
\qquad n\ge1,
\end{equation*}
and therefore
\begin{equation}\label{eq:exp-reverse}
\E_{\mu_{\mathcal A}}^{\widetilde Q^\ast}\!\left[a^{\tau_{\mathcal A}}\right]
<\infty.
\end{equation}
Since \(\mathcal A\) is an atom for \(\widetilde Q^\ast\), the left-hand side of
\eqref{eq:exp-reverse} is constant over \(z\in\mathcal A\); hence
\begin{equation}\label{eq:exp-uniform}
\sup_{z\in\mathcal A}
\E_{z}^{\widetilde Q^\ast}\!\left[a^{\tau_{\mathcal A}}\right]
<\infty.
\end{equation}
By Proposition~11.1.4 of \cite{douc2018markov}, \(\widetilde Q\) is
\(\psi\)-irreducible and aperiodic. Together with \eqref{eq:exp-forward},
Theorem~15.0.1 of \cite{Meyn12_book} shows that \(\widetilde Q\) is geometrically
ergodic. Applying Steps~1--2 above to \(\widetilde Q\), we conclude that
\(\widetilde Q^\ast\) is also \(\psi\)-irreducible and aperiodic. Since
\(\mathcal A\) is a petite set for \(\widetilde Q^\ast\) and
\eqref{eq:exp-uniform} holds, another application of
Theorem~15.0.1 of \cite{Meyn12_book} yields geometric ergodicity of
\(\widetilde Q^\ast\): there exist \(\gamma\in(0,1)\) and a measurable
\(\widetilde M:\widetilde{\mathcal X}\to[0,\infty)\) with
\(\widetilde\pi(\widetilde M)<\infty\) such that
\begin{equation}\label{eq:geom-tilde}
\bigl\|\widetilde Q^{\ast n}(z,\cdot)-\widetilde\pi\bigr\|_{\mathrm{TV}}
\le \widetilde M(z)\gamma^n,
\qquad z\in\widetilde{\mathcal X},\ n\ge0.
\end{equation}
Because \(\widetilde\pi\) has \(\mathcal X\)-marginal \(\pi\), there exists a
regular conditional distribution \(\Lambda(x,\cdot)\) on \(\{0,1\}\) such that
\[
\widetilde\pi(dx,di)=\pi(dx)\Lambda(x,di).
\]
Write \(\widetilde\delta_x:=\delta_x\otimes\Lambda(x,\cdot)\) and
\(\widehat B:=B\times\{0,1\}\). Combining \eqref{eq:proj-forward} with the
\(n\)-step reversal identities for \(Q\) and \(\widetilde Q\), we obtain, for every
\(n\ge0\) and \(A,B\in\mathcal B\),
\[
\int_B \pi(dy)\,Q^{\ast n}(y,A)
=
\int_{\widehat B}\widetilde\pi(dw)\,
\widetilde Q^{\ast n}(w,\widehat A).
\]
Disintegrating \(\widetilde\pi\) with respect to \(\pi\) therefore gives
\begin{equation}\label{eq:projection}
Q^{\ast n}(x,B)=\widetilde\delta_x\widetilde Q^{\ast n}(\widehat B)
\qquad\text{for \(\pi\)-a.e.\ }x.
\end{equation}
Since \(\widetilde\pi(\widehat B)=\pi(B)\), \eqref{eq:projection} and
\eqref{eq:geom-tilde} imply
\begin{align*}
\bigl\|Q^{\ast n}(x,\cdot)-\pi\bigr\|_{\mathrm{TV}}
= \sup_{B\in\mathcal B}
   \Bigl|\widetilde\delta_x\widetilde Q^{\ast n}(\widehat B)
         -\widetilde\pi(\widehat B)\Bigr| &\le \bigl\|\widetilde\delta_x\widetilde Q^{\ast n}-\widetilde\pi
    \bigr\|_{\mathrm{TV}} \\
&\le \int \widetilde\delta_x(dz)\,
      \bigl\|\widetilde Q^{\ast n}(z,\cdot)-\widetilde\pi\bigr\|_{\mathrm{TV}} \\
&\le \Bigl(\int \widetilde\delta_x(dz)\,\widetilde M(z)\Bigr)\gamma^n
=: M_Q(x)\gamma^n.
\end{align*}
Moreover, $\pi(M_Q)=\widetilde\pi(\widetilde M)<\infty.$ Thus \(Q^\ast\) is geometrically ergodic.

Finally, let \(n=qm+r\) with \(0\le r<m\). Since \(Q^\ast=(P^\ast)^m\),
\[
\delta_x(P^\ast)^n=\delta_x(Q^\ast)^q(P^\ast)^r.
\]
Using \(\pi(P^\ast)^r=\pi\) and contraction of total variation under a Markov
kernel,
\[
\bigl\|(P^\ast)^n(x,\cdot)-\pi\bigr\|_{\mathrm{TV}}
=
\bigl\|(\delta_x(Q^\ast)^q-\pi)(P^\ast)^r\bigr\|_{\mathrm{TV}}
\le
\bigl\|Q^{\ast q}(x,\cdot)-\pi\bigr\|_{\mathrm{TV}}
\le M_Q(x)\gamma^q.
\]
If \(\rho:=\gamma^{1/m}\in(0,1)\), then \(\gamma^q\le \rho^{-(m-1)}\rho^n\), so
\[
\bigl\|(P^\ast)^n(x,\cdot)-\pi\bigr\|_{\mathrm{TV}}
\le M_P(x)\rho^n,
\qquad
M_P(x):=\rho^{-(m-1)}M_Q(x).
\]
Hence \(P^\ast\) is geometrically ergodic.

\subsection{Proof of Lemma \ref{lem:forward-meeting}}\label{sec:proof-forward-meeting}
Recall that for two probability measure $\mu, \nu$ on $(\mathcal X, \mathcal B)$, 
\begin{align*}
\|\mu-\nu\|_{\operatorname{TV}}:= \sup_{A \in \mathcal B}|\mu(A)-\nu(A)| = \frac{1}{2}\sup_{\|f\|\leq 1}|\E_\mu[f] - \E_\nu[f]|.
\end{align*}
Also define the weighted $V$-norm (for $V \geq 1$):
\begin{align*}
\|\mu-\nu\|_{V}:=\sup_{\|f\|\leq V}|\E_\mu[f] - \E_\nu[f]|.
\end{align*}
Since $V \geq 1$, $\|\mu-\nu\|_{\operatorname{TV}} \leq \frac{1}{2}\|\mu-\nu\|_{V}.$ Under Assumption \ref{assumption:markovchain}, by \cite[Theorem 15.0.1]{Meyn12_book}, there exists $\rho \in (0,1)$ such that for all $x \in \mathcal X$ and all $k \geq 0$,
\begin{align*}
    \|P^k(x, \cdot) - \mu\|_{V} \lesssim V(x)\rho^k.
\end{align*}
Therefore, we have
\begin{align*}
\|P^k(x, \cdot)-P^k(y, \cdot)\|_{\operatorname{TV}} \lesssim (V(x)+V(y))\rho^k, \quad \forall x, y \in \mathcal X.
\end{align*}
Recall from \cite{goldstein1979maximal} that for any Markov kernel $P$ and any
$x,y\in\mathcal X$, one can construct a maximal  coupling of the
entire trajectories $\{(X_t^x,X_t^y)\}_{t\ge 0}$ with $X_0^x=x$ and $X_0^y=y$,
such that once the two chains meet they evolve together thereafter, and for every
$k\ge 0$,
\begin{align*}
  \P(X_k^x\neq X_k^y)
  =
  \bigl\|P^k(x,\cdot)-P^k(y,\cdot)\bigr\|_{\operatorname{TV}}.
\end{align*}
Given such a coupling, the meeting time $T^+(x,y):= \inf\{t \geq 0: X_t^x = X_t^y\}$ satisfies
\begin{align*}
\mathbb P(T^+(x,y)>k) = \mathbb P(X_k^x \neq X_k^y) = \|P^k(x, \cdot)- P^k(y, \cdot)\|_{\operatorname{TV}} \lesssim (V(x) + V(y))\rho^k, \quad \forall k \geq 0.
\end{align*}
Therefore, for any $p \geq 1$, 
\begin{align*}
    \E[(T^+(x,y))^p] &= \sum_{k=0}^{\infty}((k+1)^p - k^p)\mathbb P(T^+(x,y)>k)\\
    &\leq p\sum_{k=0}^{\infty}(k+1)^{p-1}\mathbb P(T^+(x,y)>k)\\
    &\lesssim (V(x) + V(y))\sum_{k=0}^{\infty}(k+1)^{p-1}\rho^k \lesssim (V(x) + V(y)).
\end{align*}
This completes the proof of Lemma \ref{lem:forward-meeting}.
\section{Proof of Theorem \ref{thm:MC-Wp}}\label{sec:proof-MC-Wp}
In what follows, we prove Theorem~\ref{thm:MC-Wp} by filling in the details of the proof outline
presented in Section~\ref{sec:outline-MC-Wp}.

\subsection{Preliminaries}
In this section, we collect several useful lemmas that will be used in the proof of Theorem~\ref{thm:MC-Wp}.
\begin{lemma}\label{lem:K}
Under Lemma~\ref{lem:geom_tail}, for every \(q\ge 1\) there exists
\(C_q<\infty\) such that
\[
\E\bigl|K_n-\lfloor n/\E[\Lambda_2] \rfloor\bigr|^{q}
\le C_q\, n^{q/2},
\qquad \forall n\ge 1.
\]
\end{lemma}
\begin{proof}{Proof of Lemma \ref{lem:K}}Write
\[
\mu:=\E[\Lambda_2],\qquad
k_0:=\lfloor n/\mu\rfloor,
\qquad\text{and}\qquad
N_n:=\min\{k\ge1:T_k>n\},
\]
so that $K_n=N_n+1.$ We first show that the cycle lengths admit a uniform exponential moment.
By Lemma~\ref{lem:geom_tail}, there exist constants \(b<\infty\) and
\(\rho\in(0,1)\) such that
\[
\sup_{i\ge1}\P(L_i>\ell)\le b\,\rho^\ell,
\qquad \forall \ell\ge0.
\]
Fix \(\theta_L\in(0,-\log\rho)\), and set
\[
r:=e^{\theta_L}\rho\in(0,1).
\]
For any nonnegative integer-valued random variable \(Y\), we have the identity
\[
\E[e^{\theta_L Y}]
=
1+\sum_{\ell\ge0}\bigl(e^{\theta_L(\ell+1)}-e^{\theta_L\ell}\bigr)\P(Y>\ell).
\]
Applying this with \(Y=L_i\), we obtain
\[
\E[e^{\theta_L L_i}]
=
1+(e^{\theta_L}-1)\sum_{\ell\ge0}e^{\theta_L\ell}\P(L_i>\ell)
\le
1+(e^{\theta_L}-1)b\sum_{\ell\ge0}(e^{\theta_L}\rho)^\ell.
\]
Since \(r<1\), the geometric series converges, and therefore $\sup_{i\ge1}\E[e^{\theta_L L_i}]<\infty.$ Because \(\Lambda_i=mL_i\), it follows that there exists
\(\theta_\Lambda>0\) such that
\begin{equation}\label{eq:Lambda-exp-moment}
\sup_{i\ge1}\E[e^{\theta \Lambda_i}]<\infty,
\qquad \forall \theta\in(0,\theta_\Lambda].
\end{equation}
Next, by Lemma \ref{lem:geom_tail}, the variables \[ X_i:=\Lambda_i-\mu,\qquad i\ge 1, \] are independent and sub-exponential by \eqref{eq:Lambda-exp-moment} and centered when $i \geq 2$. Therefore, Bernstein's inequality yields constants \(c,C>0\) such that, for all \(k\ge 1\) and \(u\ge0\), \[ \P\Bigl(\Bigl|\sum_{i=1}^{k}X_i\Bigr|\ge u\Bigr) \le C\exp\!\left[-c\min\!\left(\frac{u^2}{k},u\right)\right]. \] Since $T_k-k\mu = \sum_{i=1}^{k}X_i,$ \begin{equation}\label{eq:Tk_conc} \P\bigl(|T_k-k\mu|\ge u\bigr) \le C\exp\!\left[-c\min\!\left(\frac{u^2}{k},u\right)\right]. \end{equation}
We now derive a deviation bound for \(N_n\). For any integer \(r\ge2\),
\[
\{N_n\ge k_0+r\}=\{T_{k_0+r-1}\le n\}.
\]
Moreover,
\[
n-(k_0+r-1)\mu
=
(n-k_0\mu)-(r-1)\mu
\le
\mu-(r-1)\mu
=
-(r-2)\mu,
\]
and hence
\[
\{N_n\ge k_0+r\}
\subseteq
\Bigl\{
T_{k_0+r-1}-(k_0+r-1)\mu\le -(r-2)\mu
\Bigr\}.
\]
Similarly, for \(1\le r\le k_0-1\),
\[
\{N_n\le k_0-r\}=\{T_{k_0-r}>n\},
\]
and since
\[
n-(k_0-r)\mu=(n-k_0\mu)+r\mu\ge r\mu,
\]
we obtain
\[
\{N_n\le k_0-r\}
\subseteq
\Bigl\{
T_{k_0-r}-(k_0-r)\mu\ge r\mu
\Bigr\}.
\]
When \(r\ge k_0\), the event \(\{N_n\le k_0-r\}\) is empty. Applying \eqref{eq:Tk_conc} and using that
\(k_0+r-1\lesssim n+r\) and \(k_0-r\lesssim n\), we obtain constants
\(c',C'>0\) such that, for all \(r\ge1\),
\begin{equation}\label{eq:N_tail}
\P\bigl(|N_n-k_0|\ge r\bigr)
\le
C'\exp\!\left[-c'\min\!\left(\frac{r^2}{n},r\right)\right].
\end{equation}
Finally, by the tail-integral formula,
\[
\E|N_n-k_0|^q
=
q\int_0^\infty t^{q-1}\P(|N_n-k_0|\ge t)\,dt.
\]
Using \eqref{eq:N_tail}, we get
\[
\E|N_n-k_0|^q
\le
qC'\int_0^\infty t^{q-1}
\exp\!\left[-c'\min\!\left(\frac{t^2}{n},t\right)\right]dt.
\]
We split the integral at \(t=n\). For \(0\le t\le n\), we have
\(\min(t^2/n,t)=t^2/n\), and the change of variables \(t=\sqrt n\,x\) yields
\[
\int_0^n t^{q-1}e^{-c't^2/n}\,dt
\lesssim
n^{q/2}\int_0^\infty x^{q-1}e^{-c'x^2}\,dx
\lesssim
n^{q/2}.
\]
For \(t\ge n\), we have \(\min(t^2/n,t)=t\), so $\int_n^\infty t^{q-1}e^{-c't}dt<\infty.$ Thus 
\[
\E|N_n-k_0|^q\le \widetilde C_q n^{q/2}
\]
for some constant \(\widetilde C_q<\infty\). Since \(K_n=N_n+1\), we have
\[
|K_n-k_0|^q\le 2^{q-1}\bigl(|N_n-k_0|^q+1\bigr),
\]
and therefore
\[
\E|K_n-k_0|^q
\le
2^{q-1}\bigl(\E|N_n-k_0|^q+1\bigr)
\le
C_q\, n^{q/2}.
\]
This completes the proof.
\end{proof}
\begin{lemma}\label{lem:mds-tildeMi-proof}
Let \((\Omega,\mathcal F,\mathbb P)\) be a probability space, and let
\((\mathcal F_t)_{t\ge 0}\) be a filtration. Suppose that
\((M_t)_{t\ge 0}\) is an \(\R^d\)-valued martingale with respect to
\((\mathcal F_t)\), with \(M_0=0\). Let \((T_i)_{i\ge 0}\) be an increasing
sequence of stopping times such that
\[
0=T_0<T_1<T_2<\cdots
\qquad \text{a.s.},
\]
and define $\mathcal G_i:=\mathcal F_{T_i}$ for any $i\ge 0.$ For each \(i\ge 1\), let $\widetilde M_i:=M_{T_i}-M_{T_{i-1}}.$ Assume that there exists \(q>1\) such that
\begin{equation}\label{eq:UI-assumption}
\sup_{t\ge 0}\E\|M_{t\wedge T_i}\|^q<\infty,
\qquad \text{for every fixed } i\ge 1.
\end{equation}
Then \((\widetilde M_i,\mathcal G_i)_{i\ge 1}\) is a martingale difference
sequence, i.e.,
\[
\E[\widetilde M_i\mid \mathcal G_{i-1}]=0,
\qquad i\ge 1,
\]
and, in particular,
\[
\E[\widetilde M_i]=0,
\qquad i\ge 1.
\]
\end{lemma}

\begin{proof}{Proof of Lemma~\ref{lem:mds-tildeMi-proof}}
Fix \(i\ge 1\), and define the stopped process
\[
N_t:=M_{t\wedge T_i},\qquad t\ge 0.
\]
Since \(M\) is a martingale and \(T_i\) is a stopping time, \((N_t)_{t\ge 0}\)
is again a martingale. By \eqref{eq:UI-assumption} and the fact that \(q>1\),
the family \(\{N_t:t\ge 0\}\) is uniformly integrable. Therefore, the optional
sampling theorem for uniformly integrable martingales applies to the stopping
times \(T_{i-1}\le T_i\) and yields
\[
\E\!\left[M_{T_i}\mid \mathcal F_{T_{i-1}}\right]
=
\E\!\left[N_{T_i}\mid \mathcal F_{T_{i-1}}\right]
=
N_{T_{i-1}}
=
M_{T_{i-1}}.
\]
Recalling that \(\mathcal G_{i-1}=\mathcal F_{T_{i-1}}\), we obtain
\[
\E[\widetilde M_i\mid \mathcal G_{i-1}]
=
\E[M_{T_i}-M_{T_{i-1}}\mid \mathcal G_{i-1}]
=0.
\]
Taking expectations gives \(\E[\widetilde M_i]=0\).
\end{proof}

\begin{lemma}[Theorem~4.1 in \cite{pinelis1994optimum}]
\label{lem:BR_Hilbert}
Let \((\Omega,\mathcal F,\mathbb P)\) be a probability space equipped with a
filtration \((\mathcal F_k)_{k\ge 0}\). Let \(H\) be a real separable Hilbert
space with norm \(\|\cdot\|\). Let \((d_k)_{k\ge 1}\) be an \(H\)-valued
martingale difference sequence, i.e., for each \(k\ge 1\),
\(d_k\) is \(\mathcal F_k\)-measurable, \(\E\|d_k\|<\infty\), and
\[
\E[d_k\mid \mathcal F_{k-1}]=0
\qquad \text{a.s.}
\]
Define
\[
M_n:=\sum_{k=1}^n d_k,\qquad n\ge 1,
\qquad\text{and}\qquad
M_0:=0.
\]
Then, for every \(p\ge 2\), there exists a constant \(C_p\in(0,\infty)\),
depending only on \(p\), such that for all integers \(n\ge 1\),
\[
\E\|M_n\|^p
\le
C_p\left\{
\E\Bigl(\sum_{k=1}^n \E[\|d_k\|^2\mid \mathcal F_{k-1}]\Bigr)^{p/2}
+
\E\sum_{k=1}^n \|d_k\|^p
\right\}.
\]
\end{lemma}

The previous lemma immediately yields the following martingale-transform bound.

\begin{lemma}\label{lem:transform_BR}
Under the assumptions and notation of Lemma~\ref{lem:BR_Hilbert}, let
\((a_k)_{k\ge 1}\) be a real-valued predictable sequence, i.e.,
\(a_k\) is \(\mathcal F_{k-1}\)-measurable for each \(k\ge 1\). Define
\[
T_n:=\sum_{k=1}^n a_k d_k,\qquad n\ge 1.
\]
Then, for every \(p\ge 2\), the same constant \(C_p\) as in
Lemma~\ref{lem:BR_Hilbert} satisfies
\begin{equation}\label{eq:transform_BR}
\E\|T_n\|^p
\le
C_p\left\{
\E\Bigl(\sum_{k=1}^n a_k^2\,\E[\|d_k\|^2\mid \mathcal F_{k-1}]\Bigr)^{p/2}
+
\E\sum_{k=1}^n |a_k|^p\,\|d_k\|^p
\right\},
\qquad n\ge 1.
\end{equation}
In particular, if \(\tau\) is a stopping time taking values in
\(\{0,1,\dots,n\}\), then $a_k:=\mathbf 1_{\{k\le \tau\}}$ is predictable, and $T_n=\sum_{k=1}^{\tau} d_k = M_\tau.$ Consequently,
\[
\E\|M_{\tau}\|^p
\le
C_p\left\{
\E\Bigl(\sum_{k=1}^{\tau}\E[\|d_k\|^2\mid \mathcal F_{k-1}]\Bigr)^{p/2}
+
\E\sum_{k=1}^{\tau}\|d_k\|^p
\right\}.
\]
\end{lemma}

\subsection{Main Proof}\label{sec:proof-covariance-mismatch}

The first step relies on two auxiliary lemmas, whose proofs are deferred to
Appendices~\ref{sec:geom_tail} and~\ref{sec:iid_Mr}. We now provide the details
for the remaining two steps.

\medskip
\noindent\textbf{Details for Step~2.}
For \(i \ge 1\), define
\[
\widetilde M_i^{(o)} := M_{T_{2i-1}} - M_{T_{2i-2}},
\qquad
\widetilde M_i^{(e)} := M_{T_{2i}} - M_{T_{2i-1}},
\]
and let
\[
\mathcal G_i^{(o)} := \mathcal F_{T_{2i-1}},
\qquad
\mathcal G_i^{(e)} := \mathcal F_{T_{2i}},
\qquad i \ge 1.
\]
For completeness, we also set
\[
\mathcal G_0^{(o)} := \mathcal F_{T_0},
\qquad
\mathcal G_0^{(e)} := \mathcal F_{T_0}.
\]
By Lemma~\ref{lem:iid_Mr}, the sequence \(\{\widetilde M_i\}_{i\ge1}\) is
\(1\)-dependent and centered. In addition, for each \(i\ge1\),
\(\widetilde M_i^{(o)}\) is independent of \(\mathcal G_{i-1}^{(o)}\), and
\(\widetilde M_i^{(e)}\) is independent of \(\mathcal G_{i-1}^{(e)}\). Hence,
\[
\E\!\left[\widetilde M_i^{(o)} \mid \mathcal G_{i-1}^{(o)}\right]=0,
\qquad
\E\!\left[\widetilde M_i^{(e)} \mid \mathcal G_{i-1}^{(e)}\right]=0,
\qquad i\ge1,
\]
so both
\[
\bigl(\widetilde M_i^{(o)}, \mathcal G_i^{(o)}\bigr)_{i\ge1}
\qquad\text{and}\qquad
\bigl(\widetilde M_i^{(e)}, \mathcal G_i^{(e)}\bigr)_{i\ge1}
\]
are martingale difference sequences. Moreover,
\begin{equation}\label{eq:uniformconditional}
\sup_{i\ge1}\E\bigl[\|\widetilde M_i^{(o)}\|^p \mid \mathcal G_{i-1}^{(o)}\bigr]
< \infty,
\qquad
\sup_{i\ge1}\E\bigl[\|\widetilde M_i^{(e)}\|^p \mid \mathcal G_{i-1}^{(e)}\bigr]
< \infty.
\end{equation}
Next, define
\[
\tau_n^{(o)} := \Big\lfloor \frac{K_n+1}{2} \Big\rfloor,
\qquad
\tau_n^{(e)} := \Big\lfloor \frac{K_n}{2} \Big\rfloor,
\qquad
m_n^{(o)} := \Big\lfloor \frac{k_n+1}{2} \Big\rfloor,
\qquad
m_n^{(e)} := \Big\lfloor \frac{k_n}{2} \Big\rfloor.
\]
Then
\[
\widetilde S_n^{(o)} = \sum_{i=1}^{\tau_n^{(o)}} \widetilde M_i^{(o)},
\qquad
\widetilde S_n^{(e)} = \sum_{i=1}^{\tau_n^{(e)}} \widetilde M_i^{(e)},
\qquad
\widetilde S_n = \widetilde S_n^{(o)} + \widetilde S_n^{(e)},
\]
and similarly
\[
\bar S_n^{(o)} = \sum_{i=1}^{m_n^{(o)}} \widetilde M_i^{(o)},
\qquad
\bar S_n^{(e)} = \sum_{i=1}^{m_n^{(e)}} \widetilde M_i^{(e)},
\qquad
\bar S_n = \bar S_n^{(o)} + \bar S_n^{(e)}.
\]
Introduce the coefficients
\[
c_i^{(o)} := \mathbf{1}_{\{i \le \tau_n^{(o)}\}} - \mathbf{1}_{\{i \le m_n^{(o)}\}},
\qquad
c_i^{(e)} := \mathbf{1}_{\{i \le \tau_n^{(e)}\}} - \mathbf{1}_{\{i \le m_n^{(e)}\}},
\qquad i \ge 1.
\]
Then \(c_i^{(o)}, c_i^{(e)} \in \{-1,0,1\}\) for all \(i \ge 1\), and
\[
\widetilde S_n - \bar S_n
=
\bigl(\widetilde S_n^{(o)} - \bar S_n^{(o)}\bigr)
+
\bigl(\widetilde S_n^{(e)} - \bar S_n^{(e)}\bigr).
\]
Hence, by the triangle inequality,
\[
\|\widetilde S_n - \bar S_n\|_{L^p}
\le
\|\widetilde S_n^{(o)} - \bar S_n^{(o)}\|_{L^p}
+
\|\widetilde S_n^{(e)} - \bar S_n^{(e)}\|_{L^p}.
\]
Applying the Rosenthal--Burkholder inequality
(Lemmas~\ref{lem:BR_Hilbert} and~\ref{lem:transform_BR}) together with
\eqref{eq:uniformconditional}, we obtain
\[
\bigl\|\widetilde S_n^{(o)} - \bar S_n^{(o)}\bigr\|_{L^p}^p
=
\Big\|\sum_{i\ge1} c_i^{(o)} \widetilde M_i^{(o)}\Big\|_{L^p}^p
\lesssim
\E\bigl| \tau_n^{(o)} - m_n^{(o)} \bigr|^{p/2}
+
\E\bigl| \tau_n^{(o)} - m_n^{(o)} \bigr|.
\]
Since
\[
\bigl| \tau_n^{(o)} - m_n^{(o)} \bigr| \le |K_n-k_n|+1,
\]
Lemma~\ref{lem:K} with \(q=p/2\) and \(q=1\) yields
\[
\E|K_n-k_n|^{p/2} = \mathcal O\bigl(n^{p/4}\bigr),
\qquad
\E|K_n-k_n| = \mathcal O(\sqrt n),
\]
and therefore $\|\widetilde S_n^{(o)} - \bar S_n^{(o)}\|_{L^p}
= \mathcal O\bigl(n^{1/4}\bigr).$ The same argument applies to the even part, so $\|\widetilde S_n^{(e)} - \bar S_n^{(e)}\|_{L^p}
= \mathcal O\bigl(n^{1/4}\bigr).$ Consequently,
\[
\|\widetilde S_n - \bar S_n\|_{L^p}
= \mathcal O\bigl(n^{1/4}\bigr),
\qquad
\W_p\!\left(
\law\!\left(\widetilde S_n/\sqrt{n}\right),
\law\!\left(\bar S_n/\sqrt{n}\right)
\right)
= \mathcal O\bigl(n^{-1/4}\bigr).
\]

\medskip
\noindent\textbf{Details for Step~3.}
By Assumption~\ref{assumption:markovchain} and the martingale decomposition
\eqref{eq:Sn-mart-tv},
\begin{align*}
\E\|S_n\|^2
=
\E\|g(x_0)-g(x_n)+M_n\|^2
\lesssim \E\|g(x_0)-g(x_n)\|^2 + \E\|M_n\|^2 \lesssim 1 + \sum_{t=0}^{n-1}\E\|\xi_{t+1}\|^2 \lesssim n,
\end{align*}
where we used \(\sup_{t\ge0}\|g(x_t)\|_{L^2}<\infty\) and the orthogonality of
martingale differences:
\[
\E\|M_n\|^2 = \sum_{t=0}^{n-1}\E\|\xi_{t+1}\|^2.
\]
Similarly, by Lemma~\ref{lem:transform_BR} with \(p=2\) and
Lemma~\ref{lem:K},
\[
\E\|\widetilde S_n\|^2
\lesssim
\E\|\widetilde S_n^{(o)}\|^2 + \E\|\widetilde S_n^{(e)}\|^2
\lesssim
\E[K_n]
\lesssim n,
\]
where we used the martingale difference structure of the odd and even
subsequences.

Now set
\[
A_n := S_n - \E[S_n],
\qquad
B_n := \widetilde S_n - \E[\widetilde S_n].
\]
Then
\[
A_nA_n^\top - B_nB_n^\top
=
(A_n-B_n)A_n^\top + B_n(A_n-B_n)^\top.
\]
Therefore, using \eqref{eq:triangle-reminder} and the Cauchy--Schwarz
inequality,
\begin{align*}
\|\operatorname{Var}(S_n) - \operatorname{Var}(\widetilde S_n)\|
=
\bigl\|\E[A_nA_n^\top] - \E[B_nB_n^\top]\bigr\| &\le
\bigl\|\E[(A_n-B_n)A_n^\top]\bigr\|
+
\bigl\|\E[B_n(A_n-B_n)^\top]\bigr\| \\
&\lesssim
\|A_n-B_n\|_{L^2}\bigl(\|A_n\|_{L^2}+\|B_n\|_{L^2}\bigr) \\
&\lesssim
\|S_n-\widetilde S_n\|_{L^2}
\bigl(\|S_n\|_{L^2}+\|\widetilde S_n\|_{L^2}\bigr) \lesssim \sqrt n.
\end{align*}
It follows that
\[
\left\|
\Sigma_n - \frac{1}{n}\operatorname{Var}(\widetilde S_n)
\right\|
= \mathcal O\bigl(n^{-1/2}\bigr).
\]
Next, observe that
\[
\widetilde S_n
=
\sum_{i=1}^{K_n}\widetilde M_i
=
\sum_{i=1}^{\infty}\mathbf{1}_{\{K_n\ge i\}}\widetilde M_i.
\]
Since the event \(\{K_n \ge i\}\) depends only on the past up to time
\(T_{i-2}\), equation~\eqref{eq:martingaledifference} implies that
\[
\operatorname{Var}(\widetilde S_n)
=
\E\!\left[\sum_{i=1}^{K_n}\widetilde M_i\widetilde M_i^\top\right].
\]
Hence,
\begin{align*}
\|\operatorname{Var}(\widetilde S_n)-\operatorname{Var}(\bar S_n)\|
=
\left\|
\E\!\left[
\sum_{i=1}^{K_n}\widetilde M_i\widetilde M_i^\top
-
\sum_{i=1}^{k_n}\widetilde M_i\widetilde M_i^\top
\right]
\right\| &\le
\|\widetilde S_n^{(o)}-\bar S_n^{(o)}\|_{L^2}^2
+
\|\widetilde S_n^{(e)}-\bar S_n^{(e)}\|_{L^2}^2 \\
&\lesssim \E|K_n-k_n|
\lesssim \sqrt n,
\end{align*}
where in the last step we again used Lemma~\ref{lem:K}. Consequently,
\[
\left\|
\frac{1}{n}\operatorname{Var}(\widetilde S_n)
-
\frac{1}{n}\operatorname{Var}(\bar S_n)
\right\|
=
\mathcal O\bigl(n^{-1/2}\bigr).
\]
Combining this estimate with the bound
\[
\left\|
\Sigma_n - \frac{1}{n}\operatorname{Var}(\widetilde S_n)
\right\|
=
\mathcal O\bigl(n^{-1/2}\bigr)
\]
proved above yields \eqref{eq:cov-mismatch}, and hence completes the proof of
Theorem~\ref{thm:MC-Wp}.

\subsection{Proof of Lemma \ref{lem:geom_tail}}\label{sec:geom_tail}
We first record two auxiliary lemmas, whose proofs are deferred to Sections~\ref{sec:returntime-m} and~\ref{sec:entrance-mgf}.
\begin{lemma}\label{lem:returntime-m}
Under Assumption~\ref{assumption:markovchain} and conditions \eqref{eq:return}
and \eqref{eq:M}, there exists \(\bar\kappa>1\) such that
\[
\bar M(a):=\sup_{x\in\bar C}\E_x\!\left[a^{\sigma_{\bar C,m}}\right]<\infty,
\qquad a\in(1,\bar\kappa],
\]
where
\[
\sigma_{\bar C,m}:=\inf\{t\ge1:x_{tm}\in\bar C\}.
\]
\end{lemma}

\begin{lemma}\label{lem:entrance-mgf}
Let \(\{x^{(m)}_t\}_{t\ge0}\) be the \(m\)-skeleton. Under the geometric-drift
condition for \(P\) with petite set \(C\) and Lyapunov function \(V\ge1\), there
exist constants \(s_0>1\) and \(K_{\mathrm{in}}<\infty\) such that, for
\[
T_{\bar C}:=\inf\{t\ge0:x_{tm}\in\bar C\},
\]
we have
\[
\E\!\left[s^{T_{\bar C}}\right]\le K_{\mathrm{in}}\E[V(x_0)],
\qquad 1<s\le s_0.
\]
\end{lemma}

Let \(Q:=P^m\), and let \(\Phi=(X_k,Y_k)_{k\ge0}\) be the split chain for \(Q\) with
atom $\mathcal A:=\bar C\times\{1\}.$

Rewrite
\[
r_1:=\inf\{k\ge0:(X_k,Y_k)\in\mathcal A\},\qquad
r_i:=\inf\{k>r_{i-1}:(X_k,Y_k)\in\mathcal A\},\quad i\ge2.
\]
Then, the cycle lengths satisfy
\[
L_1=r_1+1,\qquad L_i=r_i-r_{i-1},\quad i\ge2.
\]
By the strong Markov property at the atom, \(\{L_i\}_{i\ge1}\) are independent and
\(\{L_i\}_{i\ge2}\) are i.i.d.

By Lemma~\ref{lem:returntime-m}, we have \(\bar M(a)\downarrow 1\) as \(a\downarrow 1\), and hence \(H(a)\downarrow 1\) as well. Therefore, we may choose \(a>1\) sufficiently close to \(1\) such that
\begin{equation}\label{eq:a}
a\in(1,\min\{s_0,\bar\kappa\}),
\qquad
(1-\beta)\bar M(a)<1.
\end{equation}
\noindent
\textbf{Tail for \(i\ge2\).}
Fix \(i\ge2\). At time \(r_{i-1}+1\), the regeneration state
\(X_{r_{i-1}+1}\) has law \(\nu\), hence lies in \(\bar C\) almost surely. Let
\(J_i\) be the number of visits of the skeleton to \(\bar C\), starting from time
\(r_{i-1}+1\), until the first successful split coin. Then
\[
J_i\sim\mathrm{Geom}(\beta)\qquad\text{on }\{1,2,\dots\}.
\]
If \(\eta_{i,1},\eta_{i,2},\dots\) denote the successive failed excursion lengths,
then
\[
L_i=1+\sum_{j=1}^{J_i-1}\eta_{i,j}.
\]
Consequently, using \eqref{eq:a} and Lemma \ref{lem:returntime-m},
\[
\E[a^{L_i}]
 \leq a \mathbb{E}\left[(\bar{M}(a))^{J_i-1}\right]= \frac{a\beta}{1-(1-\beta) \bar{M}(a)}<\infty
\qquad i\ge2.
\]
Hence, by Markov's inequality,
\[
\P(L_i>\ell)\le \frac{a\beta}{1-(1-\beta) \bar{M}(a)} a^{-\ell},
\qquad \ell\ge0,\ \ i\ge2.
\]

\medskip
\noindent
\textbf{Tail for \(i=1\).}
After time \(T_{\bar C}\), the remainder of the first cycle has the same structure
as above. More precisely, if \(J\sim\mathrm{Geom}(\beta)\) is the number of visits
to \(\bar C\) until the first successful split coin, and \(\eta_1,\eta_2,\dots\)
are the corresponding failed excursion lengths, then
\[
L_1=1+ T_{\bar C}+\sum_{j=1}^{J-1}\eta_j.
\]
By Lemma~\ref{lem:entrance-mgf},
\[
\E[a^{T_{\bar C}}]\le K_{\mathrm{in}}\E[V(x_0)].
\]
Conditioning on \(\mathcal F_{T_{\bar C}}\) and using the same estimate as above,
we obtain
\[
\E[a^{L_1}]
\le \E[a^{T_{\bar C}}]\,F(a)
\le \frac{a\beta K_{\mathrm{in}}\E[V(x_0)]}{1-(1-\beta) \bar{M}(a)}.
\]
Therefore,
\[
\P(L_1>\ell)\le \frac{a\beta K_{\mathrm{in}}\E[V(x_0)]}{1-(1-\beta) \bar{M}(a)} a^{-\ell},
\qquad \ell\ge0.
\]
Finally, set
\[
\rho:=a^{-1}\in(0,1),
\qquad
b:=\max\Bigl\{\frac{a\beta }{1-(1-\beta) \bar{M}(a)},\frac{a\beta K_{\mathrm{in}}\E[V(x_0)]}{1-(1-\beta) \bar{M}(a)}\Bigr\}.
\]
Then
\[
\P(L_i>\ell)\le b\,\rho^\ell,
\qquad i\ge1,\ \ell\ge0.
\]
This completes the proof of Lemma \ref{lem:geom_tail}.
\subsubsection{Proof of Lemma \ref{lem:returntime-m}}\label{sec:returntime-m}
Choose any \(b\in(1,\kappa]\), and let $M(b):=\sup_{x\in\bar C}\E_x\!\left[b^{\sigma_{\bar C}}\right]<\infty.$ Let \((\mathcal F_t)_{t\ge0}\) be the natural filtration, and define the successive
ordinary return times to \(\bar C\) by
\[
\tau_0:=0,
\qquad
\tau_j:=\inf\{n>\tau_{j-1}:x_n\in\bar C\},
\qquad j\ge1.
\]
Thus \(\tau_1=\sigma_{\bar C}\). By the strong
Markov property, we have
\begin{equation}\label{eq:tauj-moment-clean}
\sup_{x\in\bar C}\E_x\!\left[b^{\tau_j}\right]\le M(b)^j,
\qquad j\ge0.
\end{equation}
Next, Lemma~\ref{lem:small-set-return}, together with Theorem~15.1.5 and Corollary~15.1.4 of \cite{douc2018markov}, yields constants \(\varsigma<\infty\) and \(\varrho\in(0,1)\) such that
\[
\sup_{x\in\bar C}\bigl|P^n(x,\bar C)-\pi(\bar C)\bigr|
\le \varsigma \varrho^n,
\qquad n\ge0.
\]
Since \(\bar C\) is accessible, \(\pi(\bar C)>0\). Choose \(n_0\ge1\) so large that $\varsigma \varrho^{n_0}\le \frac{\pi(\bar C)}{2},$ and set $\varepsilon:=\frac{\pi(\bar C)}{2}>0.$ Then, for every \(n\ge n_0\),
\begin{equation}\label{eq:uniform-large-n-clean}
\inf_{x\in\bar C}P^n(x,\bar C)
\ge \pi(\bar C)-\varsigma\varrho^n
\ge \varepsilon.
\end{equation}
Set $R:=n_0+m-1.$ For each residue class \(r\in\{0,1,\dots,m-1\}\), let \(\ell(r)\) be the unique
integer in \(\{n_0,\dots,R\}\) such that $r+\ell(r)\equiv 0\pmod m.$ For \(j\ge0\), define $A_j:=\{x_{\tau_j+\ell(\tau_j\bmod m)}\in\bar C\}.$
Since \(\ell(r)\in\{n_0,\dots,R\}\), the strong Markov property at \(\tau_j\) and
\eqref{eq:uniform-large-n-clean} imply
\begin{align}
\P_x(A_j\mid\mathcal F_{\tau_j})
=
\sum_{r=0}^{m-1}\mathbf 1_{\{\tau_j\equiv r\ (\mathrm{mod}\ m)\}}
\,P^{\ell(r)}(x_{\tau_j},\bar C) \ge \varepsilon
\qquad\text{a.s.} \label{eq:Aj-prob-clean}
\end{align}
Moreover,
\[
A_j
=
\bigcup_{r=0}^{m-1}
\Bigl(\{\tau_j\equiv r\ (\mathrm{mod}\ m)\}\cap
\{x_{\tau_j+\ell(r)}\in\bar C\}\Bigr)
\in \mathcal F_{\tau_j+R}.
\]
Since each increment \(\tau_{k+1}-\tau_k\ge1\), we have \(\tau_{j+R}\ge \tau_j+R\),
and hence
\begin{equation}\label{eq:Aj-measurable-clean}
A_j\in \mathcal F_{\tau_{j+R}}.
\end{equation}
Now look only at every \(R\)-th ordinary return. Define
\[
B_q:=A_{qR},
\qquad
K:=\inf\{q\ge0:B_q\text{ occurs}\}.
\]
By \eqref{eq:Aj-measurable-clean}, \(B_q\in\mathcal F_{\tau_{(q+1)R}}\), so
\[
\{K\ge q\}=\bigcap_{i=0}^{q-1}B_i^c\in \mathcal F_{\tau_{qR}}.
\]
Using \eqref{eq:Aj-prob-clean} with \(j=qR\), we get
\[
\P_x(B_q^c\mid\mathcal F_{\tau_{qR}})\le 1-\varepsilon
\qquad\text{a.s.}
\]
Therefore,
\[
\P_x(K\ge q+1)
=
\E_x\!\left[\mathbf 1_{\{K\ge q\}}\mathbf 1_{B_q^c}\right]
\le (1-\varepsilon)\P_x(K\ge q),
\]
and hence, by induction,
\begin{equation}\label{eq:K-tail-clean}
\P_x(K\ge q)\le (1-\varepsilon)^q,
\qquad q\ge0.
\end{equation}
If \(K=q\), then \(B_q=A_{qR}\) occurs, so
\[
x_{\tau_{qR}+\ell(\tau_{qR}\bmod m)}\in\bar C,
\qquad
\tau_{qR}+\ell(\tau_{qR}\bmod m)\equiv 0\pmod m.
\]
Thus \(\tau_{qR}+\ell(\tau_{qR}\bmod m)\) is a positive multiple of \(m\) at which
the chain is in \(\bar C\), and therefore
\[
m\,\sigma_{\bar C,m}\le \tau_{qR}+\ell(\tau_{qR}\bmod m)\le \tau_{qR}+R.
\]
Since this holds on \(\{K=q\}\) for every \(q\ge0\), we conclude that
\begin{equation}\label{eq:sigma-vs-tauKR-clean}
m\,\sigma_{\bar C,m}\le \tau_{KR}+R
\qquad\text{a.s.}
\end{equation}
Finally, choose \(p>1\) so large that
\[
\theta:=M(b)^{R/p}(1-\varepsilon)^{1-1/p}<1,
\]
and define $\bar\kappa:=b^{m/p}>1.$ By \eqref{eq:sigma-vs-tauKR-clean}, $\bar\kappa^{\sigma_{\bar C,m}}
=
b^{m\sigma_{\bar C,m}/p}
\le b^{R/p}b^{\tau_{KR}/p}.$ Hence, for every \(x\in\bar C\),
\begin{align*}
\E_x\!\left[\bar\kappa^{\sigma_{\bar C,m}}\right]
\le
b^{R/p}\sum_{q=0}^\infty
\E_x\!\left[b^{\tau_{qR}/p}\mathbf 1_{\{K=q\}}\right] \le
b^{R/p}\sum_{q=0}^\infty
\E_x\!\left[b^{\tau_{qR}/p}\mathbf 1_{\{K\ge q\}}\right].
\end{align*}
Applying H\"older's inequality with exponents \(p\) and \(p/(p-1)\), and then
using \eqref{eq:tauj-moment-clean} and \eqref{eq:K-tail-clean}, we obtain
\begin{align*}
\E_x\!\left[b^{\tau_{qR}/p}\mathbf 1_{\{K\ge q\}}\right]
&\le
\left(\E_x\!\left[b^{\tau_{qR}}\right]\right)^{1/p}
\left(\P_x(K\ge q)\right)^{1-1/p} \le
M(b)^{qR/p}(1-\varepsilon)^{q(1-1/p)}
=
\theta^q.
\end{align*}
Therefore,
\[
\sup_{x\in\bar C}\E_x\!\left[\bar\kappa^{\sigma_{\bar C,m}}\right]
\le
b^{R/p}\sum_{q=0}^\infty \theta^q
<\infty.
\]
Since \(a^{\sigma_{\bar C,m}}\le \bar\kappa^{\sigma_{\bar C,m}}\) whenever
\(1<a\le \bar\kappa\), it follows that
\[
\sup_{x\in\bar C}\E_x\!\left[a^{\sigma_{\bar C,m}}\right]<\infty,
\qquad 1<a\le \bar\kappa.
\]
This proves Lemma \ref{lem:returntime-m}.
\subsubsection{Proof of Lemma \ref{lem:entrance-mgf}}\label{sec:entrance-mgf}
Let $Q:=P^m$ and $X_t:=x_{tm}$ for $t\ge0$. Then, 
\[
T_{\bar C}=\inf\{t\ge0:X_t\in\bar C\}.
\]
Since
\[
QV=P^mV\le \lambda^mV+b_m,
\qquad
b_m:=L\sum_{j=0}^{m-1}\lambda^j\le \frac{L}{1-\lambda},
\]
we may choose \(\rho\in(\lambda^m,1)\) and \(R>0\) such that
\[
QV\le \rho V \quad \text{on}\quad D^c,
\qquad
D:=\{V\le R\}.
\]
Let $T_D:=\inf\{t\ge0:X_t\in D\}.$ Then
\[
M_t:=\rho^{-(t\wedge T_D)}V(X_{t\wedge T_D}),\qquad t\ge0,
\]
is a nonnegative supermartingale. Hence, for every \(x\in\mathcal X\) and \(t\ge0\),
\[
\P_x(T_D>t)\le \rho^t V(x).
\]
Therefore, for every \(1<s<\rho^{-1}\),
\begin{equation}\label{eq:TD-mgf}
\E_x[s^{T_D}]
=
1+(s-1)\sum_{t=0}^\infty s^t\P_x(T_D>t)
\le \frac{s}{1-s\rho}V(x).
\end{equation}
Since \(Q\) is \(V\)-uniformly geometrically ergodic with invariant measure \(\pi\),
there exist constants \(C<\infty\) and \(\alpha\in(0,1)\) such that 
\[
\|Q^n(x,\cdot)-\pi\|_V\le C V(x)\alpha^n,
\qquad x\in\mathcal X,\ n\ge0.
\]
Because \(V\ge1\), this implies
\[
\|Q^n(x,\cdot)-\pi\|_{\mathrm{TV}}\le C V(x)\alpha^n.
\]
Since \(\bar C\) is accessible, \(\pi(\bar C)>0\). Choose \(r\ge1\) so large that $CR\alpha^r\le \frac{\pi(\bar C)}{2}.$ Then, for every \(x\in D\),
\begin{equation}\label{eq:Qr-hit-short}
Q^r(x,\bar C)
\ge \pi(\bar C)-\|Q^r(x,\cdot)-\pi\|_{\mathrm{TV}}
\ge \frac{\pi(\bar C)}{2}
=:p>0.
\end{equation}
Now define
\[
G(s):=\sup_{x\in D}\E_x[s^{T_{\bar C}}],
\qquad 1<s<\rho^{-1}.
\]
For \(x\in D\), the pathwise bounds
\[
T_{\bar C}\le r+\mathbf 1_{\{X_r\notin\bar C\}}\,T_{\bar C}\circ\theta_r,
\qquad
T_{\bar C}\le T_D+T_{\bar C}\circ\theta_{T_D},
\]
together with the strong Markov property, yield
\[
\E_x[s^{T_{\bar C}}]
\le
s^r+s^r\E_x\!\left[
\mathbf 1_{\{X_r\notin\bar C\}}\E_{X_r}[s^{T_D}]
\right]G(s).
\]
Set $B_r:=\sup_{x\in D}Q^rV(x).$ By iterating \(QV\le \lambda^mV+b_m\),
\[
B_r\le \lambda^{mr}R+b_m\sum_{j=0}^{r-1}\lambda^{mj}<\infty.
\]
Using \eqref{eq:TD-mgf} and \eqref{eq:Qr-hit-short}, we obtain
\[
\sup_{x\in D}
\E_x\!\left[
\mathbf 1_{\{X_r\notin\bar C\}}\E_{X_r}[s^{T_D}]
\right]
\le
1-p+\frac{s-1}{1-s\rho}B_r
=:c(s).
\]
Hence
\[
G(s)\le s^r+s^rc(s)\,G(s).
\]
Since \(c(s)\to 1-p<1\) as \(s\downarrow1\), we may choose
\(s_0\in(1,\rho^{-1})\) so close to \(1\) that $\eta:=\sup_{1<s\le s_0}s^rc(s)<1.$ Then, for every \(1<s\le s_0\),
\[
G(s)\le \frac{s^r}{1-s^rc(s)}
\le \frac{s_0^r}{1-\eta}
=:G_0<\infty.
\]
Finally, for arbitrary \(x\in\mathcal X\), the strong Markov property at \(T_D\)
and \eqref{eq:TD-mgf} give, for \(1<s\le s_0\),
\[
\E_x[s^{T_{\bar C}}]
\le \E_x[s^{T_D}]\,G(s)
\le \frac{s}{1-s\rho}V(x)\,G_0
\le \frac{s_0}{1-s_0\rho}G_0\,V(x).
\]
Thus, with $K_{\mathrm{in}}:=\frac{s_0}{1-s_0\rho}G_0<\infty,$ we have
\[
\E_x[s^{T_{\bar C}}]\le K_{\mathrm{in}}V(x),
\qquad x\in\mathcal X,\ \ 1<s\le s_0.
\]
Integrating over the initial law of \(x_0\) yields
\[
\E[s^{T_{\bar C}}]\le K_{\mathrm{in}}\E[V(x_0)],
\qquad 1<s\le s_0,
\]
which proves Lemma \ref{lem:entrance-mgf}.
\subsection{Proof of Lemma~\ref{lem:iid_Mr}}\label{sec:iid_Mr}
Recall that 
\[
g:=\sum_{k=0}^\infty P^k h
\qquad\text{satisfies}\qquad
\|g(x)\|\lesssim V(x)^{1/r},
\]
where \(r>2\). Define
\[
H(x,x'):=g(x')-Pg(x).
\]
Then
\begin{equation}\label{eq:H-bound}
\|H(x,x')\|
\le \|g(x')\|+\|Pg(x)\|
\lesssim V(x')^{1/r}+V(x)^{1/r},
\end{equation}
and hence
\begin{equation}\label{eq:H-r-bound}
\|H(x,x')\|^r \lesssim V(x)+V(x').
\end{equation}
For \(i\ge1\) and \(j\ge0\), let
\[
B_{i,j}:=T_{i-1}+jm,
\]
so that the \(i\)-th cycle is decomposed into \(m\)-blocks indexed by
\(j=0,\dots,L_i-1\). For \(0\le s\le m\), set
\[
a_{i,j}^{(s)}
:=
\E\!\left[V(x_{B_{i,j}+s})\,\mathbf 1_{\{L_i>j\}}\right].
\]
We first show that there exist constants \(C<\infty\) and \(\eta\in(0,1)\) such
that
\begin{equation}\label{eq:cycle-V-bound}
a_{i,j}^{(s)}\le C\eta^j,
\qquad
i\ge1,\ j\ge0,\ 0\le s\le m.
\end{equation}
Indeed, since \(V\ge1\) is finite-valued and
\[
P^m(x,\cdot)\ge \beta \nu(\cdot),\qquad x\in \bar C,
\]
we have
\begin{equation}\label{eq:nuV-upper}
\beta \nu(V)\le P^mV(x)\le \lambda^mV(x)+\frac{L}{1-\lambda},
\qquad \forall x\in \bar C.
\end{equation}
It remains to justify that $\bar C$ contains at least one point at which $V$ is finite.
Suppose for contradiction that $V(x)=\infty$ for all $x\in\bar C$.
Since $\pi(V)<\infty$ by Lemma \ref{lem:pi(V)}, this would force $\pi(\bar C)=0$.
Now use accessibility of $\bar C$: for each $n\ge1$ define
\[
  A_n:=\{x\in\mathcal X:\ P^n(x,\bar C)>0\}.
\]
Accessibility implies $\bigcup_{n\ge1}A_n=\mathcal X$.
On the other hand, by invariance of $\pi$ we have, for every $n\ge1$,
\[
  \pi(\bar C)=\pi P^n(\bar C)
  =\int_{\mathcal X} P^n(x,\bar C)\pi(\d x).
\]
If $\pi(\bar C)=0$, then the nonnegativity of $P^n(x,\bar C)$ implies
$P^n(x,\bar C)=0$ for $\pi$-a.e.\ $x$, i.e.\ $\pi(A_n)=0$ for all $n\ge1$.
Consequently,
\[
  1=\pi(\mathcal X)
  =\pi\Big(\bigcup_{n\ge1}A_n\Big)
  \le \sum_{n\ge1}\pi(A_n)
  =0,
\]
a contradiction. Therefore, there exists $x^*\in\bar C$ such that $V(x^*)<\infty$.
Plugging $x=x^*$ into \eqref{eq:nuV-upper} yields
\[
  \nu(V)
  \ \le\ \frac{\lambda^m V(x^*) + \frac{L}{1-\lambda}}{\beta}
  \ <\ \infty.
\]
Thus
\[
a_{1,0}^{(0)}=\E[V(x_0)]<\infty,
\qquad
a_{i,0}^{(0)}=\nu(V)<\infty,\quad i\ge2,
\]
so \(\sup_{i\ge1}a_{i,0}^{(0)}<\infty\).

Next, using the Markov property at time \(B_{i,j}\), the drift condition, and
Lemma~\ref{lem:geom_tail},
\begin{align*}
a_{i,j+1}^{(0)}
=
\E\!\left[V(x_{B_{i,j+1}})\mathbf 1_{\{L_i>j+1\}}\right] \le
\E\left[\mathbf 1_{\{L_i>j\}}P^mV(x_{B_{i,j}})\right] &\le
\lambda^m a_{i,j}^{(0)}
+
\frac{L}{1-\lambda}\,\P(L_i>j) \\
&\le
\lambda^m a_{i,j}^{(0)}+C_0\rho^j.
\end{align*}
Since \(\lambda^m\in(0,1)\) and \(\rho\in(0,1)\), iterating this recursion yields
constants \(C_1<\infty\) and \(\eta\in(0,1)\) such that
\[
a_{i,j}^{(0)}\le C_1\eta^j,
\qquad i\ge1,\ j\ge0.
\]
Finally, for \(0\le s\le m\),
\begin{align*}
a_{i,j}^{(s)}
&=
\E\left[\mathbf 1_{\{L_i>j\}}\,P^sV(x_{B_{i,j}})\right] \le
\lambda^s a_{i,j}^{(0)}+\frac{L}{1-\lambda}\,\P(L_i>j)
\le C\eta^j,
\end{align*}
which proves \eqref{eq:cycle-V-bound}.

For \(i\ge1\), \(j\ge0\), and \(0\le s\le m-1\), define the tail of the
\(i\)-th cycle started at the within-block position \((j,s)\) by
\[
U_{i,j,s}
:=
\sum_{t=B_{i,j}+s}^{T_i-1}\xi_{t+1},
\qquad
\xi_{t+1}=H(x_t,x_{t+1}).
\]
In particular, $\widetilde M_i=U_{i,0,0}.$ Using Minkowski's inequality together with \eqref{eq:H-r-bound}, we obtain
\begin{align*}
\Bigl(\E\|U_{i,j,s}\|^r\Bigr)^{1/r}
\le
\sum_{\ell\ge j}\sum_{u=0}^{m-1}
\Bigl(
\E\bigl[\|\xi_{B_{i,\ell}+u+1}\|^r\mathbf 1_{\{L_i>\ell\}}\bigr]
\Bigr)^{1/r} \lesssim
\sum_{\ell\ge j}\sum_{u=0}^{m-1}
\Bigl(
a_{i,\ell}^{(u)}+a_{i,\ell}^{(u+1)}
\Bigr)^{1/r}.
\end{align*}
Hence, by \eqref{eq:cycle-V-bound},
\[
\sup_{i\ge1}\sup_{j\ge0}\sup_{0\le s\le m-1}\E\|U_{i,j,s}\|^r<\infty.
\]
Taking \(j=s=0\), we obtain
\begin{equation}\label{eq:block-r-moment}
\sup_{i\ge1}\E\|\widetilde M_i\|^r<\infty.
\end{equation}
We now turn to the remainder \(R_n\). By definition of \(K_n\), $T_{K_n-2}\le n<T_{K_n-1}<T_{K_n}.$ Thus \(n\) lies either at the regeneration time \(T_{K_n-1}\), in which case
\(R_n=\widetilde M_{K_n}\), or inside the \((K_n-1)\)-st cycle, in which case
there exist random indices \(J_n\ge0\) and \(0\le S_n\le m-1\) such that
\[
n=B_{K_n-1,J_n}+S_n, \qquad R_n=U_{K_n-1,J_n,S_n}+\widetilde M_{K_n}.
\]
Therefore, by Minkowski's inequality and \eqref{eq:block-r-moment},
\[
\sup_{n\ge1}\E\|R_n\|^r<\infty.
\]
Let \(\Phi\) denote the split chain of the \(m\)-skeleton, and let
\(\tau_0<\tau_1<\tau_2<\cdots\) be its regeneration times at the atom.
For \(i\ge1\), let \(\mathcal E_i\) be the entire original-time trajectory on
the \(i\)-th regeneration excursion, excluding the terminal regeneration time:
\[
\mathcal E_i:=\bigl(x_t\bigr)_{T_{i-1}\le t<T_i}.
\]
It is standard in the regeneration literature that the excursion sequence
\(\{\mathcal E_i\}_{i\ge1}\) is \(1\)-dependent. Moreover, for each \(i\ge1\),
the endpoint \(x_{T_i}\) is independent of the excursion \(\mathcal E_{i+2}\);
see, for example, \cite{Meyn12_book}. Because \(h\) is time-homogeneous, the block increment \(\widetilde M_i\)
depends only on the trajectory over the interval \([T_{i-1},T_i]\). Then, there exists a measurable map \(F\) such that
\[
\widetilde M_i = F\bigl(\mathcal E_i,x_{T_i}\bigr),\qquad i\ge1.
\]
Consequently, \(\widetilde M_i\) may depend on \(\widetilde M_{i+1}\) through
the shared boundary value \(x_{T_i}\), but if \(|j-i|\ge2\), then
\(\widetilde M_i\) and \(\widetilde M_j\) are measurable with respect to
collections of excursions separated by at least one full regeneration cycle.
Hence \(\{\widetilde M_i\}_{i\ge1}\) is \(1\)-dependent.

Finally, set \(\mathcal G_i:=\mathcal F_{T_i}\). Fix \(i\ge1\). Every stopped
value \(M_{n\wedge T_i}\) can be written as the sum of at most \(i-1\) complete
cycle increments together with one tail \(U_{j,\ell,s}\). Therefore, by
Minkowski's inequality and the uniform bounds established above,
\[
\sup_{n\ge0}\E\|M_{n\wedge T_i}\|^r<\infty.
\]
Since \(r>1\), the family \(\{M_{n\wedge T_i}:n\ge0\}\) is uniformly
integrable. Applying Lemma~\ref{lem:mds-tildeMi-proof} to the stopping times
\(T_{i-1}\le T_i\) yields
\begin{equation}\label{eq:martingaledifference}
\E[\widetilde M_i\mid \mathcal G_{i-1}]=0,
\qquad\text{and hence}\qquad
\E[\widetilde M_i]=0.
\end{equation}
This completes the proof.

\end{document}

%% file: preamble-arxiv.tex
\usepackage[utf8]{inputenc}
\usepackage{enumitem}

\usepackage[normalem]{ulem} %
\usepackage[unicode=true,pdfusetitle, bookmarks=true,bookmarksnumbered=true,bookmarksopen=false, breaklinks=true,pdfborder={0 0 0},pdfborderstyle={},backref=page,colorlinks=true,citecolor=blue]{hyperref}
\usepackage{booktabs}
\usepackage[table]{xcolor}
\definecolor{myrow}{RGB}{245,245,255}

\global\long\def\d{\mathrm{d}}%
\global\long\def\E{\mathbb{E}}%
\global\long\def\W{\mathcal{W}}%
\global\long\def\X{\mathcal{X}}
\global\long\def\R{\mathbb{R}}%
\global\long\def\P{\mathbb{P}}

\global\long\def\law{\mathcal{L}}

\usepackage{xcolor}
\usepackage{todonotes} %
\usepackage[showdeletions]{color-edits} %
\addauthor[Yixuan]{yz}{blue} %
\addauthor[Qiaomin]{qx}{red} %

%% file: main.bib
@book{Meyn12_book,
author={Meyn, Sean P. and Tweedie, Richard L.},
title={Markov Chains and Stochastic Stability},
series={Cambridge Mathematical Library},
year={2009},
edition={2nd},
publisher={Cambridge University Press},
address={Cambridge},
isbn={9780521731829},
url={https://doi.org/10.1017/CBO9780511626630}
}

@article{bonis2020stein,
  title={Stein's method for normal approximation in Wasserstein distances with application to the multivariate central limit theorem},
  author={Bonis, Thomas},
  journal={Probability Theory and Related Fields},
  volume={178},
  number={3},
  pages={827--860},
  year={2020},
  publisher={Springer}
}

@article{raivc2018multivariate,
  title={A multivariate central limit theorem for Lipschitz and smooth test functions},
  author={Rai{\v{c}}, Martin},
  journal={arXiv preprint arXiv:1812.08268},
  year={2018}
}

@article{gallouet2018regularity,
  title={Regularity of solutions of the Stein equation and rates in the multivariate central limit theorem},
  author={Gallou{\"e}t, Thomas and Mijoule, Guillaume and Swan, Yvik},
  journal={arXiv preprint arXiv:1805.01720},
  year={2018}
}

@article{baxendale2005renewal,
  title={Renewal theory and computable convergence rates for geometrically ergodic Markov chains},
  author={Baxendale, Peter H},
  year={2005}
}

@article{wu2025uncertainty,
  title={Uncertainty quantification for Markov chains with application to temporal difference learning},
  author={Wu, Weichen and Wei, Yuting and Rinaldo, Alessandro},
  year={2025},
  publisher={stat. ml}
}

@book{douc2018markov,
  title={Markov chains},
  author={Douc, Randal and Moulines, Eric and Priouret, Pierre and Soulier, Philippe},
  volume={4},
  year={2018},
  publisher={Springer}
}

@inproceedings{zhang2024prelimit,
  title={Prelimit coupling and steady-state convergence of constant-stepsize nonsmooth contractive SA},
  author={Zhang, Yixuan and Huo, Dongyan and Chen, Yudong and Xie, Qiaomin},
  booktitle={Abstracts of the 2024 ACM SIGMETRICS/IFIP PERFORMANCE Joint International Conference on Measurement and Modeling of Computer Systems},
  pages={35--36},
  year={2024}
}

@article{lemanczyk2021general,
  title={General Bernstein-like inequality for additive functionals of Markov chains},
  author={Lema{\'n}czyk, Micha{\l}},
  journal={Journal of Theoretical Probability},
  volume={34},
  number={3},
  pages={1426--1454},
  year={2021},
  publisher={Springer}
}

@book{puterman2014markov,
  title={Markov decision processes: discrete stochastic dynamic programming},
  author={Puterman, Martin L},
  year={2014},
  publisher={John Wiley \& Sons}
}

@book{sutton1998reinforcement,
  title={Reinforcement learning: An introduction},
  author={Sutton, Richard S and Barto, Andrew G and others},
  volume={1},
  number={1},
  year={1998},
  publisher={MIT press Cambridge}
}

@article{dieuleveut2020bridging,
  title={Bridging the gap between constant step size stochastic gradient descent and Markov chains},
  author={Dieuleveut, Aymeric and Durmus, Alain and Bach, Francis},
  year={2020}
}

@article{yu2021analysis,
  title={An analysis of constant step size SGD in the non-convex regime: Asymptotic normality and bias},
  author={Yu, Lu and Balasubramanian, Krishnakumar and Volgushev, Stanislav and Erdogdu, Murat A},
  journal={Advances in Neural Information Processing Systems},
  volume={34},
  pages={4234--4248},
  year={2021}
}

@article{zhang2025piecewise,
  title={A Piecewise Lyapunov Analysis of Sub-quadratic SGD: Applications to Robust and Quantile Regression},
  author={Zhang, Yixuan and Huo, Dongyan and Chen, Yudong and Xie, Qiaomin},
  journal={ACM SIGMETRICS Performance Evaluation Review},
  volume={53},
  number={1},
  pages={85--87},
  year={2025},
  publisher={ACM New York, NY, USA}
}

@article{hastings1970monte,
  title={Monte Carlo sampling methods using Markov chains and their applications},
  author={Hastings, W Keith},
  year={1970},
  publisher={Oxford University Press}
}

@article{ho2020denoising,
  title={Denoising diffusion probabilistic models},
  author={Ho, Jonathan and Jain, Ajay and Abbeel, Pieter},
  journal={Advances in neural information processing systems},
  volume={33},
  pages={6840--6851},
  year={2020}
}

@book{asmussen2003applied,
  title={Applied probability and queues},
  author={Asmussen, S{\o}ren},
  year={2003},
  publisher={Springer}
}

@book{Meyn2007,
  title={Control Techniques for Complex Networks},
  author={Meyn, Sean P.},
  year={2007},
  publisher={Cambridge University Press}
}

@article{song1993inventory,
  title={Inventory control in a fluctuating demand environment},
  author={Song, Jing-Sheng and Zipkin, Paul},
  journal={Operations Research},
  volume={41},
  number={2},
  pages={351--370},
  year={1993},
  publisher={INFORMS}
}

@article{srikant2024rates,
  title={Rates of convergence in the central limit theorem for markov chains, with an application to td learning},
  author={Srikant, Rayadurgam},
  journal={Mathematics of Operations Research},
  year={2025},
  publisher={INFORMS}
}

@article{samsonov2024gaussian,
  title={Gaussian approximation and multiplier bootstrap for polyak-ruppert averaged linear stochastic approximation with applications to td learning},
  author={Samsonov, Sergey and Moulines, Eric and Shao, Qi-Man and Zhang, Zhuo-Song and Naumov, Alexey},
  journal={Advances in Neural Information Processing Systems},
  volume={37},
  pages={12408--12460},
  year={2024}
}

@book{villani2008optimal,
  title={Optimal transport: old and new},
  author={Villani, C{\'e}dric and others},
  volume={338},
  year={2008},
  publisher={Springer}
}

@inproceedings{zhang2024constant,
  title={Constant Stepsize Q-learning: Distributional Convergence, Bias and Extrapolation},
  author={Zhang, Yixuan and Xie, Qiaomin},
  booktitle={Reinforcement Learning Conference},
   year={2024}
}

@article{huo2023bias,
  title={Bias and Extrapolation in Markovian Linear Stochastic Approximation with Constant Stepsizes},
  author={Huo, Dongyan and Chen, Yudong and Xie, Qiaomin},
  journal={ACM SIGMETRICS Performance Evaluation Review},
  volume={51},
  number={1},
  pages={81--82},
  year={2023},
  publisher={ACM New York, NY, USA}
}

@article{huo2024collusion,
  title={The collusion of memory and nonlinearity in stochastic approximation with constant stepsize},
  author={Huo, Dongyan Lucy and Zhang, Yixuan and Chen, Yudong and Xie, Qiaomin},
  journal={Advances in Neural Information Processing Systems},
  volume={37},
  pages={21699--21762},
  year={2024}
}

@article{nummelin1978splitting,
  title={A splitting technique for Harris recurrent Markov chains},
  author={Nummelin, Esa},
  journal={Zeitschrift f{\"u}r Wahrscheinlichkeitstheorie und verwandte Gebiete},
  volume={43},
  number={4},
  pages={309--318},
  year={1978},
  publisher={Springer}
}

@incollection{baldi1989normal,
  title={A normal approximation for the number of local maxima of a random function on a graph},
  author={Baldi, Pierre and Rinott, Yosef and Stein, Charles},
  booktitle={Probability, statistics, and mathematics},
  pages={59--81},
  year={1989},
  publisher={Elsevier}
}

@article{goldstein1996multivariate,
  title={Multivariate normal approximations by Stein's method and size bias couplings},
  author={Goldstein, Larry and Rinott, Yosef},
  journal={Journal of Applied Probability},
  volume={33},
  number={1},
  pages={1--17},
  year={1996},
  publisher={Cambridge University Press}
}

@book{chen2010normal,
  title={Normal approximation by Stein’s method},
  author={Chen, Louis HY and Goldstein, Larry and Shao, Qi-Man},
  year={2010},
  publisher={Springer Science \& Business Media}
}

@article{bobkov2024fourier,
  title={Fourier analytic bounds for Zolotarev distances, and applications to empirical measures},
  author={Bobkov, Sergey G and Ledoux, Michel},
  journal={Preprint},
  year={2024}
}

@article{bobkov2024quantified,
  title={Quantified Cram$\backslash$'er-Wold Continuity Theorem for the Kantorovich Transport Distance},
  author={Bobkov, Sergey G and G{\"o}tze, Friedrich},
  journal={arXiv preprint arXiv:2412.10276},
  year={2024}
}

@article{liu2023wasserstein,
  title={Wasserstein-p bounds in the central limit theorem under local dependence},
  author={Liu, Tianle and Austern, Morgane},
  journal={Electronic Journal of Probability},
  volume={28},
  pages={1--47},
  year={2023},
  publisher={The Institute of Mathematical Statistics and the Bernoulli Society}
}

@article{roberts1999bounds,
  title={Bounds on regeneration times and convergence rates for Markov chains},
  author={Roberts, Gareth O and Tweedie, Richard L},
  journal={Stochastic Processes and their applications},
  volume={80},
  number={2},
  pages={211--229},
  year={1999},
  publisher={Elsevier}
}

@article{rosenthal1995minorization,
  title={Minorization conditions and convergence rates for Markov chain Monte Carlo},
  author={Rosenthal, Jeffrey S},
  journal={Journal of the American Statistical Association},
  volume={90},
  number={430},
  pages={558--566},
  year={1995},
  publisher={Taylor \& Francis}
}

@inproceedings{douc2008bounds,
  title={Bounds on regeneration times and limit theorems for subgeometric Markov chains},
  author={Douc, Randal and Guillin, Arnaud and Moulines, Eric},
  booktitle={Annales de l'IHP Probabilit{\'e}s et statistiques},
  volume={44},
  number={2},
  pages={239--257},
  year={2008}
}

@article{rinott1994normal,
  title={On normal approximation rates for certain sums of dependent random variables},
  author={Rinott, Yosef},
  journal={Journal of Computational and Applied Mathematics},
  volume={55},
  number={2},
  pages={135--143},
  year={1994},
  publisher={Elsevier}
}

@article{romano2000more,
  title={A more general central limit theorem for m-dependent random variables with unbounded m},
  author={Romano, Joseph P and Wolf, Michael},
  journal={Statistics \& probability letters},
  volume={47},
  number={2},
  pages={115--124},
  year={2000},
  publisher={Elsevier}
}

@inproceedings{diananda1955central,
  title={The central limit theorem for m-dependent variables},
  author={Diananda, PH},
  booktitle={Mathematical Proceedings of the Cambridge Philosophical Society},
  volume={51},
  number={1},
  pages={92--95},
  year={1955},
  organization={Cambridge University Press}
}

@article{fang2023p,
  title={From p-Wasserstein bounds to moderate deviations},
  author={Fang, Xiao and Koike, Yuta},
  journal={Electronic Journal of Probability},
  volume={28},
  pages={1--52},
  year={2023},
  publisher={The Institute of Mathematical Statistics and the Bernoulli Society}
}

@article{powers1970free,
  title={Free states of the canonical anticommutation relations},
  author={Powers, Robert T and St{\o}rmer, Erling},
  journal={Communications in Mathematical Physics},
  volume={16},
  number={1},
  pages={1--33},
  year={1970},
  publisher={Springer}
}

@article{breiman2001random,
  title={Random forests},
  author={Breiman, Leo},
  journal={Machine learning},
  volume={45},
  number={1},
  pages={5--32},
  year={2001},
  publisher={Springer}
}

@article{mentch2016quantifying,
  title={Quantifying uncertainty in random forests via confidence intervals and hypothesis tests},
  author={Mentch, Lucas and Hooker, Giles},
  journal={Journal of Machine Learning Research},
  volume={17},
  number={26},
  pages={1--41},
  year={2016}
}

@article{breiman1996bagging,
  title={Bagging predictors},
  author={Breiman, Leo},
  journal={Machine learning},
  volume={24},
  number={2},
  pages={123--140},
  year={1996},
  publisher={Springer}
}

@article{bickel2008covariance,
  title={Covariance regularization by thresholding},
  author={Bickel, Peter J and Levina, Elizaveta},
  year={2008}
}

@article{bickel2008regularized,
  title={Regularized estimation of large covariance matrices},
  author={Bickel, Peter J and Levina, Elizaveta},
  year={2008}
}

@article{dempster1972covariance,
  title={Covariance selection},
  author={Dempster, Arthur P},
  journal={Biometrics},
  pages={157--175},
  year={1972},
  publisher={JSTOR}
}

@article{goldstein1979maximal,
  title={Maximal coupling},
  author={Goldstein, Sheldon},
  journal={Zeitschrift f{\"u}r Wahrscheinlichkeitstheorie und verwandte Gebiete},
  volume={46},
  number={2},
  pages={193--204},
  year={1979},
  publisher={Springer}
}

@inproceedings{rio2009upper,
  title={Upper bounds for minimal distances in the central limit theorem},
  author={Rio, Emmanuel},
  booktitle={Annales de l'IHP Probabilit{\'e}s et statistiques},
  volume={45},
  number={3},
  pages={802--817},
  year={2009}
}

@article{pinelis1994optimum,
  title={Optimum bounds for the distributions of martingales in Banach spaces},
  author={Pinelis, Iosif},
  journal={The Annals of Probability},
  pages={1679--1706},
  year={1994},
  publisher={JSTOR}
}

@article{panaretos2019statistical,
  title={Statistical aspects of Wasserstein distances},
  author={Panaretos, Victor M and Zemel, Yoav},
  journal={Annual review of statistics and its application},
  volume={6},
  number={1},
  pages={405--431},
  year={2019},
  publisher={Annual Reviews}
}

@incollection{kuhn2019wasserstein,
  title={Wasserstein distributionally robust optimization: Theory and applications in machine learning},
  author={Kuhn, Daniel and Esfahani, Peyman Mohajerin and Nguyen, Viet Anh and Shafieezadeh-Abadeh, Soroosh},
  booktitle={Operations research \& management science in the age of analytics},
  pages={130--166},
  year={2019},
  publisher={Informs}
}

@book{petrov2012sums,
  title={Sums of independent random variables},
  author={Petrov, Valentin V},
  volume={82},
  year={2012},
  publisher={Springer Science \& Business Media}
}

@article{bobkov2013entropic,
  title={Entropic approach to E. Rio’s central limit theorem for W2 transport distance},
  author={Bobkov, Sergey G},
  journal={Statistics \& Probability Letters},
  volume={83},
  number={7},
  pages={1644--1648},
  year={2013},
  publisher={Elsevier}
}

@article{kipf2016semi,
  title={Semi-supervised classification with graph convolutional networks},
  author={Kipf, TN},
  journal={arXiv preprint arXiv:1609.02907},
  year={2016}
}

@article{hamilton2017inductive,
  title={Inductive representation learning on large graphs},
  author={Hamilton, Will and Ying, Zhitao and Leskovec, Jure},
  journal={Advances in neural information processing systems},
  volume={30},
  year={2017}
}

@article{law1984confidence,
  title={Confidence intervals for steady-state simulations: I. A survey of fixed sample size procedures},
  author={Law, Averill M and Kelton, W David},
  journal={Operations Research},
  volume={32},
  number={6},
  pages={1221--1239},
  year={1984},
  publisher={INFORMS}
}

@incollection{hoeffding1992class,
  title={A class of statistics with asymptotically normal distribution},
  author={Hoeffding, Wassily},
  booktitle={Breakthroughs in statistics: Foundations and basic theory},
  pages={308--334},
  year={1992},
  publisher={Springer}
}

@article{callaert1978berry,
  title={The Berry-Esseen theorem for U-statistics},
  author={Callaert, Herman and Janssen, Paul},
  journal={The Annals of Statistics},
  pages={417--421},
  year={1978},
  publisher={JSTOR}
}

@article{gotze1987approximations,
  title={Approximations for multivariate U-statistics},
  author={G{\"o}tze, Friedrich},
  journal={Journal of multivariate analysis},
  volume={22},
  number={2},
  pages={212--229},
  year={1987},
  publisher={Elsevier}
}

@article{chen2018gaussian,
  title={Gaussian and bootstrap approximations for high-dimensional U-statistics and their applications},
  author={Chen, Xiaohui},
  year={2018}
}

@article{fang2019wasserstein,
  title={Wasserstein-2 bounds in normal approximation under local dependence},
  author={Fang, Xiao},
  year={2019}
}

@article{matthes2009family,
  title={A family of nonlinear fourth order equations of gradient flow type},
  author={Matthes, Daniel and McCann, Robert J and Savar{\'e}, Giuseppe},
  journal={Communications in Partial Differential Equations},
  volume={34},
  number={11},
  pages={1352--1397},
  year={2009},
  publisher={Taylor \& Francis}
}

@book{malinovskii2021level,
  title={Level-crossing problems and inverse Gaussian distributions: closed-form results and approximations},
  author={Malinovskii, Vsevolod K},
  year={2021},
  publisher={Chapman and Hall/CRC}
}

@article{cox1962renewal,
  title={Renewal theory},
  author={Cox, David Roxbee},
  journal={Chapman and Hall},
  year={1962}
}

@book{cox2017theory,
  title={The theory of stochastic processes},
  author={Cox, David Roxbee},
  year={2017},
  publisher={Routledge}
}

@article{bentkus2003dependence,
  title={On the dependence of the Berry--Esseen bound on dimension},
  author={Bentkus, Vidmantas},
  journal={Journal of Statistical Planning and Inference},
  volume={113},
  number={2},
  pages={385--402},
  year={2003},
  publisher={Elsevier}
}

@article{chen2019randomized,
  title={Randomized incomplete U-statistics in high dimensions},
  author={Chen, Xiaohui and Kato, Kengo},
  year={2019}
}

@article{muller1997integral,
  title={Integral probability metrics and their generating classes of functions},
  author={M{\"u}ller, Alfred},
  journal={Advances in applied probability},
  volume={29},
  number={2},
  pages={429--443},
  year={1997},
  publisher={Cambridge University Press}
}

@article{adamczak2008tail,
  title={A tail inequality for suprema of unbounded empirical processes with applications to Markov chains},
  author={Adamczak, Radoslaw},
  year={2008}
}

@article{adamczak2015exponential,
  title={Exponential concentration inequalities for additive functionals of Markov chains},
  author={Adamczak, Rados{\l}aw and Bednorz, Witold},
  journal={ESAIM: Probability and Statistics},
  volume={19},
  pages={440--481},
  year={2015}
}

@article{bertail2018new,
  title={New Bernstein and Hoeffding type inequalities for regenerative Markov chains},
  author={Bertail, Patrice and Cio{\l}ek, Gabriela},
  year={2018}
}

@article{bolthausen1982berry,
  title={The Berry-Esseen theorem for strongly mixing Harris recurrent Markov chains},
  author={Bolthausen, Erwin},
  journal={Zeitschrift f{\"u}r Wahrscheinlichkeitstheorie und verwandte Gebiete},
  volume={60},
  number={3},
  pages={283--289},
  year={1982},
  publisher={Springer}
}

@article{bolthausen1980berry,
  title={The Berry-Esseen theorem for functionals of discrete Markov chains},
  author={Bolthausen, Erwin},
  journal={Zeitschrift f{\"u}r Wahrscheinlichkeitstheorie und verwandte Gebiete},
  volume={54},
  number={1},
  pages={59--73},
  year={1980},
  publisher={Springer}
}

@article{mykland1995regeneration,
  title={Regeneration in Markov chain samplers},
  author={Mykland, Per and Tierney, Luke and Yu, Bin},
  journal={Journal of the American Statistical Association},
  volume={90},
  number={429},
  pages={233--241},
  year={1995},
  publisher={Taylor \& Francis}
}

@article{jones2001honest,
  title={Honest exploration of intractable probability distributions via Markov chain Monte Carlo},
  author={Jones, Galin L and Hobert, James P},
  journal={Statistical Science},
  pages={312--334},
  year={2001},
  publisher={JSTOR}
}

@book{nourdin2012normal,
  title={Normal approximations with Malliavin calculus: from Stein's method to universality},
  author={Nourdin, Ivan and Peccati, Giovanni},
  volume={192},
  year={2012},
  publisher={Cambridge University Press}
}

@article{barbour1986asymptotic,
  title={Asymptotic expansions based on smooth functions in the central limit theorem},
  author={Barbour, Andrew D},
  journal={Probability Theory and Related Fields},
  volume={72},
  number={2},
  pages={289--303},
  year={1986},
  publisher={Springer}
}

@article{ibragimov1966accuracy,
  title={On the accuracy of Gaussian approximation to the distribution functions of sums of independent variables},
  author={Ibragimov, IA},
  journal={Theory of Probability \& Its Applications},
  volume={11},
  number={4},
  pages={559--579},
  year={1966},
  publisher={SIAM}
}

@article{chernozhukov2013gaussian,
  title={Gaussian approximations and multiplier bootstrap for maxima of sums of high-dimensional random vectors},
  author={Chernozhukov, Victor and Chetverikov, Denis and Kato, Kengo},
  journal={The Annals of Statistics},
  pages={2786--2819},
  year={2013},
  publisher={JSTOR}
}

@article{raivc2004multivariate,
  title={A multivariate CLT for decomposable random vectors with finite second moments},
  author={Rai{\v{c}}, Martin},
  journal={Journal of Theoretical Probability},
  volume={17},
  number={3},
  pages={573--603},
  year={2004},
  publisher={Springer}
}

@article{barbour1989central,
  title={A central limit theorem for decomposable random variables with applications to random graphs},
  author={Barbour, Andrew D and Karo{\'n}ski, Michal and Ruci{\'n}ski, Andrzej},
  journal={Journal of Combinatorial Theory, Series B},
  volume={47},
  number={2},
  pages={125--145},
  year={1989},
  publisher={Elsevier}
}

@article{bonis2024improved,
  title={Improved rates of convergence for the multivariate Central Limit Theorem in Wasserstein distance},
  author={Bonis, Thomas},
  journal={Electronic Journal of Probability},
  volume={29},
  pages={1--18},
  year={2024},
  publisher={The Institute of Mathematical Statistics and the Bernoulli Society}
}

@article{shi2021timing,
  title={Timing it right: Balancing inpatient congestion vs. readmission risk at discharge},
  author={Shi, Pengyi and Helm, Jonathan E and Deglise-Hawkinson, Jivan and Pan, Julian},
  journal={Operations Research},
  volume={69},
  number={6},
  pages={1842--1865},
  year={2021},
  publisher={INFORMS}
}

@article{varma2023dynamic,
  title={Dynamic pricing and matching for two-sided queues},
  author={Varma, Sushil Mahavir and Bumpensanti, Pornpawee and Maguluri, Siva Theja and Wang, He},
  journal={Operations Research},
  volume={71},
  number={1},
  pages={83--100},
  year={2023},
  publisher={INFORMS}
}

@article{ben2025data,
  title={Data-driven policies for the online ride-hailing problem with fairness},
  author={Ben-Gal, Shachaf and Tzur, Michal},
  journal={Transportation Science},
  volume={59},
  number={3},
  pages={647--669},
  year={2025},
  publisher={INFORMS}
}

@article{chan2026optimizing,
  title={Optimizing Interhospital Patient Transfer Decisions: A Queueing Network Approach},
  author={Chan, Timothy CY and Park, Jangwon and Pogacar, Frances and Sarhangian, Vahid and Hellsten, Erik and Razak, Fahad and Verma, Amol},
  journal={Manufacturing \& Service Operations Management},
  year={2026},
  publisher={INFORMS}
}
